\documentclass[11pt]{amsart}
\usepackage{graphicx} %
\usepackage{amsmath,amssymb}
\usepackage[top=1in, bottom=1in, left=1in, right=1in]{geometry} 
\usepackage{comment} 
\usepackage{microtype}
\usepackage{xcolor}
\usepackage{url}
\usepackage{mathtools}
\usepackage{tikz-cd}
\usepackage{enumerate}
\usepackage{xspace}
\usepackage{hyperref}
\usepackage{soul}
\usepackage{colonequals}
\usepackage{mathscinet} %
\usepackage{mathpazo}
\usepackage[cal=boondox]{mathalfa}

\usepackage{cellspace} %
\DeclareMathAlphabet{\mathbx}{U}{BOONDOX-ds}{m}{n}
\SetMathAlphabet{\mathbx}{bold}{U}{BOONDOX-ds}{b}{n}
\DeclareMathAlphabet{\mathbxx} {U}{BOONDOX-ds}{b}{n}

\hypersetup{
   colorlinks,
   citecolor=blue,
   filecolor=blue,
   linkcolor=blue,
   urlcolor=blue
}

\theoremstyle{plain}
\newtheorem{theorem}{Theorem}[section]
\newtheorem{lemma}[theorem]{Lemma}
\newtheorem{proposition}[theorem]{Proposition}
\newtheorem{propositiondefn}[theorem]{Proposition/Definition}
\newtheorem{corollary}[theorem]{Corollary}

\theoremstyle{definition}
\newtheorem{example}[theorem]{Example}
\newtheorem{definition}[theorem]{Definition}

\newtheorem{remark}[theorem]{Remark}

\newtheorem{notation}[theorem]{Notation}

\newcommand{\Z}{\mathbf{Z}}

\newcommand{\spa}{\mathrm{span}_k}
\newcommand{\ord}{\mathrm{ord}}

\newcommand{\Res}{\mathrm{Res}}
\newcommand{\PP}{\mathbf{P}}
\newcommand{\Om}{\Omega}

\newcommand{\Jac}{\text{Jac}}
\newcommand{\cO}{\mathcal{O}}
\newcommand{\cD}{\mathcal{D}}

\newcommand{\Gal}{\mathrm{Gal}}

\DeclareMathOperator{\HH}{\mathrm{H}}
\newcommand{\HdR}{\HH_{\textrm{dR}}}

\newcommand{\codim}{\mathrm{codim}}

\newcommand{\nil}{\mathrm{nil}}
\newcommand{\bij}{\mathrm{bij}}

\newcommand{\ceiling}[1]{\lceil #1 \rceil}
\newcommand{\floor}[1]{\left \lfloor #1 \right \rfloor}

\title{Ekedahl-Oort types of $\Z/2\Z$-covers in characteristic $2$}

\author{Jeremy Booher}
\address{University of Florida, Gainesville, Florida, USA}
\email{jeremybooher@ufl.edu}

\author{Steven R. Groen} 
\address{Korteweg-de Vries Institute for Mathematics, University of Amsterdam, Amsterdam, Netherlands}
\email{s.r.groen@uva.nl}

\author{Joe Kramer-Miller}
\address{Lehigh University, Bethlehem, Pennsylvania, USA}
\email{jjk221@lehigh.edu}

\begin{document}

\begin{abstract}
In this article we study the Ekedahl-Oort types of $\Z/2\Z$-Galois covers
$\pi:Y \to X$ in characteristic two. When the base curve $X$ is ordinary, we show
that the Ekedahl-Oort type of $Y$ is completely determined by the
genus of $X$ and the ramification of $\pi$. For a general base curve $X$,
we prove bounds on the Ekedahl-Oort depending on the Ekedahl-Oort type
of $X$ and the ramification of $\pi$. Along the way, we develop
a theory of \emph{enhanced differentials of the second kind}. This theory
allows us to study algebraic de Rham cohomology in any characteristic
by working directly with differentials, in contrast to the standard \v{C}ech resolution. 


2020 MSC primary: 11G20, 14H10, 14H30, 14H40, 14F40. Secondary: 14G17, 14L15.
\end{abstract}

\maketitle

\section{Introduction}

Let $p$ be a prime and $k$ be a perfect field of characteristic $p$.  If $X, Y$ are nice curves over $k$ (smooth, projective, and geometrically connected) and $\pi: Y \to X$ is a branched $\Z/p\Z$-Galois cover of curves, it is natural to study how ``natural properties'' of $Y$ depend on $X$ and the ramification of $\pi$.  Most famously, the Riemann-Hurwitz formula relates the genus of $Y$ to the genus of $X$ and the ramification of $\pi$. 
Another well-known example is the Deuring-Shafarevich theorem, which gives a formula for the $p$-rank of $Y$ in terms of the $p$-rank of $X$ and the ramification \cite{SubraoDS}.  
More recent work has focused on finer invariants related to
the $p$-power torsion of the Jacobians $\Jac(X)$ and $\Jac(Y)$, which reveals a
more nuanced picture. 
For instance, there is no analog of the Deuring-Shafarevich theorem for $a$-numbers. Instead, there are bounds on the $a$-number of $Y$ in terms
of the $a$-number of $X$ and the ramification of $\pi$ (see \cite{PriesConstant,boohercais} and \cite{groen} for similar results with
higher $a$-numbers). Similarly, there is no Deuring-Shafarevich theorem
for `higher slopes' of the Newton polygon. There is, however,
a `Newton-over-Hodge' type phenomenon, which gives a lower bound
on the Newton polygon of $Y$ in terms of the Newton polygon of $X$ and
the ramification (see \cite{Kramer-Miller} and \cite{Kramer-Miller-Upton}). 
All of this work can be subsumed into the following goal: 
describe the cohomology of $Y$ in terms of the cohomology of $X$ and the ramification of $\pi$. In this article we work with algebraic de Rham cohomology in characteristic two,
endowed with the structure of a mod-$p$ Dieudonn\'e module.
When the base curve $X$ is ordinary, we completely determine the
Dieudonn\'e module of $Y$ in terms of the Dieudonn\'e module of $X$
and the ramification. For a more general base curve we provide
bounds on the Dieudonn\'e module of $Y$ in terms of this same information.

Recall that for a nice curve $C$ over $k$, the algebraic de Rham cohomology $\HdR^1(C)$ is naturally equipped with the semilinear operators Frobenius and Verschiebung, $F$ and $V$, subject to the relation $FV=VF=0$.  In particular, $\HdR^1(C)$ is
a Dieudonn\'{e} module, i.e., a module over the mod-$p$ Dieudonn\'{e} ring $D_k=k[F,V]/(FV)$,  where $Fw=w^pF$ and $Vw^p=wV$ for $w \in k$. 
There is also a symplectic pairing $\langle \cdot , \cdot \rangle$ on $\HdR^1(C)$ for which $F$ and $V$ are `skew' adjoint, i.e., they satisfy $\langle Vx,y\rangle^p=\langle x,Fy\rangle$.
Furthermore, there is a short exact sequence
\[
0 \to H^0(C,\Omega^1_C) \to \HdR^1(C) \to H^1(C,\cO_C) \to 0
\]
with $H^0(C,\Omega^1_C) = \ker F = \operatorname{Im} V$.  So knowledge of $\HdR^1(C)$ as a $D_k$-module %
gives, among other things, an understanding of the genus of $C$ (half the dimension), the $p$-rank (the stable-rank of $V$ on $H^0(C,\Omega^1_C)$), and the $a$-number (the dimension of the kernel of $V$ on $H^0(C,\Omega^1_C)$, i.e. $\dim_k (\ker F \cap \ker V)$).  Furthermore, a result of Oda \cite{Oda} shows that $\HdR^1(C)$ is naturally isomorphic as a Dieudonn\'{e} module to the Dieudonn\'{e} module of $\Jac(C)[p]$, the $p$-torsion in the Jacobian of $C$.  
In particular, knowledge of the Dieudonn\'{e} module $\HdR^1(C)$ is equivalent to  understanding the \emph{Ekedahl-Oort type} of $C$, which is the isomorphism class of $\Jac(C)[p]$ as a finite flat group scheme.  See \cite{pries_guide} and Section~\ref{ss:final types} for background on Ekedahl-Oort types and how to combinatorially express them as \emph{final types}.
Thus an equivalent way to state our question is as follows: given a branched $\Z/p\Z$-Galois cover of curves $\pi : Y \to X$, describe the Ekedahl-Oort type of $Y$ in terms of the Ekedahl-Oort type of $X$ and the ramification of $\pi$.  

The only previous work on the question at this level of generality is Cais and Ulmer's work on unramified covers \cite{caisulmer} and Elkin and Pries's work for covers of $\PP^1$ in characteristic two \cite{ElkinPries}.  In the present paper, we focus on one of the few general situations where we expect the Ekedahl-Oort type of $X$ and the ramification of $\pi$ to determine the Ekedahl-Oort type of $Y$.  In particular, in characteristic $p=2$ when the base curve $X$ is ordinary we show the ramification of $\pi : Y \to X$ determines the Ekedahl-Oort type of $Y$.  (Recall $X$ being ordinary means that the $p$-rank of $X$ equals the genus of $X$, which fully determines the Ekedahl-Oort type.)  To state the result more precisely, we must set up some notation.

Fix $p=2$ and consider a $\Z/2 \Z$-cover $\pi : Y \to X$ which is ramified over $m$ points $\{P_i\}_{i=1,\ldots,m}$ of $X$.  Let $d_i$ be the \emph{ramification break} at $P_i$, the unique break in the upper ramification filtration above $P_i$.  The final type is a combinatorial way to describe a polarized mod-$p$ Dieudonn\'{e} module, which we review in Section~\ref{ss:final types}.

\begin{definition}
Let $M_{\textrm{ord}}$ be the polarized mod-$p$ Dieudonn\'{e} module with final type $[1]$ (i.e. the Dieudonn\'{e} module of an ordinary elliptic curve).
For a positive odd integer $d$, let $M_d$ be the polarized mod-$p$ Dieudonn\'{e} module with final type $\left[0,1,1,2,2,\ldots, \floor{\frac{d-1}{4}}\right]$.  (The last entry occurs once or twice depending on $d$ modulo $4$.) 
\end{definition}

\begin{theorem} \label{thmA}
Let $\pi : Y \to X$ be a $\Z/2 \Z$-cover of smooth, proper, geometrically connected curves over a perfect field $k$ of characteristic two, and suppose $X$ is ordinary.  Then as Dieudonn\'{e} modules
\[
D(\Jac(Y)[2]) \simeq\HdR^1(Y) \simeq M_{\ord}^ {2 g_X -1 +m } \oplus \bigoplus_{i=1}^m M_{d_i}.
\]
\end{theorem}

Here $g_X$ is the genus of $X$, equal to the $p$-rank of $X$; note that $2 g_X -1 + m$ is the $p$-rank of $Y$ by the Deuring-Shafarevich formula \cite{SubraoDS}.

\begin{remark}
Theorem~\ref{thmA} is a simultaneous generalization of a result of Elkin and Pries, which explicitly treats the case that $X = \PP^1$ \cite{ElkinPries}, and of a result of Voloch which computes the $a$-number of an Artin-Schreier cover of an ordinary curve in characteristic two \cite{voloch88}.
\end{remark}

\begin{remark}
The Dieudonn\'{e} modules in Theorem~\ref{thmA} can equivalently be described by giving the $F$ and $V$ actions: see Lemma~\ref{lemma: decomposition k[F,V]}, Section~\ref{ss:fandv}, and in particular the proof of Theorem~\ref{thm: ordinary EO type}.
\end{remark}

By fixing a base curve in characteristic two and varying the cover subject to a specific set of ramification invariants, we can construct positive dimensional families of genus $g$ curves with constant Ekedahl-Oort type.  For example:

\begin{corollary} \label{cor:families}
For any positive integer $g$, there is a $(g-1)$-dimensional family of smooth curves of genus $g$ with constant Ekedahl-Oort type: in particular, the final type is $[0,1,1,2,2,\ldots, \floor{g/2}]$. %

For any positive integer $n$, let $d = 2^n+1$.  There is a $(d+1)/2$-dimensional family of smooth curves of genus $(d+3)/2$ with constant Ekedahl-Oort type: in particular, the final type is $[0,1,2,3,3,4,4, \ldots, (d-1)/4+1,(d-1)/4+1,(d-1)/4+2]$. %
\end{corollary}

\begin{remark}
    These families are ``unlikely intersections'' in the sense that, when $g$ is sufficiently large, the codimension of the Ekedahl-Oort strata in the moduli space of principally polarized Abelian varieties of dimension $g$ exceeds $3g-3$, the dimension of the moduli space of curves of genus $g$ (see for example Remark~\ref{remark:unlikely}).
\end{remark}

When the base curve $X$ is not ordinary, the Ekedahl-Oort type of the cover $Y$ is no longer
determined entirely by the Ekedahl-Oort type of $X$ and the ramification breaks. Instead, we are able to constrain the Ekedahl-Oort type of the cover $Y$ based on the ramification breaks. In Section~\ref{sec: non ordinary base} we consider the simplest non-ordinary case, which is a supersingular elliptic curve.
When the cover is ramified at a single point with break $d$, we determine all possible $k[V]$-structures on $H^0(Y,\Omega_Y^1)$ (see Theorem \ref{t: k[V] structure on ss ell curve}) and determine the codimensions of the strata in the moduli space of such covers (see Theorem \ref{t: codimension of strata}).  The difficulty in extending this analysis to the
full Ekedahl-Oort type lies in understanding the pairing $\langle\cdot ,\cdot  \rangle$. In particular, in Example~\ref{ex:theoryd15} we exhibit two covers whose
global differentials have the same $k[V]$-structure, but whose Ekedahl-Oort
types differ.

More generally, let $X$ be any curve and let $[v_1,\dots,v_{g_Y}]$ be the final type of $Y$. In Theorem \ref{theorem: final type bounds} we estimate $v_l$ with an error that is essentially logarithmic in terms of the ramification breaks. To avoid introducing unnecessary notation, we put off
stating Theorem \ref{theorem: final type bounds} in full generality until Section~\ref{sec:bounds}. Instead, we state 
a special case of this theorem for the case where $Y\to X$ is branched at one point. 

\begin{theorem} \label{thm: log intro}
Let $Y \to X$ be a $\Z/2 \Z$-cover of smooth, proper, geometrically connected curves over a perfect field $k$ of characteristic two,
branched at one point with ramification invariant $d$. Let $f_Y=2f_X$ be the $p$-rank of $Y$ and $l_X=g_X-f_X$ be the local rank of $X$. Letting $[\nu_1, \ldots, \nu_{g_Y} ]$ denote the final type of $Y$:
\begin{enumerate}
    \item For $1\leq l \leq f_Y$, we have $\nu_l=l$.
    \item For $f_Y< l < f_Y+2l_X,$ we have $f_Y \leq \nu_l \leq l-1$.
    \item For $f_Y+2l_X \leq l \leq g_Y$, we have 
    \begin{equation} 
        \left | \nu_l - \left(f_Y + \left \lfloor \frac{l-f_Y}{2} \right \rfloor \right)   \right| \leq \begin{cases}   \frac{3\lceil \log_2(d-1) \rceil}{2}  l_X &\hbox{if $d \leq 4g_X-5$}  \\ \ceiling{\log_2(d-1)} l_X &\hbox{if $d > 4g_X-5$.} \end{cases} 
    \end{equation}
\end{enumerate}
\end{theorem}

In particular, we have an estimate for the final type of $Y$
whose error is logarithmic in the ramification invariant.

\subsection{Structure of the paper}

The main technical tool we introduce is a generalization of differentials of the second kind which works in positive characteristic.  
We develop a theory of \emph{enhanced differentials of the second kind} for smooth proper connected curves over a field in Section~\ref{sec:enhanced differentials}, based on a folklore pole-order resolution for the de Rham complex.  Concretely, an enhanced differential of the second kind consists
of a classical differential of the second kind $\omega$ (i.e. a locally exact differential) together with a choice of antiderivative for the meromorphic part of $\omega$ at each of its poles. In characteristic
zero there is only one such antiderivative, and we therefore recover the classical theory of differentials of the second kind. 
In contrast, when $k$ has characteristic $p$ there are many possible 
local antiderivatives, since $\frac{d}{dx}f(x)^p=0$.
This theory gives a concrete way to represent de Rham cohomology classes: we prove that $\HdR^1(X)$ is isomorphic to the space of enhanced differentials
of the second kind modulo those that are globally exact. Note that classical differentials of the second kind have been studied in positive
characteristic by Rosenlicht \cite{rosenlicht}, but they fail to compute
de Rham cohomology. 
The standard techniques for studying the Cartier operator on regular differentials have natural generalizations to enhanced differentials of the second kind, which inspires our analysis.  (While a similar analysis could surely be done using the representation of de Rham cohomology classes coming from the \v{C}ech resolution, the generalization of these techniques is far less natural from that perspective.)

In Section~\ref{section: omegaij}, we apply the theory developed in Section~\ref{sec:enhanced differentials}  to study $\Z/2\Z$-covers $\pi : Y \to X$ in characteristic two.  We decompose 
$\HdR^1(Y) = U \oplus Z \oplus L$ where $V$ is bijective on $U$ (and $F$ is zero), where $F$ is bijective on $Z$ (and $V$ is zero), and where $F$ and $V$ are both nilpotent on $L$. The latter space is the most interesting, and we construct a convenient set of enhanced differentials of the second kind $\{\widetilde{\omega}_{i,j}\}$ which almost span $L$ and for which the action of $V$ is particularly simple (see Proposition~\ref{prop:nilpotent wij} and Definition~\ref{defn:wijtilde}). In particular, these enhanced differentials are built using pullbacks of meromorphic differentials on the base curve $X$: this is the same spirit as Voloch's argument computing the $a$-number of $Y$ using a space of differentials with bounded poles on the ordinary base curve \cite{voloch88}.  For a general base curve, we know a lot about the $V$-action on $L$ and the form of the duality pairing on $\HdR^1(Y)$ via studying $\{\widetilde{\omega}_{i,j}\}$ but not quite enough to determine the Ekedahl-Oort type of $Y$. We expect the ideas in Section~\ref{sec:enhanced differentials}
and Section~\ref{section: omegaij} will be useful for studying
$\HdR^1(X)$ for curves in any characteristic.

In Section~\ref{sec:ordinary base} we study the case when the base curve $X$ is ordinary, in which case $\{\widetilde{\omega}_{i,j}\}$ forms a basis for $L$ and we are able to completely describe the action of $V$.  This is enough information to fully determine the Ekedahl--Oort type and establish Theorem~\ref{thmA}. We obtain the first family described in Corollary~\ref{cor:families} by looking at covers of $\PP^1$ ramified at one point with ramification break $d = 2g+1$, which is a case covered by Elkin and Pries \cite{ElkinPries}.  %

In Section~\ref{sec: non ordinary base} we study the case
when $X$ is a supersingular elliptic curve. This is the
simplest non-ordinary case. We first give explicit examples
that demonstrate that the Ekedahl--Oort type of $Y$ is
not determined by that of $X$ and the ramification of the cover. 
This is done by exploring the difference between the span of $\{\widetilde{\omega}_{i,j}\}$ and the space $L$. 
Next, we consider the $k[V]$-structure of $Y$ for covers
branched over one point. We determine all possible
$k[V]$-structures that occur. Furthermore, we determine
the codimensions of the $k[V]$-module strata in the moduli
space of curves admitting a $\Z/2\Z$-cover to $X$ branched
at a single point with fixed ramification break $d$. Finally,
we prove that the Ekedahl-Oort type is determined by the ramification when $d=2^n+1$, which gives the second case of Corollary~\ref{cor:families}.

In Section~\ref{sec:bounds}, we focus on bounding the Ekedahl-Oort type of $Y$ when $X$ is not ordinary.  The final type is determined by the dimensions of the spaces $ w(\HdR^1(Y))$ where
$w = V^{n_t} \bot V^{n_{t-1}} \bot \ldots \bot V^{n_1}$. Here $\bot$ denotes taking symplectic
complement under the pairing $\langle \cdot , \cdot \rangle$. The uncertainty about the $V$-action and its interaction with the duality pairing introduce additional uncertainty when taking the symplectic complements.  As seen in Example~\ref{ex:theoryd15}, this leads to legitimate variation of the Ekedahl-Oort type. Our approach is to bound this uncertainty in terms
of the number of $\bot$'s that occur in a word $w$. Proposition~\ref{prop: bounds dim(w(L))} gives constants $L(X,\pi,w)$ and $ U(X,\pi,w)$ such that
\begin{equation}\label{eq: lower and upper bound}
    L(X,\pi,w) \leq \dim_k( w(L)) \leq U(X,\pi,w).
\end{equation}
These constants are determined inductively on $t$, i.e., the number of $\bot$'s
occurring in the word $w$.
When $t=1$, so that there are no $\bot$'s in $w$, the range in \eqref{eq: lower and upper bound}
is approximately $l_X=g_X-f_X$, where $f_X$ is the $p$-rank of $X$. When 
$w = V^{n_{t+1}}\bot w'$ and $w'$ is a word containing $\bot$ exactly $t-1$ times,
we find that the range of possible values of $\dim_k(w(L))$ (i.e. $U(X,\pi,w) - L(X,\pi,w)$)
is approximately $2l_X$ more than the range for $\dim_k(w'(L)) $. %
In particular, 
we determine $\dim_k( w(L))$ up to an error that is approximately $2l_X \cdot t$. 
Our main result, Theorem~\ref{theorem: final type bounds}, then ascertains the effects of these bounds on the final type of $Y$. Finally, Theorem~\ref{thm: log intro} follows
by bounding the total number of $\bot$'s that are necessary to obtain all possible
subspaces
$ w(\HdR^1(Y))$ when $\pi$ has only one branch point.

\begin{remark}
Our analysis and the Ekedahl-Oort types in Theorems~\ref{thmA} and \ref{theorem: final type bounds} exhibit a kind of local-to-global principle.  Each point of ramification makes an independent contribution depending on the ramification break, and this combined with the Ekedahl-Oort type of the base gives the overall behavior.  This sort of behavior has also been seen in the Riemann-Hurwitz formula, the Deuring-Shafarevich formula, and bounds for $a$-numbers and higher $a$-numbers \cite{boohercais,groen}.  This also appears in work of Garnek, which obtains local-global decompositions of $\HdR^1(Y)$ as a $\Z/p\Z$-module \cite{garnek,garnek2}.  (It is unfortunately not compatible with Frobenius and Verschiebung.)  However, there is a not a simple local-to-global principle for Ekedahl-Oort types (or even $a$-numbers) of Artin-Schreier curves.  Groen gives an example of an Artin-Schreier curve whose $a$-number depends on the location of its branch points, not just the breaks in the ramification filtration \cite[Theorem 2.7.1]{StevenPhD}.
\end{remark}

\begin{remark}
Our results are specific to characteristic two for (at least) two reasons.  The first is that in odd characteristic even if the base curve is ordinary (i.e. $\PP^1$) the Ekedahl-Oort type of the cover (or even just the $a$-number) is not determined by the ramification and can take on a wide range of possibilities.  So there is no hope of a direct analog of Theorem~\ref{thmA} %
when the base curve is not ordinary. 

The second reason is technical, and has to do with the construction of the enhanced differentials $\{\widetilde{\omega}_{i,j}\}$ which we essentially understand.  The fact that most of $L \subset \HdR^1(Y)$ can be constructed with differentials pulled back from $X$ is also specific to characteristic two.  In contrast, on an Artin-Schreier curve in odd characteristic this is patently false.  For example, consider the curve in characteristic three given by $y^3 - y = f(x)$ where $f(x)$ is a polynomial of degree $d$ with $\gcd(d,3)=1$.   Writing $\omega = (g_ 0(x) + g_1(x) y + g_2(x) y^2) dx$ with $g_i(x)$ polynomials, we see that $\omega$ is regular if and only if $\deg(g_0) \leq \ceiling{\frac{2d}{3}}-2$, $\deg(g_1) \leq \ceiling{\frac{d}{3}}-2$, and $g_2=0$ \cite[Lemma 3.7]{boohercais}.  So there are many differentials which involve $y$ and cannot be analyzed based on the action of $V$ on spaces of differentials with poles on the base.
\end{remark}

\subsection*{Acknowledgments}

The authors thank Bryden Cais, Rachel Pries, and Damiano Testa for helpful conversations.

\section{Enhanced differentials of the second kind and de Rham cohomology} \label{sec:enhanced differentials}

In this section we let $X$ be a smooth proper connected curve over a field $k$. We make no assumption on the characteristic of $k$. Let $k(X)$ denote the function field of $X$. For any closed point $Q \in X$ we let $\widehat{\mathcal{O}}_{X,Q}$ denote the completion of the local ring at $Q$ and we let $K_Q$ denote the fraction field of $\widehat{\mathcal{O}}_{X,Q}$.

\subsection{Enhanced differentials of the second kind}
Differentials of the second kind are a classic concept over the complex numbers.  Rosenlicht introduced and studied the concept in arbitrary characteristic \cite{rosenlicht}.

\begin{definition}
    A differential of the second kind on $X$ is a meromorphic differential (i.e. a differential that may have poles)  $\omega$ such that for every $Q \in X$, there exists
    $f \in K_Q$ such that $\omega - df \in \Omega^1_{\widehat{\mathcal{O}}_{X,Q}/k}$.
\end{definition}

Rosenlicht requires there be $f \in k(X)$ with $\omega -d f$ regular at $Q$, which is equivalent.

\begin{remark}
Note that when $k$ has characteristic zero, a differential $\omega$ is of the second kind if and only if the residue at $Q$ is zero for all $Q \in X $.   This fails in characteristic $p$ as the derivative of a $p$-th power is zero.  In particular, to be of the second kind the local expansion of $\omega$ in terms of a uniformizer $t$ at $Q$ cannot include terms of the form $t^i dt$ where $i<0$ and $i \equiv -1 \mod{p}$.
 
Note that exact differentials are automatically of the second kind. 
\end{remark}
 
The classical isomorphism between de Rham cohomology and differentials of the second kind modulo exact differentials breaks down in characteristic $p$. (For an expository account, see the note of Gurski \cite{gurski}.)  
We introduce the notion of enhanced differentials of the second kind to recover this connection.  An enhanced differential of the second kind consists of a differential of the second kind along with a choice of an antiderivative of the ``tail'' of $\omega$ at each $Q$.
 
    \begin{definition}
        An enhanced differential of the second kind on $X$ is a pair $(\omega,(f_Q)_{Q \in X})$ where
        \begin{enumerate}
            \item $\omega$ is a differential of the second kind on $X$;
            \item for each $Q \in X$, $f_Q \in K_Q/\widehat{\mathcal{O}}_{X,Q}$ and if $\tilde{f}_Q \in K_Q$ represents $f_Q$ then $\omega-d\tilde{f}_Q \in \Omega^1_{\widehat{\mathcal{O}}_{X,Q}/k}$. 
        \end{enumerate} 
    \end{definition}

Note that almost all the $f_Q$'s are automatically in $\widehat{\mathcal{O}}_{X,Q}$.  We will often abuse notation by not distinguishing between the element $\tilde{f}_Q$ of $K_Q$ and the equivalence class $f_Q$ it represents.
    
    \begin{definition} \label{def:S-enhanced}
        Let $S$ be a finite set of closed points in $X$. An $S$-enhanced differential of the second kind on $X$ is a pair $(\omega,(f_Q)_{Q \in S})$ such that
        \begin{enumerate}
            \item $\omega$ is a differential of the second kind that is regular away from $S$;
            \item $f_Q \in K_Q/\widehat{\mathcal{O}}_{X,Q}$ and
            if $\tilde{f}_Q$ is a lift of $f_Q$ to $K_Q$ we have $\omega-d\tilde{f}_Q \in \Omega^1_{\widehat{\mathcal{O}}_{X,Q}/k}$. 
        \end{enumerate} 
    \end{definition}

If we take $f_Q=0$ for $Q \in X-S$, an $S$-enhanced differential of the second kind naturally becomes an enhanced differential of the second kind.

\begin{definition}
We let $E_X$ (resp. $E_{X,S}$) denote the $k$-vector space of enhanced differentials (resp. $S$-enhanced differentials) on $X$. 
Define $d_X : k(X) \to E_X$ by $d_X(f)= (df, (f)_{Q \in X})$. 
Letting $A_S = \mathcal{O}_X(X-S)$ be the ring of regular functions on $X - S$, we naturally restrict $d_X$ to obtain a map $d_{X,S} : A_S \to E_{X,S}$.  When $X$ and $S$ are clear from context, we will drop them from the notation.
\end{definition}
\subsection{Relation to algebraic de Rham cohomology}

    \begin{theorem}
        \label{t: enhanced diff of second kind compute H1dr}
        Letting $S$ be a nonempty finite set of closed points of $X$, we see
        \begin{equation} \label{eq:derhamisomorphism}
            \HdR^1(X) \simeq \frac{E}{d_X(k(X))} \cong  \frac{E_S}{d_{X,S}(A_S)}.
        \end{equation}
        Furthermore, the duality pairing $\HdR^1(X) \times \HdR^1(X) \to \HdR^2(X) \simeq k$ is given by
        \begin{align*}
            \langle (\omega,(f_Q)_{Q \in S}), (\tau, (g_Q)_{Q \in S}) \rangle &=  \sum_{Q \in S} \Res_Q( \tilde{g}_Q \omega - \tilde{f}_Q \tau - \tilde{g}_Qd\tilde{f}_Q) .
        \end{align*}
        In particular, if $\omega$ is a global differential we have
                \begin{align*}
            \langle (\omega,(0)_{Q \in S}), (\tau, (g_Q)_{Q \in S}) \rangle &=  \sum_{Q \in S} \Res_Q( \tilde{g}_Q \omega ) .
        \end{align*}
    \end{theorem}

We will establish \eqref{eq:derhamisomorphism} using an acyclic ``pole order'' resolution of the de Rham complex.   We are not aware of a direct reference, but similar ideas appear in various places in the literature.  For example, Coleman introduces a somewhat similar description using differentials of the second kind with bounded poles \cite[\S5]{coleman98}. %
This can also be viewed as an extension of the r\'{e}partition description of $H^1(X,\cO_X)$ \cite[\S II.5]{Serre1988}.

Fix a nonempty set $S$ of closed points on the curve $X$, and let $\displaystyle D = \sum_{Q \in S} Q$. 

\begin{proposition} \label{prop:acyclic resolution}
 If $\deg(nD) > \max( 2g_X-2, 0)$, the double complex $\cD(n)$ in Figure~\ref{fig:acyclic resolution} is an acyclic resolution of the de Rham complex $\Omega^\bullet_X$.
\end{proposition}

\begin{figure}[h]
\[\begin{tikzcd}
	& 0 \\
	0 & {\Omega^1_X} & {\Omega^1_X((n+1)D)} & {\Omega^1_X((n+1)D)|_{(n+1)D}} & 0 \\
	0 & {\mathcal{O}_X} & {\mathcal{O}_X(nD)} & {\mathcal{O}_X(nD)|_{nD}} & 0 \\
	& 0
	\arrow[from=2-1, to=2-2]
	\arrow[from=2-2, to=1-2]
	\arrow[from=2-2, to=2-3]
	\arrow[from=2-3, to=2-4]
	\arrow[from=2-4, to=2-5]
	\arrow[from=3-1, to=3-2]
	\arrow["d"', from=3-2, to=2-2]
	\arrow[from=3-2, to=3-3]
	\arrow["d"', from=3-3, to=2-3]
	\arrow[from=3-3, to=3-4]
	\arrow["d"',from=3-4, to=2-4]
	\arrow[from=3-4, to=3-5]
	\arrow[from=4-2, to=3-2]
\end{tikzcd}\]
\caption{The Pole Order Resolution $\cD(n)$ of the de Rham Complex.}
\label{fig:acyclic resolution}
\end{figure}
    
\begin{proof}
Recall that $\cO_X(nD)|_{nD}$ is naturally isomorphic to $\cO_{nD}(nD)$, and the bottom row can also be obtained by twisting the closed subscheme exact sequence.  There is a similar interpretation for the row of differentials.  Thus the rows are exact.
Now the degrees of $\Omega^1_X(-nD)$ and $\cO_X(-(n+1)D)$ are negative as $\deg(nD) > \max(2g_X-2,0)$.  Hence by Serre duality $\cO_X(nD)$ and $\Omega^1_X((n+1)D)$ are acyclic.  The third term in each row is acyclic as it is supported in dimension zero.  
\end{proof}

\begin{definition}
For a positive integer $n$,  define
\begin{equation*}
E_{X,nD} \colonequals \{ (\omega,(f_Q)_{Q \in S}) \in E_{X,S} : \ord_Q(\omega) \geq - (n+1) \textrm{ and }  \ord_Q(f_Q) \geq -n \textrm{ for all } Q \in S\}.
\end{equation*}
\end{definition}

\begin{theorem} \label{theorem:differentials second kind fixed poles}
If $\deg(nD) > \max( 2g_X-2, 0)$ then the first de Rham cohomology of $X$ is isomorphic to $E_{X,nD} / d_X(\Gamma(\cO_X(nD))$, where $\Gamma$ denotes global sections.
\end{theorem}

\begin{proof}
The algebraic de Rham cohomology of $X$ is the hypercohomology of the de Rham complex, which we can compute as the cohomology of the total complex of the global sections of any acyclic resolution.  Using the resolution from Proposition~\ref{prop:acyclic resolution}, the total complex is
\begin{equation} \label{eq:totalycomplex}
0 \to  \Gamma(\cO_X(nD)) \overset{\delta_0} \to \Gamma(\Omega^1_X((n+1)D)) \oplus \Gamma(\cO_X(nD)|_{nD}) \overset{\delta_1} \to  \Gamma(\Omega^1_X((n+1)D)|_{(n+1)D}) \to 0
\end{equation}
Now $\cO_X(nD)|_{nD}$ is a skyscraper sheaf supported on $D$, with the stalk at $Q$ consisting of functions with a pole of order at most $n$ at $Q$ modulo functions regular at $Q$.  Similarly
$\Gamma(\Omega^1_X((n+1)D)|_{(n+1)D})$ consists of ``tails'' of differentials at $S$.  
Recalling the definition of the maps $\delta_i$ in the total complex, we see $\ker \delta_1$ consists of pairs $(\omega, (f_Q)_{Q \in S})$ where $$\delta_1(\omega, (f_Q)_{Q \in S}) = \pm (\omega -  d(f_Q))_{Q \in S} =0 \in \Omega^1_X((n+1)D).$$  In other words, $\omega - d(f_Q)$ is regular at $Q$ for each $Q \in S$.  Thus there is an isomorphism $\ker \delta_1 \to E_{X,nD}$.  For $f \in \Gamma(\cO_X(nD))$, note that $\delta_0(f)= (df, (f_Q)_{Q \in S})$.  Thus we obtain an isomorphism $\ker \delta_1 / \textrm{Im} \delta_0 \simeq E_{X,nD} / d_X(\Gamma(\cO_X(nD))$. 
\end{proof}

\begin{proof}[Proof of  Theorem~\ref{t: enhanced diff of second kind compute H1dr}]
When $n \leq m$, there are natural maps $\cD(n) \to \cD(m)$ and $E_{X,nD} \to E_{X,mD}$.  These induce maps $\HdR^1(X) \to \HdR^1(X)$ and $E_{X,nD} / d_X(\Gamma(\cO_X(nD))) \to E_{X,mD} / d_X(\Gamma(\cO_X(mD)))$ which are compatible with the identification of Theorem~\ref{theorem:differentials second kind fixed poles}.  Taking the limit, we obtain an isomorphism between $\HdR^1(X)$ and $E_{X,S}/d_{X,S}(A_S)$.  A similar argument taking the limit over nonempty finite sets $S$ of $X$ gives the isomorphism between $\HdR^1(X)$ and $E/d_X(k(X))$.

It remains to establish the formula for the pairing.  We will obtain it indirectly in Proposition~\ref{p: pairing}, by comparing the description of $\HdR^1(X)$ in terms of differentials of the second kind to the \v{C}ech description and using the known form of the pairing in that case.  
\end{proof}

\begin{remark}
We could also construct resolutions of $\cO_X$ (resp. $\Omega^1_X$) using functions (resp. differentials) on $X$ having poles only along $D$, without specifying the particular bound $n$ (resp. $n+1$) on the pole order.  The same approach would then give Theorem~\ref{t: enhanced diff of second kind compute H1dr} without the need to take limits.  The advantage of Theorem~\ref{theorem:differentials second kind fixed poles} is that it allows explicit computation of algebraic de Rham cohomology of curves, similar to Weir's implementation of algebraic de Rham cohomology using the \v{C}ech resolution \cite{weir_code}.
\end{remark}

\subsection{Relation to the \v{C}ech description of algebraic de Rham
cohomology}
\label{ss: 2nd kind vs Cech}
We now compare the enhanced differential description of $\HdR^1(X)$
with the \v{C}ech description given in \cite[\S 5]{Oda} (see also
\cite{ElkinPries}).  Let $S$ be a finite set of points in $X$. Let $U_1=X\backslash S$ and let $U_2$ be an affine open subscheme of $X$ containing $S$. Then $\mathfrak{U}=\{U_1,U_2\}$ is a cover of $X$. We define the \v{C}ech 1-cocycles to be
\[ Z^1(\mathfrak{U}) = \left\{ \left(\omega_1,\omega_2,f\right)~\middle| ~\begin{array}{c} \omega_i \in \Omega^1_X(U_i) \\f \in \mathcal{O}_X(U_1 \cap U_2) \\ \omega_1 - \omega_2 = df\end{array}\right\}.\]
The \v{C}ech 1-coboundaries are defined by
\[B^1(\mathfrak{U}) = \left\{(\omega_1,\omega_2,f) ~\middle | ~ \begin{array}{c}  \text{There exists }f_1,f_2 \text{ with } f_i \in \mathcal{O}_X(U_i) \\
\text{such that }\omega_i=df_i \text{ and } f=f_1-f_2\end{array}\right\}.\]
Then we have
\begin{equation*}
    \HdR^1(X) \cong Z^1(\mathfrak{U})/B^1(\mathfrak{U}).
\end{equation*}
We now define a map $\imath : Z^1(\mathfrak{U}) \to E_S$ by
\[\imath (\omega_1,\omega_2,f) = \left( \omega_1, (f \mod \widehat{\mathcal{O}}_{X,Q})_{Q \in S} \right).\]
Note the only poles of $\omega_1$ are at $S$, and since $\omega_2$ is regular at the points of $S$, the relation $\omega_1-\omega_2=df$ means that
$\omega_1 - df \in \Omega^1_{\widehat{\mathcal{O}}_{X,Q}}$ for each $Q \in S$. 

\begin{proposition} \label{proposition:cech second kind}
The map $\imath$ induces an isomorphism
\begin{equation}\label{eq: map from cech to enhanced differentials} \frac{Z^1(\mathfrak{U})}{B^1(\mathfrak{U}) }\to \frac{E_S}{d_{X,S}(A_S)}.
\end{equation}
\end{proposition}

\begin{proof}
Note that $\imath$ factors through the quotient as if
$(\omega_1,\omega_2,f)$ is a \v{C}ech 1-coboundary, then $\left( \omega_1, (f \mod \widehat{\mathcal{O}}_{X,Q})_{Q \in S} \right)$ is equal to $d_{X,S}(f_1)$.  
To see that this map is an isomorphism, we describe the inverse.
Let $(\omega, (f_Q)_{Q \in S})$ be an $S$-enhanced differential. 
Let $f$ be a function in $\mathcal{O}_X(U_1\cap U_2)$ with 
$f \equiv f_Q \mod \widehat{\mathcal{O}}_{X,Q}$ for each $Q \in S$.
Such an $f$ exists by Riemann-Roch.  Then $(\omega, \omega-df, f)$ is a \v{C}ech 1-cocycle.  
A short calculation shows that a different choice of $f$ gives
a $1$-cocycle that differs by a $1$-coboundary.
If $(\omega, (f_Q)_{Q \in S})$ is exact, so equal to $d_{X,S}(f_1)$, then
we can take $f=f_1$, so the \v{C}ech 1-cocycle obtained is $(\omega,0,f_1)$,
which is a \v{C}ech 1-coboundary.  It is straightforward to verify this induces an inverse to the map in \eqref{eq: map from cech to enhanced differentials}. 
\end{proof}

The \v{C}ech $2$-cocycles and $2$-coboundaries are given as follows
\begin{align*}
    Z^2(\mathfrak{U}) &= \Omega_X^1(U_1 \cap U_2) \\
    B^2(\mathfrak{U}) &= \left\{ \omega \in \Omega_X^1(U_1 \cap U_2)~\middle | ~ \begin{array}{c} \text{There exists }\omega_1,\omega_2 \text{ with } \omega_i \in \Omega^1_X(U_i) \\
\text{such that }\omega = \omega_1-\omega_2\end{array}\right\}.
\end{align*}
Then we have
\[ \HdR^2(X)\cong H^1(X,\Omega_X^1)   \cong Z^2(\mathfrak{U})/B^2(\mathfrak{U}).\]
The trace map on \v{C}ech $2$-cocycles is given explicitly up to sign (see \cite[Theorem 5.2.3]{conrad_00} for the connection between the modern formulation and the classical formulation in terms of residues) %
by 
\begin{align*}
    t:\HdR^2(X) &\to k \\
    \omega &\mapsto \sum_{Q \in S} \Res_Q(\omega).
\end{align*}
The cup product is given by
\[(\omega_1,\omega_2,f) \cup (\tau_1,\tau_2,g) = g\omega_1 -f\tau_1\in \Omega_X^1(U_1\cap U_2), \]
so the pairing $\langle\cdot , \cdot \rangle :\HdR^1(X) \times \HdR^1(X) \to k$ is given by 
\begin{equation}\label{eq: pairing on cech}
    \langle (\omega_1,\omega_2,f), (\tau_1,\tau_2,g)\rangle = \sum_{Q \in S} \Res(g\omega_1 - f\tau_1).
\end{equation}

\subsection{The cup product on enhanced differentials of the second kind}

\begin{proposition}
    \label{p: pairing} Let $(\omega, (f_Q)_{Q \in S})$ and $(\tau,(g_Q)_{Q \in S})$ be
$S$-enhanced differentials of the second kind. Then the
pairing $[\cdot,\cdot]$ on the $S$-enhanced differentials of the second kind defined by
\begin{equation} \label{eq: definition of pairing} \left[ (\omega, (f_Q)_{Q \in S}), (\tau, (g_Q)_{Q \in S}) \right] =
\sum_{Q \in S}\left( \Res_Q(\tilde{g}_Q\omega) - \Res_Q(\tilde{f}_Q\tau) - \Res_Q(\tilde{g}_Qd\tilde{f}_Q) \right),
\end{equation}
where $\tilde{f}_Q$ (resp. $\tilde{g}_Q$) is any lift of $f_Q$ (resp. $g_Q$), is well-defined and agrees with the pairing $\langle \cdot,\cdot\rangle $
on $\HdR^1(X)$.
\end{proposition}
\begin{proof}
    First, we show that $[\cdot,\cdot]$ is independent of the lift $\tilde{f}_Q$ and $\tilde{g}_Q$. Since the right side of \eqref{eq: definition of pairing} is clearly anti-symmetric, it is enough
    to show that the expression is independent of lift of $g_Q$. Any other
    lift is of the form $\tilde{g}_Q + h_Q$ where $h_Q \in \widehat{\mathcal{O}}_{X,Q}$. Write $\omega=d\tilde{f}_Q + \omega_0$,
    where $\omega_0$ is holomorphic. Then $\Res_Q(h_Q\omega_0)=0$. In particular, we see that $\Res_Q(h_Q\omega)= \Res_Q(h_Qd\tilde{f}_Q)$.
    A short calculation shows $[\cdot,\cdot]$ does not depend on the lift.

    Let $f$ (resp. $g$) be a function in $\mathcal{O}_X(U_1\cap U_2)$ with 
    $f \equiv f_Q \mod \widehat{\mathcal{O}}_{X,Q}$ (resp. 
    $g \equiv g_Q \mod \widehat{\mathcal{O}}_{X,Q}$) for each $Q \in S$.
    Then using $f$ (resp. $g$) for the lift of $f_Q$ (resp. $g_Q$) we obtain 
    \[ \left[ (\omega, (f_Q)_{Q \in S}), (\tau, (g_Q)_{Q \in S}) \right] =
\sum_{Q \in S} \left( \Res_Q(g\omega) - \Res_Q(f\tau) - \Res_Q(gdf). \right)\]
    We then see that $[\cdot,\cdot]$ agrees with the pairing
    $\langle \cdot, \cdot\rangle$ by looking explicitly at the isomorphism
    $\imath$ described in Section \S \ref{ss: 2nd kind vs Cech}
    and the formula \eqref{eq: pairing on cech} for the pairing  of
    \v{C}ech $1$-cocycles.\qedhere

    \end{proof}

\begin{remark}
    In characteristic $0$, we can choose $\tilde{f}_Q$ and $\tilde{g}_Q$
    to be actual antiderivatives of $\omega$ and $\tau$. In this case,
    the pairing becomes
    \begin{equation*}
        \langle (\omega, (f_Q)_{Q \in S}), (\tau, (g_Q)_{Q \in S}) \rangle = \sum_{Q \in S} \Res_Q(\tilde{g}_Q \omega),
    \end{equation*}
    since $d(\tilde{f}_Q\tilde{g}_Q)$ has zero residue. In particular,
    this agrees with the pairing on classical differentials
    of the second kind given by Chevalley \cite{chevalley}.  Note the Equation~\eqref{eq: definition of pairing} is also quite similar to the one in Coleman's setting \cite[Corollary 5.1]{coleman98}, which is also proven by relating to the \v{C}ech description.
\end{remark}

\subsection{Frobenius and Verschiebung operators}
\begin{definition}
    Suppose that the field $k$ is perfect of characteristic $p$, and let $\sigma : k \to k$ be the $p$-th power Frobenius map. The $p$-th power map
    induces a $\sigma$-linear map of the de Rham complex $\Omega_X^\bullet$.
    We define the Frobenius map $F : \HdR^1(X) \to \HdR^1(X)$ to
    be the induced $\sigma$-linear map on cohomology. We define 
    the Verschiebung $V : \HdR^1(X) \to \HdR^1(X)$ to be the
    $\sigma^{-1}$-linear map adjoint to $F$. Note due to semilinearity, the adjointness condition is $\langle V x,y \rangle^p = \langle x, Fy \rangle$. %
\end{definition}

\begin{definition}
    We define the Cartier operator $V_X: \Omega_{k(X)/k} \to \Omega_{k(X)/k}$ as
    follows:
    If we let $t_Q$ be a local parameter at $Q \in X$ then we may uniquely write 
\begin{equation}\label{eq: break up differential}
\omega = \sum_{i=0}^{p-1} h_{Q,i}^p t^i_Q \frac{dt_Q}{t_Q}
\end{equation}
where the $h_{Q,i}$ are rational functions on $X$.  Then a direct definition of the Cartier operator is that $V_X(\omega) = h_{Q,0} \frac{dt_Q}{t_Q}$.  
It is well known that this definition does not depend on the choice of $Q$ or
$t_Q$ and that $V_X(\omega)=0$ if and only if $\omega$ is exact (see e.g. \cite[Section 7]{Katz-Nilpotent_connections_and_monodromy_theory} for a detailed discussion).
\end{definition}

\begin{proposition} \label{prop:F and V on S-enhanced differentials}
Let $S$ be a nonempty set of points of $X$, and $(\omega,(f_Q)_{Q \in S})$ be an $S$-enhanced differential of the second kind.  Then 
\[
F (\omega,(f_Q)_{Q \in S}) \sim (0, (f_Q^p)_{Q \in S}) \quad \text{and} \quad V(\omega,(f_Q)_{Q \in S})\sim (V_X(\omega), (0)_{Q \in S}),
\]
where $\sim$ denotes equivalence in $\HdR^1(X)$.
\end{proposition}

\newcommand{\Coeff}{\operatorname{Coeff}}

\begin{proof}
On \v{C}ech $1$-cocycles we have $V(\omega_1,\omega_2,f)=(V(\omega_1),V(\omega_2),0)$ and 
$F(\omega_1,\omega_2,f)=(0,0,f^p)$ (see \cite[Definition 5.6]{Oda}).
The result then follows using the comparison in Section~\ref{ss: 2nd kind vs Cech}.  %
\end{proof}

\section{\texorpdfstring{$S$}{S}-enhanced differentials on double covers in characteristic \texorpdfstring{$2$}{2}} \label{section: omegaij}

We now assume $k$ is an algebraically closed field of characteristic $2$. Let $\pi: Y \to X$ be a $\Z/2\Z$-cover of smooth, proper, connected curves. %
Let $P_1, \ldots , P_m \in X$ be the branch points of $\pi$ and let $Q_i= \pi^{-1}(P_i) \in Y$ be the ramified points of $\pi$. Let $B=\{P_1, \ldots, P_m\}$ and let $S= \{Q_1, \ldots , Q_m\}$. For each $i$, let $d_i$ be the unique break in the ramification filtration of $\text{Gal}(K_{Q_i}/K_{P_i})$. Note that $d_i$ is an odd positive integer. Let $g_X$ be the genus of $X$ and let $g_Y$ be the genus of $Y$. Then the Riemann-Hurwitz theorem \cite[Corollary IV.2.4]{Hartshorne} gives
\begin{align*}
    g_Y &= 2g_X - 1 + \sum_{i=1}^m \frac{d_i+1}{2}.
\intertext{Moreover, let $f_X$ be the $p$-rank of $X$ and let $f_Y$ be the $p$-rank of $Y$. Then the Deuring-Shafarevich formula \cite[Theorem 4.1]{SubraoDS} yields}
    f_Y &= 2f_X - 1 + m .
\end{align*}
We define $l_X\colonequals g_X-f_X$ to be the \emph{local rank} of $X$, and similarly 
$$l_Y\colonequals g_Y-f_Y = 2l_X + \sum_{i=1}^m \frac{d_i-1}{2}$$ is the local rank of $Y$. Note that $X$ is ordinary if and only if $l_X=0$.

\subsection{A decomposition of Dieudonn\'{e} modules}
The first step towards understanding the Dieudonn\'{e} module structure of $\HdR^1(Y)$ is to split off the parts on which $F$ or $V$ act bijectively. Let $\langle -, - \rangle$ denote the symplectic pairing $\HdR^1(Y) \times \HdR^1(Y) \to k$ described in Theorem~\ref{theorem:differentials second kind fixed poles}.

\begin{lemma} \label{lemma: decomposition k[F,V]}
There exists a decomposition of Dieudonn\'{e} modules
\begin{equation} \label{eq: decomposition k[F,V]}
    \HdR^1(Y) = U \oplus Z \oplus L
\end{equation}
where $U$ and $Z$ both have dimension $f_Y$ and $L$ has dimension $2l_Y$.
Furthermore:
\begin{enumerate}[(i)]
    \item $V$ acts bijectively on $U$ and $F$ acts trivially on $U$;
    \item $F$ acts bijectively on $Z$ and $V$ acts trivially on $Z$;
    \item $F$ and $V$ act nilpotently on $L$.
    \item $\langle \cdot, \cdot \rangle : U \times Z \to k$ is a perfect pairing.
\end{enumerate}
\end{lemma}
\begin{proof}
First, viewing $\HdR^1(Y)$ as a $k[V]$-module, (a semilinear version of) Fitting's lemma yields a decomposition $\HdR^1(Y) = U \oplus R$, where $V$ acts bijectively on $U$ and nilpotently on $R$. Since $F \circ V = 0$, it follows that $F$ acts trivially on $U$. Then, viewing $R$ as a $k[F]$-module, Fitting's lemma gives a decomposition $R=Z \oplus L$, where $F$ acts bijectively on $Z$ and nilpotently on $L$. Again, it follows from $V \circ F = 0$ that $V$ must act trivially on $Z$. 

It remains to verify the claim about the pairing. For this, we use the fact $\langle Fm_1 , m_2\rangle = \langle m_1 , Vm_2 \rangle ^p$ (see \cite[(2.6)]{Moonengsas}), which implies that elements of $U$ can only pair non-trivially with elements of $Z$ and vice versa. Finally, the statement follows from the fact that $U$ and $Z$ have the same dimension and that the pairing on $\HdR^1(Y)$ is perfect.  %
\end{proof}

Thus the Dieudonn\'{e} module structure of $\HdR^1(Y)$ is determined by the Dieudonn\'{e} module structure of the local part $L$.  To this end, we will construct a large subspace $W$ of $L$ and investigate the action of $F$ and $V$ and the pairing on this subspace throughout the rest of this section.

\subsection{Constructing \texorpdfstring{$S$}{S}-enhanced differentials}

\begin{lemma}
    \label{l: uniformizer up above} We can choose a uniformizer $t_i$ at $P_i$ and $u_i$ at $Q_i$ such that
    \begin{align*}
         t_i &= u_i^2 + u_i^{ d_i+2} + O(u_i^{2d_i + 2}).
    \intertext{In particular, we have }
         dt_i &= [u_i^{d_i + 1} + O(u_i^{2d_i + 2})]du_i.
    \end{align*}
\end{lemma}

\begin{proof}

    Let $t$ be a uniformizer of $K_{P_i}$, and
    consider an Artin-Schreier equation defining $K_{Q_i}$ over $K_{P_i}$
    \begin{equation}\label{eq: AS explicit}
        y^2 + y= f(t).
    \end{equation}
    We may take $f(t)$ to have a pole of order $d$. First, we claim that
    there exists a uniformizer $t_i$ of $K_{P_i}$ with $t_i^{-d_i}=f(t)$. To see this, note that
    $t^{d_i}f(t)=c\cdot u(t)$, where $c \in k^\times$ and $u$ is a $1$-unit. Then $u$ has a $d_i$-th root since $d_i$ is coprime to the characteristic. Also,
    $c$ has a $d_i$-th root, since $k$ is algebraically closed. In particular, we have
    $t^{d_i}f(t)=v^{d_i}$. Then we can take $t_i$ to be $tv^{-1}$, so that \eqref{eq: AS explicit}
    becomes
    \begin{equation}\label{eq: AS explicit2}
        y^2+y=t_i^{-d_i}.
    \end{equation}
    Let $u_i=yt_i^{\frac{d_i+1}{2}}$. Since $y$ has a pole of order $d_i$ at $Q_i$, we see that
    $u_i$ is a uniformizer of $K_{Q_i}$. Multiplying \eqref{eq: AS explicit2} by $t_i^{d_i+1}$ gives
    \begin{equation}\label{eq: AS explicit3}
        t_i = u_i^2 + t_i^{\frac{d_i+1}{2}}u_i. 
    \end{equation}
    Recursively expanding \eqref{eq: AS explicit3}  gives
    \begin{equation*}
        t_i = u_i^2+\left( u_i^2 + t_i^{\frac{d_i+1}{2}}u_i\right)^{\frac{d_i+1}{2}}u_i=u_i^2 + u_i^{d_i+2} + O(u_i^{2d+2}).\qedhere
    \end{equation*}

\end{proof}

\begin{corollary} \label{cor: t^j}
    For all $j \in \Z$ we have $t_i^j = u_i^{2j} + O(u_i^{2j + d_i})$.
\end{corollary}
\begin{proof}
    From Lemma \ref{l: uniformizer up above} we know $\frac{t_i}{u_i^2} = 1 + O(u_i^{d_i})$. Thus, $\frac{t_i^j}{u_i^{2j}} = 1 + O(u_i^{d_i})$ (e.g. use the generalized binomial theorem). The corollary follows by multiplying by $u_i^{2j}$.
\end{proof}

\begin{lemma}
\label{l: construct wij}
Let $i \in \{ 1, \ldots, m\}$ and $j \in \{1, 2, \ldots, d_i-1\}$.  There exists a differential $w_{i,j}$ on $X$ such that, letting $\omega_{i,j} \colonequals \pi^* w_{i,j}$ denote the pullback to $Y$:
\begin{enumerate}[(i)]
\item $w_{i,j}$ is regular away from $P_i$ and there is $c'_i \in k$ such that the local expansion at $P_i$ is  
$$w_{i,j} = (c'_i t_i^{-j-1} + O(1))dt_i.$$

\item  if $1 \leq j \leq (d_i-1)/2$ then
    \begin{align*}
        \omega_{i,j} &= [u_i^{-2j + d_i - 1} + O(u_i^{d_i+1})]du_i 
    \end{align*}
\item  if $(d_i-1)/2 < j \leq d_i-1$ then
    \begin{align*}
        \omega_{i,j} &= [u_i^{-2j + d_i - 1} + O(u_i^{-2j + 2d_i - 1})]du_i.
    \end{align*}
\end{enumerate}
Furthermore, the above properties uniquely determine $w_{i,j}$ up to adding elements of $H^0(X,\Omega^1_X)$
and the order of $\omega_{i,j}$ at $Q_i$ is $-2j + d_i - 1$.
\end{lemma}

\begin{proof}
By Riemann-Roch there exists a differential $w'_{i,j}$ on $X$ that is regular away from $P_i$ and whose local expansion at $P_i$ is  
$$w'_{i,j} = (t_i^{-j-1} + O(1))dt_i.$$
The difference of two such differentials is regular everywhere, so the choice of $w'_{i,j}$ is unique up to adding an element of $H^0(X,\Om_X^1)$.
 By Corollary \ref{cor: t^j} we know that
    $t_i^{-j-1} + O(1) = u_i^{-2j-2} + O\left(u_i^{\min(0,-2j-2+d_i)}\right)$. Then by
    Lemma \ref{l: uniformizer up above} we have
    \begin{align*}
        \omega_{i,j} &= \left(u_i^{-2j-2} + O\left(u_i^{\min(0,-2j-2+d_i)}\right)\right)\left( c_i u_i^{d_i + 1} + O(u_i^{2d_i + 2})\right) du_i \\
        &= \left(c_iu_i^{-2j+d_i-1} + O\left(u_i^{\min(d_i+1,-2j+2d_i-1)}\right)  \right)du_i.
    \end{align*}
    When $j=1,\dots,\frac{d_i-1}{2}$ the minimum in the exponent is $d_i+1$ and
    when $j=\frac{d_i+1}{2}, \dots, d_i-1$ the minimum in the exponent is $-2j+2d_i - 1$.  Finally rescale so the leading term in $\omega_{i,j}$ is monic.
\end{proof}

\begin{corollary}
    \label{c: estimating corollary of tail of differential at off-points}
    For $i' \neq i$ the stalk of $\omega_{i,j}$ at $Q_{i'}$ satisfies $\omega_{i,j} = O\left(u_{i'}^{d_{i'}+1}\right)du_{i'}$. In particular, $\omega_{i,j}$
    is regular away from $Q_i$.
\end{corollary}
\begin{proof}
    Since $w_{i,j}$ is regular at $P_{i'}$, we have $w_{i,j} = O(1)dt_{i'}$. The corollary follows immediately from Lemma~\ref{l: uniformizer up above}.
\end{proof}

\begin{proposition}
    \label{p: description of S-enhanced differential}
    For $j=1,\dots, \frac{d_i-1}{2}$, the element 
    \[ \widehat{\omega}_{i,j} \colonequals  (\omega_{i,j}, (0, \dots,  0)) \]
    is an $S$-enhanced differential on $Y$. For $j=\frac{d_i+1}{2}, \dots, d_i-1$ the element 
    \[\widehat{\omega}_{i,j} \colonequals  (\omega_{i,j}, (0, \dots, u_i^{-2j + d_i}, \dots , 0))\] is an 
    $S$-enhanced differential on $Y$.  These are unique up to adding the pullback of a regular differential on $X$.
\end{proposition}

\begin{proof}
    For $j=1,\dots,\frac{d_i-1}{2}$, we know from Lemma~\ref{l: construct wij} and Corollary \ref{c: estimating corollary of tail of differential at off-points} that $\omega_{i,j}$
    is a global holomorphic differential on $Y$, which gives the desired result. For $j=\frac{d_i+1}{2},\dots,d_i-1$, recalling Definition~\ref{def:S-enhanced} we need to prove that
    \begin{align*} 
        \omega_{i,j} &= [u_i^{-2j + d_i - 1} + O(1)]du_i.
    \end{align*}
    But Lemma \ref{l: construct wij} shows that if $-2j + 2d_i - 1 \geq 1$ then
    \begin{align*}
                \omega_{i,j} &= [u_i^{-2j + d_i - 1} + O(u_i^{})]du_i
    \end{align*}
    which suffices to give the result.  

    Since the differentials $w_{i,j}$ on $X$ are unique up to adding an element of $\HH^0(X,\Om_X^1)$, the $S$-enhanced differentials $\widehat{\omega}_{i,j}$ are unique up to adding $S$-enhanced differentials of the form $ (\pi^*\eta , (0, \ldots ,0))$, for $\eta \in \HH^0(X,\Om_X^1)$. 
\end{proof}

\begin{example}
The $\omega_{i,j}$ for fixed $i$ describe a portion of the cohomology of the cover that mimics the cohomology of an Artin-Schreier cover of $\PP^1$.
Suppose $X = \PP^1$ and let $Y$ be given by an Artin-Schreier equation $y^2 + y = x^{-d_1}$. Then $Y \to X$ is ramified only over $\{0\}$ with break $d_1$ and a uniformizer above $0$ is given by $u = y x^{(d_1+1)/2}$. We directly see that
\[
\omega_{1,j} = x^{-j-1} dx \quad \text{and} \quad \widehat{\omega}_{1,j} = \begin{cases}
    (x^{-j-1} dx, 0), & j \leq \frac{d_1-1}{2}\\
   ( x^{-j-1} dx, (u^{-2j + d_1 }) ), & j > \frac{d_1-1}{2}.
\end{cases}
\]
\end{example}

\subsection{The action of \texorpdfstring{$F$}{F} and \texorpdfstring{$V$}{V}}

Since we wish to understand the Dieudonn\'{e} module structure of $\HdR^1(Y)$, it is a natural next step to study the action of $F$ and $V$ on the $S$-enhanced differentials $\widehat{\omega}_{i,j}$.

Using the decomposition of $k[F]$-modules
$$\HdR^1(Y) = \HdR^1(Y)^{\bij} \oplus \HdR^1(Y)^{\nil}$$
coming from Fitting's lemma, for any $\widehat{\omega} \in \HdR^1(Y)$ we can write $\widehat{\omega} = \widehat{\omega}^{\bij} + \widehat{\omega}^{\nil}$. 

\begin{proposition} \label{prop:nilpotent wij}
There exists a choice of differentials $\widehat{\omega}_{i,j}$ as in Proposition~\ref{p: description of S-enhanced differential} such that 
$$V(\widehat{\omega}_{i,j}^{\nil}) = \begin{cases}
        \widehat{\omega}_{j/2}^{\nil} + \pi^*\eta_{i,j} &\hbox{if $j$ is even} \\
        \pi^*\eta_{i,j} &\hbox{if $j$ is odd},
    \end{cases}$$
    where $\eta_{i,j}$ is an element of $\HH^0(X,\Om_X^1)$ on which $V$ acts nilpotently.  In particular, $V$ acts nilpotently on $\widehat{\omega}_{i,j}^{\nil}$.
\end{proposition}

\begin{proof}
Note first that $V$ kills $\widehat{\omega}_{i,j}^{\bij}$ since it is in the image of $F$. Thus it suffices to show the lemma for the $\widehat{\omega}_{i,j}$.

Since $\omega_{i,j}$ is regular away from $P_i$, so is $V(\omega_{i,j})$. Locally at $P_i$, we have
$$ V(w_{i,j}) = V((t_i^{-j-1} +O(1))dt_i) = \begin{cases}
    \left(t_i^{-\frac{j}{2}+1}+O(1)\right)dt_i & \hbox{if $j$ is even} \\
    O(1)dt_i &\hbox{if $j$ is odd}.
\end{cases} $$
Pulling back along $\pi$ and using Proposition~\ref{prop:F and V on S-enhanced differentials}, the local expansions show that $V(\widehat{\omega}_{i,j}) - \omega_{i,j/2}$ is the pullback of a differential which is regular on $X$.   Noting that $\widehat{\omega}_{i,j/2}= \widehat{\omega}_{i,j/2}^{\nil}$ since a regular differential is killed by $F$, we obtain the formula in the lemma.
It remains to be shown that we may assume that $V$ acts nilpotently on $\eta_{i,j}$ (and hence on $\widehat{\omega}_{i,j}$). For this, recall the decomposition of $k[V]$-modules from Fitting's lemma
$$\HH^0(X,\Om_X^1) = \HH^0(X,\Om_X^1)^{\bij} \oplus \HH^0(X,\Om_X^1)^{\nil}$$
and write $\eta_{i,j} = \eta_{i,j}^{\bij} + \eta_{i,j}^{\nil}.$ There exists a differential $\gamma_{i,j}\in \HH^0(X,\Om_X^1)$ such that $V(\gamma_{i,j})=\eta_{i,j}^{\bij}$. Recall that the differential $w_{i,j}$ on $X$ is chosen up to an element of $\HH^0(X,\Om_X^1)$, so we may replace $w_{i,j}$ by $w_{i,j}-\gamma_{i,j}$. Then we have 
$$V(w_{i,j}-\gamma_{i,j})=V(w_{i,j}) - V(\gamma_{i,j}) = \begin{cases}
    w_{i,j/2} + \eta_{i,j}^{\nil} &\hbox{if $j$ is odd} \\
    \eta_{i,j}^{\nil} &\hbox{if $j$ is even}.
\end{cases}$$
Thus we may choose $w_{i,j}$ such that $V$ is nilpotent on $w_{i,j}$ and therefore $V$ is nilpotent on $\widehat{\omega}_{i,j}^{\nil}$. 
\end{proof}

\begin{definition} \label{defn:wijtilde}
Fixing differentials as in Proposition~\ref{prop:nilpotent wij}, we define 
$$\widetilde{\omega}_{i,j} \colonequals \widehat{\omega}_{i,j}^{\nil}=\widehat{\omega}_{i,j}-\widehat{\omega}_{i,j}^{\bij}.$$
For $i=1, \ldots ,m$, define $W_i \colonequals \spa \{\widetilde{\omega}_{i,j} \; | \; 1 \leq j \leq d_i-1\}$.  Given $1 \leq \ell \leq d_i-1$, we also define
    $$W_{i,l} \colonequals  \spa \{ \widetilde{\omega}_{i,j} \; | \; j \leq l \} \subseteq W_i.$$
Finally, define $W \colonequals \bigoplus_{i=1}^m W_i$.
\end{definition}

Note that $F$ is nilpotent on $\widetilde{\omega}_{i,j}$ and that $F$ kills regular differentials so $\widetilde{\omega}_{i,j} = \widehat{\omega}_{i,j}$ if $j \leq \frac{d_i-1}{2}$.  It is also clear that $\dim_k W_{i,l} =l$.

\begin{remark} \label{remark: W subspace of L}
Note that $W$ is a subspace of the space $L$ in Equation~\eqref{eq: decomposition k[F,V]}.
In general, $W$ is not a Dieudonn\'{e} submodule of $L$.   The dimension of $L$ is $2l_Y = 4l_X + \sum_{i=1}^m (d_i-1)$ and the dimension of $W$ is $\sum_{i=1}^m (d_i - 1)$.
\end{remark}

\subsection{The symplectic pairing}

We already know that the pairing $U \times Z \to k$ is perfect and the spaces $U \oplus Z$ and $L$ are orthogonal.  So the next step is to analyze the symplectic pairing between the $S$-enhanced differentials $\widetilde{\omega}_{i,j}$ on $Y$.  

\begin{proposition} \label{proposition: symplectic pairing}
Let $\widetilde{\omega}_{i,j}$ be an $S$-enhanced differential with $1 \leq j \leq \frac{d_i-1}{2}$. Let $\widetilde{\omega}_{i',j'}$ be arbitrary. Then the symplectic pairing is given by
\begin{equation*} \label{eq:omij pairing}
\langle \widetilde{\omega}_{i,j} , \widetilde{\omega}_{i',j'} \rangle =
\begin{cases}
    1 & \hbox{if $i=i'$ and $j+j' = d_i$} \\
    0 & \hbox{otherwise.}
\end{cases}
\end{equation*}
\end{proposition}

\begin{proof}
Recall that we had defined $\widetilde{\omega}_{i',j'} = \widehat{\omega}_{i',j'} - \widehat{\omega}_{i',j'}^{\bij}$. Since $j \leq (d_i-1)/2$, we have $\widetilde{\omega}_{i,j} = \widehat{\omega}_{i,j}$. By Proposition~\ref{prop:nilpotent wij}, $V$ is nilpotent on $\widehat{\omega}_{i,j}$, so let $r$ be such that $V^r(\widehat{\omega}_{i,j})=0$.  Note there exists $\delta_{i',j'}$ such that $F^r(\delta_{i',j'})=\widehat{\omega}_{i',j'}^{\bij}$. Then observe
\begin{align*}
    \langle \widetilde{\omega}_{i,j} , \widetilde{\omega}_{i',j'} \rangle 
    &= \langle \widehat{\omega}_{i,j} , \widehat{\omega}_{i',j'} - \widehat{\omega}_{i',j'}^{\bij} \rangle \\
    &= \langle \widehat{\omega}_{i,j}, \widehat{\omega}_{i',j'} \rangle 
    - \langle \widehat{\omega}_{i,j} , F^r(\delta_{i',j'}) \rangle \\
    &= \langle \widehat{\omega}_{i,j} ,\widehat{\omega}_{i',j'} \rangle 
    - \langle V^r(\widehat{\omega}_{i,j}), \delta_{i', j'} \rangle^{p^r} \\
    &= \langle \widehat{\omega}_{i,j}, \widehat{\omega}_{i',j'} \rangle.
\end{align*}
The rest of the proof is devoted to computing $\langle \widehat{\omega}_{i,j}, \widehat{\omega}_{i',j'} \rangle$ making use of  Theorem~\ref{theorem:differentials second kind fixed poles}.

First assume $i'\neq i$. %
Since $j \leq \frac{d_i-1}{2}$, the class $\widehat{\omega}_{i,j}$ only consists of a regular differential. If $j' \leq \frac{d_{i'}}{2}$, then we immediately obtain $\langle \omega_{i,j}, \omega_{i',j'} \rangle =0$. Otherwise, write $\widehat{\omega}_{i',j'} = (\omega_{i',j'}, (0, \ldots , u_{i'}^{-2j'+d_{i'}}, \ldots ,0).$ Then the pairing is given by
$$ \langle \widehat{\omega}_{i,j}, \widehat{\omega}_{i',j'} \rangle = \Res_{Q_{i'}} (u^{-2j'+d_{i'}} \omega_{i,j}).$$
Now, since $w_{i,j}$ is regular at $P_{i'}$, it follows from Lemma~\ref{l: uniformizer up above} that $\ord_{Q_{i'}}(\omega_{i,j}) \geq d_{i'}+1,$ so that
$$\ord_{Q_{i'}}(u^{-2j'+d_{i'}} \omega_{i,j}) \geq -2j'+d_{i'} + (d_{i'}+1) \geq 3,$$
since $j' \leq d_{i'}-1$. Therefore the residue vanishes.

Next we compute the shape of the pairing restricted to $W_i$. For $j' \leq \frac{d_i-1}{2}$, we know $\omega_{i,j}$ and $\omega_{i,j'}$ are regular differentials, and thus $\langle \widehat{\omega}_{i,j}, \widehat{\omega}_{i,j'}\rangle=0$. When $\omega_{i,j'}$ is not regular, we obtain
\begin{align*}
    \langle \widehat{\omega}_{i,j}, \widehat{\omega}_{i,j'}\rangle&= \Res_{Q_i}(u_i^{-2j'+d_i} \omega_{i,j}).
\end{align*}
By Lemma~\ref{l: construct wij} we have
\[u_i^{-2j'+d_i} \omega_{i,j} =  [u_i^{-2(j+j')+2d_i-1} + O(u_i^{-2j'+2d_i-1})]du_i \]
Since $j'\leq d_i-1$ we know that $O(u_i^{-2j'+2d_i-1})du_i$ has zero residue. Thus, we see that $u_i^{-2j'+d_i} \omega_{i,j}$ has residue one if $j+j'=d_i$ and zero otherwise. 
\end{proof}

\subsection{Further Decomposition of \texorpdfstring{$L$}{L}}

We have constructed and studied a large subspace $W = \bigoplus_{i=1}^m W_i \subset L$.  We end by studying the rest of $L$.  
Note we have natural inclusions
\begin{align*}
	\pi^* \HH^0(X,\Om_X^1) &\subset \HH^0(Y,\Om_Y^1) \subset \HdR^1(Y) \\
	\pi^* \HH^0(X,\Om_X^1) &\subset \pi^*\HdR^1(X) \subset \HdR^1(Y).
\end{align*}  

\begin{definition}
Define
\begin{align*}
	L_0 &\colonequals  L \cap \HH^0(Y,\Om_Y^1) \\
	M   &\colonequals  L \cap \pi^* \HdR^1(X) \\
	M_0 &\colonequals  L \cap \pi^* \HH^0(X,\Om_X^1).
\end{align*}
\end{definition}

We immediately obtain the following.

\begin{lemma} \label{lemma:Ldecomposition}
Let $T$ be a vector space complement of $M \oplus W$ in $L$, and define $T_0 \colonequals T \cap \HH^0(Y,\Om_Y^1)$.  Let $M_1$ be a complement to $M_0$ in $M$, and $T_1$ be a complement to $T_0$ in $T$.  Then there is a decomposition of vector spaces
\begin{align}
L &= M \oplus T \oplus \bigoplus_{i=1}^m W_i= M_0 \oplus M_1 \oplus T_0 \oplus T_1 \oplus \bigoplus_{i=1}^m W_i  \label{eq:Ldecomp} \\
L_0 &= V(L)= M_0 \oplus T_0 \oplus \bigoplus_{i=1}^m W_{i,\frac{d_i-1}{2}} = M_0 \oplus T_0 \oplus V\left(\bigoplus_{i=1}^m W_i \right). \label{eq:L0decomp}
\end{align}

We have $\dim M_0 = \dim T_0 = l_X$ and $\dim M = \dim T = 2 l_X$.
The spaces $M_0, M, M_0 \oplus V\left(\bigoplus_{i=1}^m W_i \right)$, and $M \oplus \bigoplus_{i=1}^m W_i $ are stable under $V$. 

Finally, %
$T_0$ may be chosen so that $\dim_k (\pi_{T_0}(V^n(T_0))) = l_X - a_X^n$. 
\end{lemma}

Here $\pi_U : L \to U  $ denotes the projection to $U$ for any $U$ occurring in these decompositions.

\begin{proof}
The claims about the vector space decompositions and dimensions are elementary.  Note that $M$ and $M_0$, being pullbacks of enhanced differentials on the base, are certainly preserved by $V$.  The other two spaces are stable under $V$ by Proposition~\ref{prop:nilpotent wij}.

For the final statement about $T_0$, note that the trace map $\pi_* :\HH^0(Y,\Omega_Y^1) \to \HH^0(X, \Omega_X^1)$ is surjective. Hence, given any $\omega_1 \in \HH^0(X,\Omega_X^1)$, let $\eta$ be an element of $\HH^0(Y,\Omega_Y^1)$ whose trace is $\omega_1$.
Note that the trace map commutes with $V$. To see this, let
$y^2+y=\psi$ be the equation defining our $\Z/2\Z$-cover
and note that $\eta=\omega_0 + \omega_1 y$ where $\omega_0$
is a meromorphic differential on $X$. Then $V(\pi_*(\eta))=V(\omega_1)$. On the other hand, we have
$V(\omega_1 y) = V(\omega_1)y + V(\omega_1\psi)$, so
$\pi_*(V(\eta))=V(\omega_1)$. It follows that $\pi_{T_0}(V(\eta))=V(\omega_1)$ and thus the projection of $V^n(T_0)$ to $T_0$ has dimension $l_X - a^n_X$.

\end{proof}

Next we study the residue pairing on $L$.  As the pairing of regular differentials is zero, note
\begin{equation} \label{equation:regularperp}
L_0 = V(L) \subseteq L_0^\perp.
\end{equation}

\begin{lemma} \label{lemma: M_0 orthogonal to W}
$M_0$ pairs trivially with $W$.
\end{lemma}

\begin{proof}
Given $\eta \in \HH^0(X,\Om_X^1)$ and $\widetilde{\omega}_{i,j} \in W_i$, note
$$\langle \pi^*\eta , \widetilde{\omega}_{i,j} \rangle = \text{Res}_{Q_i} (u_i^{-2j+d_i} \pi^*(\eta)).$$
However, since $\eta$ is regular at $P_i$, Lemma~\ref{l: uniformizer up above} implies $\ord_{Q_i} (\pi^* (\eta)) \geq d_i+1$. Furthermore, since $j \leq d_i-1$, we obtain $\ord_{Q_i}(u_i^{-2j+d_i} \pi^*(\eta)) \geq 3$ and the residue vanishes, as desired.
\end{proof}

\begin{lemma} \label{lemma: M orthogonal to W_V}
Assume $\sum_{i=1}^m (d_i+1) > 4g_X-4$. Then $M$ pairs trivially with $M_0$ and $W \cap L_0$.  
\end{lemma}
\begin{proof}
Recall $M = \pi^* \HdR^1(X)$, so any class in $M$ can be represented as 
\[
\xi = (\pi^* (\omega), (\pi^*(f_1), \ldots ,\pi^*(f_m))),
\]
where $f_i$ is the tail of a local function at $P_i$.  If $\sum_{i=1}^m (d_i+1)/2 > 2g_X -2$, then for any $n > (d_i+1)/2$ the Riemann-Roch theorem gives a function $h$ on $X$ which is regular away from $\{P_1,\ldots, P_m\}$ with a pole of order exactly $n$ at $P_i$ and a pole of order at most $(d_j+1)/2$ at $P_j$ when $j \neq i$.  By modifying the de Rham class representative using these functions (i.e. replacing $\xi$ with $\xi + dh$), we may assume that
$$\ord_{Q_i}(\pi^*(f_i)) = 2\ord_{P_i}(f_i) \geq - (d_i+1)$$ for every $i$. By Corollary~\ref{cor: t^j}, we may write 
\begin{equation}
f_i = \sum_{l=\frac{d_i+1}{2}}^0  c_l t_i^{-l}  + O(1)\quad \text{and} \quad \pi^*(f_i) = \sum_{l=-\frac{d_i+1}{2}}^0 \left(c_lu_i^{2l} + O\left(u_i^{2l+d_i}\right) \right). \label{eq: pi^*(f_i)}
\end{equation}
In particular, $\ord_{Q_i}(f_i) \geq - (d_i+1)$.

If $\eta\in \HH^0(X,\Om_X^1)$ then Lemma~\ref{l: uniformizer up above} and the bound on the order show that $\pi^*(f_i \eta)$ is regular at $Q_i$ and hence 
$$\langle \pi^*(\eta), \xi \rangle = \sum_{i=1}^m \Res_{Q_i}(\pi^*(f_i) \pi^*(\eta)) = 0.$$
Hence $M$ pairs trivially with $M_0$.  

We now show that any $\widetilde{\omega}_{i,j} \in W \cap L_0$ pairs trivially with $\xi$.  (Necessarily $j \leq (d_i-1)/2$.) Note
$$\langle \widetilde{\omega}_{i,j}, \xi \rangle = \sum_{i'=1}^m \Res_{Q_{i'}} (\pi^*(f_{i'}) \omega_{i,j}).$$
The terms with $i' \neq i$ are zero, as $\pi^*(f_{i'}) \omega_{i,j}$ is regular at $Q_{i'}$ by Equation~\eqref{eq: pi^*(f_i)} and Lemma~\ref{l: uniformizer up above}.
Combining the local expansion $\omega_{i,j} = \left[u_i^{-2j+d_i-1} + O\left(u_i^{d_i+1}\right)\right ]du_i$ at $Q_i$ from Lemma~\ref{l: construct wij}(ii) with Equation~\eqref{eq: pi^*(f_i)} gives
$$\pi^*(f_i) \omega_{i,j} = \sum_{l=\frac{d_i+1}{2}}^0 \left(c_lu_i^{-2l} + O\left(u_i^{-2l+d_i}\right) \right)\left(u_i^{-2j+d_i-1} + O\left(u_i^{d_i+1}\right)\right)du_i, $$
which does not have a $u_i^{-1}$ term since $d_i$ is odd. Therefore the residue vanishes, so $M$ pairs trivially with a basis for $W \cap L_0$.
\end{proof}

\section{Ordinary Base Curves} \label{sec:ordinary base}

As before, let $\pi: Y \to X$ be a double cover of smooth, proper, connected curves over an algebraically closed field $k$ of characteristic $2$. In this section, we make the additional assumption that the base curve $X$ is ordinary. We use Section~\ref{section: omegaij} to completely describe the Ekedahl-Oort type of $Y$, proving \cite[Conjecture 8.1.1]{StevenPhD}.
\subsection{The action of \texorpdfstring{$F$}{F} and \texorpdfstring{$V$}{V}} \label{ss:fandv}
When $X$ is ordinary, we have by definition $g_X=f_X$ and hence $l_X=0$. As in Remark~\ref{remark: W subspace of L}, we know $W = \bigoplus_{i=1}^m W_i$ is a subspace of $L$ of codimension $4l_X$, so $W=L$.  This gives a decomposition 
\begin{equation} \label{eq: k[F,V] decomposition ordinary rough}
\HdR^1(Y) = U \oplus Z \oplus W
\end{equation}
that is compatible with $F$ and $V$.  
Recall the notation from Definition~\ref{defn:wijtilde}.

\begin{lemma} \label{lemma:V ordinary}
When $X$ is ordinary, we have $V(\widetilde{\omega}_{i,j}) = \begin{cases}
    \widetilde{\omega}_{i,j/2} &\hbox{if $j$ is even} \\
    0 &\hbox{if $j$ is odd}.
\end{cases}$
\end{lemma}
\begin{proof}
This follows immediately from Proposition~\ref{prop:nilpotent wij} and the observation that, since $X$ is ordinary, there are no nonzero elements of $\HH^0(X,\Om_X^1)$ on which $V$ acts nilpotently.
\end{proof}

This fact allows us to decompose $\HdR^1(Y)$ further.

\begin{lemma} \label{lemma: decomposition fine}
When $X$ is ordinary, we have a direct sum decomposition (compatible with $F$ and $V$)
\begin{equation} \label{eq: k[F,V] decomposition ordinary fine}
\HdR^1(Y) = U \oplus Z \oplus \bigoplus_{i=1}^m W_i.
\end{equation}
\end{lemma}

\begin{proof}
Appealing to Equation~\eqref{eq: k[F,V] decomposition ordinary rough}, we need to show that $F$ and $V$ preserve each $W_i$. Lemma~\ref{lemma:V ordinary} shows that $V$ preserves $W_i$, so it remains to be shown that $F$ preserves $W_i$. Equation~\eqref{eq: k[F,V] decomposition ordinary rough} gives that $F(W_i) \subseteq W$, so it suffices to check that the projection of $F(\omega_{i,j})$ to $W_{i'}$ is zero when $i'\neq i$. For this, let $j'$ be arbitrary. Then using  Proposition~\ref{proposition: symplectic pairing} we have 
$$\langle \widetilde{\omega}_{i',j'} , F(\widetilde{\omega}_{i,j}) \rangle  = \langle V(\widetilde{\omega}_{i',j'} ), \widetilde{\omega}_{i,j}\rangle^p =0,$$
since $V(\widetilde{\omega}_{i',j'})$ is a regular differential in $W_{i'}$ by Lemma~\ref{lemma:V ordinary}. Since $F(\widetilde{\omega}_{i,j})$ pairs trivially with each $\widetilde{\omega}_{i',j'}$ when $i \neq i'$ and the pairing is non-degenerate when restricted to $W_{i'}$, it follows that the projection of $F(\widetilde{\omega}_{i,j})$ to $W_{i'}$ is zero. 
\end{proof}

Next we describe the action of $F$ on the $S$-enhanced differentials $\widetilde{\omega}_{i,j}$.

\begin{lemma} \label{lemma: F ordinary}
When $X$ is ordinary, we have $F(\widetilde{\omega}_{i,j})= \begin{cases}
        0 &\hbox{if $1 \leq j \leq \frac{d_i-1}{2}$} \\
        \widetilde{\omega}_{i,2j-d_i} &\hbox{if $ \frac{d_i-1}{2} < j \leq d_i-1$} .
    \end{cases}$
\end{lemma}
\begin{proof}
The first case is immediate since $F$ kills regular differentials. By Lemma~\ref{lemma: decomposition fine}, $F(\widetilde{\omega}_{i,j})$ lies in $W_i$. Moreover, for any $1 \leq j \leq d_i-1$, Lemma~\ref{lemma:V ordinary} and Proposition~\ref{proposition: symplectic pairing} yield
$$ \langle \widetilde{\omega}_{i,j'}, F(\widetilde{\omega}_{i,j}) \rangle 
= \langle V(\widetilde{\omega}_{i,j'}), \widetilde{\omega}_{i,j} \rangle ^p = \begin{cases}
     1 & \hbox{if $\frac{j'}{2} + j =d_i$} \\
     0 & \hbox{otherwise}.
\end{cases}$$
Rearranging the condition gives $j'=2d_i-2j$, so by Proposition~\ref{proposition: symplectic pairing} we obtain
\[F(\widetilde{\omega}_{i,j})  = \widetilde{\omega}_{i,2j-d_i}. \qedhere\]
\end{proof}

The isomorphism class of a $p$-torsion group scheme is its Ekedahl-Oort type, and can be studied via Dieudonn\'{e} theory \cite{DemGroups}.  For a $p$-torsion group scheme $G$, denote by $D(G)$ its Dieudonn\'{e} module.  Then Oda \cite{Oda} showed that $D(\Jac(Y)[p]) = \HdR^1(Y) $.  Thus, the Ekedahl-Oort type is determined by the actions of $F$ and $V$ on $\HdR^1(Y)$.

\subsection{Final types} \label{ss:final types} 
Recall that the Ekedahl-Oort type of a curve $C$ is the isomorphism class of the polarized mod-$p$ Dieudonn\'{e} module $\HdR^1(C)$, or equivalently the isomorphism class of the group scheme $\Jac(C)[p]$.
The final type is a combinatorial way to encode the Ekedahl-Oort type, which we will now recall. See Pries~\cite{pries_guide} for additional information.

Let $N$ be a polarized mod-$p$ Dieudonn\'{e} module (for instance $N=\HdR^1(C)$) with dimension $2g$ as a $k$-vector space.  A \emph{final filtration} is a filtration 
$$0 \subset N_1 \subset \cdots \subset N_{g} = V(N) \subset N_{g+1} \subset \ldots \subset N_{2g} = N$$
that is stable under $V$ and $\perp$ (symplectic complement) satisfying $\dim_k (N_l) = l$. By \cite[Lemma 5.2]{OortStrat}, one may require instead that the filtration is stable under $V$ and $F^{-1}$, giving an equivalent definition. The \emph{final type} is a string of integers defined by $\nu_l = \dim_k (V(N_l))$. Through the duality coming from the polarization, the first $g$ elements of the final type determine the entire final type. Thus we declare the final type $\nu$ of $N$ to be
$$\nu = [ \nu_1, \cdots ,\nu_{g} ].$$
Oort proves that two polarized mod-$p$ Dieudonn\'{e} modules are isomorphic if and only if they have the same final type \cite{OortStrat}.   

\begin{definition}
A \emph{simple word} in $V$ and $\perp$ is a string of $V$'s and $\perp$'s that ends with $V$ and does not contain consecutive copies of $\perp$.  More concretely, simple words are of the form $V^{n_t} \bot V^{n_{t-1}} \ldots V^{n_2} \bot V^{n_1}$.
\end{definition}

A simple word $w = V^{n_t} \bot V^{n_{t-1}} \ldots V^{n_2} \bot V^{n_1}$ acts on subspaces $L \subset N$ by
\[
w(L) = V^{n_t}(\cdots (V^{n_2}((V^{n_1}(L))^\perp)^\perp) \cdots)
\]
This is convenient for studying final types and the canonical filtration as knowing $\dim_k w(N)$ for all simple words $w$ determines the final type.

\subsection{EO Types for Covers of Ordinary Curves}  We now determine the Ekedahl-Oort type of $Y$. 

\begin{lemma}\label{lemma: final type of Wi}
The final type of $W_i$ is given by $\left[0, 1, 1, 2, 2, \ldots , \floor{\frac{d_i-1}{4}}\right]$, i.e. $\nu_l = \left\lfloor l/2 \right \rfloor$ for $1 \leq l \leq \frac{d_i-1}{2}$.
\end{lemma}

\begin{proof}
Recall the subspaces $W_{i,l} = \spa \{ \widetilde{\omega}_{i,j} \; | \; j \leq l\}$ from Definition~\ref{defn:wijtilde}. Note that $\dim_k W_{i,l} = l$ by construction. Lemma~\ref{lemma:V ordinary} and Lemma~\ref{lemma: F ordinary} establish that the filtration 
$$0 \subset W_{i,1} \subset \cdots \subset W_{i,d_i-1} = W_i$$
is a final filtration. 
Finally, Lemma~\ref{lemma:V ordinary} yields $V(N_{i,l})=N_{i,\lfloor l/2 \rfloor}$, giving the desired result.
\end{proof}

For any odd positive integer $d$, let $G_d$ be the $p$-torsion group scheme of length $p^{d-1}$ with final type $\left[0, 1, 1, 2, 2, \ldots , \floor{\frac{d_i-1}{4}}\right]$, so that the Dieudonn\'{e} module of $G_{d_i}$ is isomorphic $W_i$ by Lemma~\ref{lemma: final type of Wi}.  As in the introduction, let $M_{\ord}$ be the mod-$p$ Dieudonn\'{e} module of an ordinary elliptic curve.

\begin{theorem} \label{thm: ordinary EO type}
Let $\pi:Y \to X$ be a double cover of smooth, proper, geometrically connected curves in characteristic $2$, with ramification breaks $d_1,\ldots,d_m$ at the ramified points. Assume $X$ is ordinary and write $f_Y=2g_X-1+m$ for the $p$-rank of $Y$. Then we have 
\[
\HdR^1(Y) \simeq M_{\ord}^{f_Y} \oplus \bigoplus_{i=1}^m W_i.
\]
Equivalently, we have
\[ \Jac(Y)[p] \cong (\underline{\mathbb{Z}/2\mathbb{Z}} \oplus \mu_2)^{f_Y} \oplus \bigoplus_{i=1}^m G_{d_i}.\]
\end{theorem}

\begin{proof}
Using Lemma~\ref{lemma: decomposition fine}, we get  isomorphisms of Dieudonn\'{e} modules
$$D(\Jac(Y)[p]) = \HdR^1(Y) = U \oplus Z \oplus \bigoplus_{i=1}^m W_i.$$
Recalling Lemma~\ref{lemma: decomposition k[F,V]},
we observe that $D((\underline{\mathbb{Z}/2\mathbb{Z}})^{f_Y}) \cong U$, $D(\mu_2^{f_Y}) \cong Z$, and $D(G_{d_i})=W_i$. 
\end{proof}

\begin{remark}
Note that this establishes Theorem~\ref{thmA} in the introduction.  %
\end{remark}

We now work out an equivalent form of Theorem~\ref{thm: ordinary EO type} concerning the final filtration.

\begin{definition} \label{def: phi}
Define a function $\phi$ taking two or more integer arguments recursively as follows:
\begin{align*}
    \phi(d;n_1)&\colonequals \left \lfloor \frac{d-1}{2^{n_1}} \right \rfloor  \\
    \phi(d;n_1, \ldots ,n_{t+1}) &\colonequals  \left \lfloor \frac{d-1-\phi(d;n_1,\ldots ,n_t)}{2^{n_{t+1}}}  \right \rfloor.
\end{align*}
For a simple word $w = V^{n_t} \bot \ldots \bot V^{n_1}$, define $\phi(d,w)=\phi(d;n_1,\ldots ,n_t)$. 

When a cover with ramification breaks $d_1,\ldots,d_m$ is clear from context, we set $\phi(w):= \sum_{i=1}^m \phi(d_i,w)$. %
\end{definition}

Note that taking $n_{t+1}=0$ gives that if $w = \bot V^{n_t} \bot \ldots \bot V^{n_1} = \bot w'$ then
\begin{equation}
\phi(d,w) = \left \lfloor d- 1 - \phi(d,n_1,\ldots,n_t) \right \rfloor = d-1-\phi(d,w').
\end{equation}

\begin{lemma} \label{lemma: phi ordinary}
When $X$ is ordinary, for any simple word $w$ we have $w (U \oplus W_i) = U \oplus W_{i,\phi(d_i,w)}$ and 
$$\dim_k (w(\HdR^1(Y))) = f_Y + \phi(w).$$
\end{lemma}

\begin{proof}
Recall that $V$ is bijective on $U$, that $V(Z)=0$, and that $U^\perp = U \oplus W$ while $W^\perp = U \oplus Z$.
By Lemma~\ref{lemma:V ordinary}, $V^n(W_{i,l})=W_{i,\left \lfloor l/2^n \right \rfloor}$ and by Proposition~\ref{proposition: symplectic pairing} $W_{i,l}^\perp \cap W = W_{i,d_i-1-l} \oplus \bigoplus_{i'\neq i} W_{i'}$.  This establishes the first claim.
Lemma~\ref{lemma: decomposition fine} gives a decomposition compatible with $V$, so induction gives  
$$w(\HdR^1(Y)) = U \oplus \bigoplus_{i=1}^m W_{\phi(d_i,w)}.$$
The second result follows as $\dim_k(W_{i,\phi(d_i,w)}) = \phi(d_i,w)$, by definition $\phi(w)=\sum_{i=1}^m \phi(d_i,w)$ and finally $f_Y = \dim_k U$. 
\end{proof}

\begin{corollary} \label{cor: nu ordinary}
When $X$ is ordinary, the final type of $Y$ is determined by the condition that 
$$\nu_{f_Y + \phi(w)} = f_Y + \phi(Vw) = f_Y + \sum_{i=1}^m \left \lfloor \frac{\phi(d_i,w)}{2}   \right \rfloor$$
for each simple word $w$.
\end{corollary}

\begin{proof}
The Ekedahl-Oort type of $Y$ is determined by the final type $\nu_l = \dim_k (V(N_l))$, where $\{N_l\}$ is any final filtration. Any final filtration is a refinement of the \emph{canonical filtration} $\{ w(\HdR^1(Y))\}$, where $w$ ranges over all words in $V$ and $\perp$. The final type is determined by the values $\dim_k (w(\HdR^1(Y)))$, which are given in Lemma~\ref{lemma: phi ordinary}.
\end{proof}

\begin{remark} \label{rmk:family1}
We can now prove the first part of Corollary~\ref{cor:families}: consider covers of $\PP^1$ ramified at infinity with ramification break $d$.  The genus of the cover is $g=(d-1)/2$ by Riemann-Hurwitz, and the final type is $[0,1,1,2,2,\ldots,\floor{g/2}]$ by Theorem~\ref{thm: ordinary EO type}. The dimension of the family can be recovered from \cite{prieszhu12}, or seen more directly as follows.  By Artin-Schreier theory, any such Artin-Schreier cover can be written in the form $y^2 + y = f(x)$ where $f(x)$ is a polynomial of degree $d$, and the extension of function fields is determined by $f(x)$ up to adding something of the form $g^2 + g$.  Thus to obtain non-isomorphic extensions we may just restrict to polynomials $f(x)$ for which the coefficient of $x^{2i}$ is zero for any integer $i$.  There are $(d+1)/2$ unconstrained coefficients, giving a $(d+1)/2$-dimensional family of Artin-Schreier covers.  However, two such covers of $\PP^1$ may be isomorphic without being isomorphic as covers of $\PP^1$.  As there is a two-dimensional family of automorphisms of $\PP^1$ which fix infinity, we obtain a $(d+1)/2-2$-dimensional family of Artin-Schreier curves.
\end{remark}

\begin{remark}\label{remark:unlikely}
By \cite[Theorem 1.2]{OortStrat},   the Ekedahl-Oort stratum with final type $[\nu_1, \ldots, \nu_g]$ has codimension $\sum_{i=1}^g (i-\nu_i)$ inside $\mathcal{A}_g$. Consider now an ordinary curve $X$ and a double cover $\pi: Y \to X$ that is branched at one point with ramification invariant $d$.  Then $g_Y = 2g_X+ (d-1)/2$.
By Theorem~\ref{thm: ordinary EO type}, the final type of $Y$ is given by
$$\left[1, 2, \ldots, g_X, g_X, g_X+1, g_X+1, \ldots, g_X+\left \lfloor \frac{d-1}{4} \right \rfloor \right].$$
The Ekedahl-Oort stratum of $\Jac(Y)$ has codimension 
$$\sum_{i=1}^{\frac{d-1}{2}} \left(i-\left\lfloor \frac{i}{2} \right\rfloor\right) = \sum_{i=1}^{\frac{d-1}{2}} \left \lceil \frac{i}{2} \right \rceil \geq \frac{(d-1)(d+1)}{16}.$$

$$3g_{Y}-3= 3\left(2g_X+ \frac{d-1}{2} \right)-3.$$
Thus, when $d$ is sufficiently large compared to $g_X$, the codimension of the Ekedahl-Oort stratum of $\Jac(Y)$ will exceed the dimension of the Torelli locus and therefore the intersection between the Torelli locus and the Ekedahl-Oort stratum of $\Jac(Y)$ is an unlikely one.
\end{remark}

\section{Examples of \texorpdfstring{$\Z/2\Z$}{Z/2Z}-covers of a supersingular
elliptic curve}\label{sec: non ordinary base}
In this section we explore some examples of degree two covers
$\pi:Y \to X$ when $X$ is the unique supersingular elliptic curve
in characteristic two to illustrate that the Ekedahl-Oort type of $Y$ is not determined by the Ekedahl-Oort type of $X$ and the ramification of $\pi$ (also see Example \ref{ex2} for
a higher genus example). This is the simplest non-ordinary example
and we observe a variety of Ekedahl-Oort types for covers with the same ramification.  In this section we restrict to covers ramified at a single point as this already illustrates the essential behavior and avoids another layer of cumbersome notation.

In Section~\ref{ss: initial examples}
we give some concrete examples found computationally.   In
Section~\ref{ss: some covers of a supersingular base curve}
we determine all possible $k[V]$-structures on $H^0(Y,\Omega^1_Y)$ occurring
when $\pi$ is ramified
over a single point. We also determine the codimension of
each $k[V]$-module stratum in the moduli space of such covers.
Finally, in Section~\ref{ss: EO type of ss covers} we investigate the Ekedahl-Oort type of $Y$ when $\pi$
is ramified over a single point. In Example \ref{ex:theoryd7} we revisit and conceptually understand the examples from Section~\ref{ss: initial examples} with a supersingular base curve.  For example, we determine the two possible Ekedahl-Oort types when the unique ramification break is $d=7$. We also prove in
Theorem \ref{theorem: d = 2^n+1} that
there is only one possible Ekedahl-Oort type when $d=2^k+1$.  But in general there are many possible Ekedahl-Oort types and we do not attempt to completely classify them as we expect the answer to be a combinatorial mess.  Instead, we will bound the Ekedahl-Oort type in Section~\ref{sec:bounds}.

\subsection{Initial Examples} \label{ss: initial examples}

We begin with some concrete examples.

\begin{example} \label{example: different EO types}
Let $X$ be the supersingular elliptic curve $y^2 + y = x^3 $ over $k$.  %
Fix a point $Q$ on $X$.  The Ekedahl-Oort type of the double cover defined by $z^2 + z = f(x,y)$, where $f(x,y)$ is regular except at $Q$ where it has a pole of order $d$, is not always determined by $d$.  Due to automorphisms of $X$, the behavior is independent of $Q$.

In particular, if we take $d=7$ and $f(x,y) = x^2 y$ then the cover $Y$ has final type $[ 0, 1, 1, 2, 3]$.
But taking $f(x,y) = (x^2+x+1)y$ the cover has final type $[ 0, 1, 2, 2, 3]$. %
These were the only two Ekedahl-Oort types found through a computational search.  In Example~\ref{ex:theoryd7} we will analyze this example, prove these are the only possible Ekedahl-Oort types, and prove the latter is generic.
\end{example}

\begin{example} \label{ex:d15}
Again let $X$ be the supersingular elliptic curve $y^2 + y = x^3 $.  Looking at covers ramified at one point with ramification invariant $d=15$, a computational search finds at least five different Ekedahl-Oort types.  Due to their number we will not list them completely, but simply highlight one interesting feature: we find covers with different final types (like $[0, 1, 2, 2, 3, 4, 4, 4, 5]$ and $[0, 1, 2, 2, 3, 3, 3, 4, 5]$)
but with the same $k[V]$-module structure for the regular differentials (in this case $k[V]/(V^5) \oplus k[V]/(V^2) \oplus k[V]/(V) \oplus k[V]/(V)$).  
In other words, this is an example where the higher $a$-numbers (dimension of the kernel of the powers of the Cartier operator on the regular differentials) do not determine the Ekedahl-Oort type. 
\end{example}

\begin{example} \label{ex2}
Let $X$ be the genus three curve given by the affine equation $y^2 + y =x^7-x$.  It is neither ordinary nor superspecial (it has final type $[ 0, 1, 1]$). 
Searching with a computer, Table~\ref{table:coversofgenus3} gives examples of covers ramified over one point with ramification invariant $d=7$ with decreasing frequency.  Note the Ekedahl-Oort type of the particularly simple cover $z^2 + z = y$ was not found via a random search, but instead chosen for its simple form: the corresponding stratum has large codimension.

\begin{table}[htb]

    \caption{Ekedahl-Oort types of covers given by $z^2 + z =f(x,y)$ with ramification break $d=7$, listed with decreasing frequency}

\begin{tabular}{|c|Sc|}
    \hline 
    Ekedahl-Oort Type    &   $ f(x,y)$ \\ [0.5ex] 
    \hline \hline 
    $[ 0, 1, 2, 3, 3, 4, 5, 6, 7]$ & $\displaystyle \frac{x^2 y + (x^4 + x^3 + x^2 + x + 1)}{x^7 + x^6 + x^5 + x^4 + x^3 + x^2 + x + 1} $ \\
        \hline 
    $[ 0, 1, 1, 2, 3, 4, 5, 6, 7]$ & $\displaystyle \frac{y + (x^5 + x^4 + x^3)}{x^7 + x^6 + x^5 + x^4 + x^3 + x^2 + x + 1} $ \\
    \hline 
    $[ 0, 0, 1, 2, 3, 4, 5, 6, 7]$ & $\displaystyle \frac{x y + (x^6 + x^5 + x^4)}{x^7 + x^6 + x^5 + x^4 + x^3 + x^2 + x + 1} $ \\
    \hline 
    $[ 0, 1, 1, 2, 2, 3, 4, 4, 5 ]$ & $y$\\
    \hline 
\end{tabular}
\label{table:coversofgenus3}
\end{table}
\end{example}

\subsection{The \texorpdfstring{$k[V]$}{k[V]}-structure of \texorpdfstring{$\Z/2\Z$}{Z/2Z}-covers
of the supersingular elliptic curve}
\label{ss: some covers of a supersingular base curve}
We again take $X$ to be the superspecial (i.e. supersingular) elliptic curve given by $y^2 + y = x^3$. 
Let $\pi: Y \to X$ be a double cover ramified over $S=\{Q\}$ with ramification break $d$. We suppose $Y$ is given by an Artin-Schreier equation $z^2 - z = \psi$ for some function $\psi$ on $X$ regular except at infinity and with $\ord_\infty(\psi)=d$. This is always possible
if $d>1$ by Riemann-Roch.

\subsubsection{The basics of $V$ acting on $\HdR^1(Y)$}
It is straightforward to compute that the two-dimensional $\HdR^1(X)$ is spanned by the enhanced differentials 
\begin{equation}
    \beta_{1} \colonequals (dx,0) \quad \text{and} \quad \beta_2 \colonequals (x dx,y/x):
\end{equation}
note that $d(y/x) = x dx + y/x^2 dx$.  Note that $t\colonequals y/x^2 + 1/x^2$ is a uniformizer at the infinite point $Q \in X$. To see this, we first remark that $x$ has a pole of order $2$ at $Q$ and 
$y$ has a pole of order $3$. Thus, $y/x^2$ has a zero of order one and $1/x^2$ has a zero of order four, so that $t$ has a zero of order one. We remark that
$dt = dx$ and  $x = t^{-2 } + O(1)$.

\begin{propositiondefn} \label{prop:Valpha}
There exists a choice $\{\widetilde{\omega}_{1,j}\}_{j=1, \ldots, d-1}$ of $S$-enhanced differentials on $Y$ such that $V(\widetilde{\omega}_{1,1}) = \beta_1 = (dx,0)$ and for $2\leq j \leq d-1$
\[
V(\widetilde{\omega}_{1,j}) = \begin{cases}
0 & \text{if } j \text{ is odd},\\
\widetilde{\omega}_{1,j/2} & \text{if } j \text{ is even}.
\end{cases}
\] 
For ease of notation, in this section we let $\widetilde{\omega}_i \colonequals \widetilde{\omega}_{1,i}$ for $1 \leq i \leq d-1$.  Thus $V(\widetilde{\omega}_i) = \widetilde{\omega}_{i/2}$ if $i$ is even, $V(\widetilde{\omega}_1) = \beta_1=(dx,0)$, and $V(\widetilde{\omega}_i)=0$ otherwise.
\end{propositiondefn}

\begin{proof}
Proposition~\ref{prop:nilpotent wij} almost gives the desired behavior, except the images under $V$ may include the pullback of a differential regular on $X$, i.e. a scalar multiple of $\beta_1 = (dx,0)$. 
First note that $\widetilde{\omega}_1$ may be chosen to be the regular differential $(xdx,0)$, which has local expansion $(t^{-2} + O(1) )dt$.  Then $V(\widetilde{\omega}_1) = \beta_1 = (dx,0)$ as desired.  
 Since $V(\beta_1)=0$  we may freely modify the other $\widetilde{\omega}_{i,j}$ by adding a multiple of $dx$ without changing the image under $V$ to arrange the desired formulas.
\end{proof}

We have two immediate corollaries of Proposition/Definition \ref{prop:Valpha}. The first is about the subspaces of $\HdR^1(Y)$ defined by
\begin{equation} \label{eq:Rn}
    R_n\coloneq \mathrm{span}_k(\beta_1,\widetilde{\omega}_1,\dots, \widetilde{\omega}_n) \quad \text{and} \quad R_0 = \mathrm{span}_k(\beta_1).
\end{equation}

\begin{corollary} \label{c: V behavior on X with poles}
    Let $r \geq 0$ satisfy $2n \geq 2^r$. Then
    \begin{align*}
        V^r \left( R_n \right) &= R_{\lfloor \frac{n}{2^r} \rfloor}.
    \end{align*}
\end{corollary}

\begin{definition}
    Let $\omega \in \HdR^1(Y)$. We define the $V$-order
    of $\omega$ to be
    \begin{equation*}
        \ord_V(\omega) = \min \{r : V^r(\omega)=0 \}. 
    \end{equation*}
\end{definition}

\begin{corollary} \label{c: alpha order label}
Writing $n = 2^k m$ with $m$ odd, we have    \begin{align*}
        \ord_V(\widetilde{\omega}_n) &=\begin{cases}
            k+2 & \text{if } m=1 \\
            k+1 & \text{if } m >1.
        \end{cases}
    \end{align*}
\end{corollary}

    \subsubsection{Generalities about $k[V]$-modules}
    Let $k[V]$ denote the `skew' polynomial ring where we have the relation
    $Vc^p=cV$ for any $c \in K$. 
	Let $M$ be a $k[V]$-module that is finite dimensional as a $k$-vector space (i.e. $H^0(Y,\Omega^1_Y)$).
	There is a decomposition 
	\begin{align*}
		M&=M^{\bij} \oplus M^{\nil},\\
		M^{\bij}&=(k[V]/(V-1))^{p(M)},\\
		M^{\nil} &= \bigoplus_{i\geq 0 } (k[V]/V^i)^{b_i(M)},
	\end{align*}
    where $p(M), b_i(M) \in \Z_{\geq 0}$ and almost all the $b_i(M)$
    are zero. $p(M)$ is the $p$-rank of $M$.
    Next, 
	we define the $r$-th higher $a$-numbers of $M$:
	\begin{equation*}
		a^r(M) \colonequals \dim(\ker(V^r|_M)).
	\end{equation*}
	Note that
	\begin{equation*}
		a^r(M) = \sum_{i\geq 0} b_i(M) \cdot \min(r,i),
	\end{equation*}
	so the numbers $a^r(M)$ determine the numbers $b_i(M)$ and
	vice versa. In particular, the $p$-rank of $M$ and the
	higher $a$-numbers completely determine the class of $M$,
    and thus completely determine $M$. 

    \begin{definition}
        The $V$-type of $M$ is a sequence of nonincreasing positive integers
        $\iota(M)=(c_0,c_1, \dots)$ defined by
        \begin{equation*}
            c_i=c_i(M) = \dim(V^i(M)) = \dim(M) - a^i(M).
        \end{equation*}
        Again, note $M$ is completely determined by its $V$-type.
        For any curve $Y$ we call $\iota(H^0(Y,\Omega_Y^1))$ the $V$-type
        of $Y$.
    \end{definition}

    An important class of examples come from spaces of differentials
	on the projective line with bounded pole. Let $n \geq 1$ and consider the $k[V]$-module
	\begin{equation*}
		M_n = H^0(\PP^1, \Omega^1_{\PP^1}((n+1)[0])).
	\end{equation*}

        \begin{lemma} \label{lemma:vtypemn}
    We have that $\dim(M_n) = n$, that $p(M_n) = 0$, and that $\displaystyle  a^r(M_n) = n- \left \lfloor \frac{n}{2^r} \right \rfloor$.
    In particular, the $V$-type of $M_n$ is
    \begin{equation*}
        \iota(M_n) = \left( \left \lfloor \frac{n}{2^i} \right \rfloor \right)_{i\geq 0}.
    \end{equation*}
    \end{lemma}
    
    \begin{proof}
    A basis for $M_n$ is given by $\{v_1,\ldots,v_n\}$ where $v_i = x^{-(i+1)} dx$.  Thus, it is clear that $V^r(M_n) = M_{\lfloor n/2^r\rfloor }$, which gives the $V$-type.
    \end{proof}

\subsubsection{Main result} To state our main result we need
    some notation.
    \begin{definition}
        Let $u(r)$ be the sequence consisting of $r+1$ ones following
        by all zeros, i.e., 
        \begin{equation*}
            u(r) = \left(\underbrace{1,\dots,1}_{r+1}, 0,0\dots \right).
        \end{equation*}
            For any $n \geq 1$ we define $r_n  \colonequals  \left\lfloor \log_2(n) \right \rfloor $.
    \end{definition}

        \begin{theorem}\label{t: k[V] structure on ss ell curve}
        For a cover $\pi : Y \to X$ as above, ramified over a single point with ramification break $d$, let $\delta \colonequals (d-1)/2$.  Then the $V$-type of $Y$ is
        \begin{equation*}
            \iota(H^0(Y,\Omega_Y^1))  = \iota(M_{\delta}) + u({r_\delta + 1}) + u({\mu_\psi}),
        \end{equation*}
        where either $\mu_\psi = r_{d-1}+1$ or $v_2(d-1) \leq \mu_\psi< r_{d-1}$.
        Furthermore, each possibility occurs for some $\pi: Y \to X$. 
    \end{theorem}
    To prove Theorem \ref{t: k[V] structure on ss ell curve} we need
    three short lemmas.  Recall that $R_n$ is defined in equation \eqref{eq:Rn}, and that $\dim(R_n)=n+1$. 
    
    \begin{lemma}
        \label{l: V structure of differentials with poles
        on ss curve}  We have that $p(R_n)=0$ and
        the $V$-type of $R_n$ is
        \begin{equation*}
            \iota(R_n)=\iota(M_n) + u({r_n+1}).
        \end{equation*}
    \end{lemma}
    \begin{proof}
        From Corollary \ref{c: V behavior on X with poles}, we know that for $r=0,\dots,r_n+1$ we have $V^r(R_n)=R_{\lfloor n/2^r \rfloor}$,
        so that $c_r(R_n)=\floor{ \frac{n}{2^r}} + 1$. For
        $r> r_n+1$ we have $V^r(R_n)=V(R_0)=0$. 
    \end{proof}

    Let $\omega_T \colonequals (z dx,0)$, which is easily seen to be a regular differential on $Y$.
    
    \begin{lemma}
        \label{l: structure of global differentials}
         We have $H^0(Y,\Omega_Y^1)=R_\delta \oplus k \cdot  \omega_T $.
    \end{lemma}
    \begin{proof}
        We know $H^0(Y,\Omega_Y^1)$ has dimension $\dim_k(R_\delta)+1 = \delta +2 = (d+3)/2$. Thus, it is enough to show $\omega_T \not \in R_\delta$. This follows by observing that $\omega_T$ is not fixed
        by $\Gal(Y/X)$, while $R_\delta$ is fixed by
        the Galois action.
    \end{proof}

    	\begin{lemma}\label{l: ss function and differential isomorphism}
		Let $\omega$ be a nowhere vanishing global differential on $X$. Then the function
		\[ s:H^0(X,\cO_X(dQ)) \to R_{d-1}=H^0(X, \Omega_X^1(dQ)), \quad f \mapsto fdx ,\]
		is an isomorphism. Furthermore, $pole(f)=pole(fdx)$. 
	\end{lemma}
    
	\begin{proof}
		This is an easy consequence of Riemann-Roch using that $X$ has genus one.
 	\end{proof}

    \begin{proof}[Proof Of Theorem \ref{t: k[V] structure on ss ell curve}]
        We know that $H^0(Y,\Omega_Y^1) \cong R_\delta \oplus k \cdot \omega_T$,
        so that
        \begin{equation*}
            V^r(H^0(Y,\Omega_Y^1))=V^r(R_\delta) + k \cdot V^r(\omega_T).
        \end{equation*}
        Thus, we see that $\iota(H^0(Y,\Omega_Y^1))= \iota(R_\delta) + u({\mu_{\psi}})$  where $\mu_\psi$ is the largest $r$ such that
        $V^r(\omega_T)\not\in V^r(R_\delta)$.
        We see 
        \[
        V(\omega_T)=V((z^2 +\psi) dx ) = V(\psi dx),
        \]
        so we are looking for the
        largest value $r$ such that $V^r(\psi dx)\not\in V^r(R_\delta)$.
        By Lemma \ref{l: ss function and differential isomorphism} we know
        \begin{equation*}
            \psi dx = c_{d-1}\widetilde{\omega}_{d-1} + \dots + c_1\widetilde{\omega}_1 + c_0 dx,
        \end{equation*}
        where $c_{d-1}\neq 0$. In particular, we have
        \begin{equation}\label{eq: definition of gamma}
            \psi dx \equiv \underbrace{c_{d-1} \widetilde{\omega}_{d-1} + \dots + c_{\delta+1}\widetilde{\omega}_{\delta+1}}_{\gamma} \mod R_{\delta}.
        \end{equation}
        From Proposition/Definition \ref{prop:Valpha} it is clear that
        $V^r(\psi dx)\in V^r(R_\delta)$ if and only if $V^r(\gamma)=0$. Thus,
        \begin{equation}\label{eq: max order V}
            \mu_\psi = \max \{\ord_V(c_i\widetilde{\omega}_i)\}_{i=\delta+1,\ldots,d-1} -1.
        \end{equation}
        Let $n$ be the largest power of two less or equal to $d-1$, i.e. $n = 2^{r_{d-1}}$.
        We consider two cases:
        \begin{enumerate}
            \item Suppose $c_{n} \neq 0$. Then
            $\ord_V(c_{n}\widetilde{\omega}_{n}) = r_{d-1} + 2$ from Corollary
            \ref{c: alpha order label}. We claim this is the unique maximal term in the right side of \eqref{eq: max order V}. 
            Consider $i\neq n = 2^{r_{d-1}}$ satisfying $d-1 \geq i  >\delta $.
            By the definition of $\delta$ and $r_{d-1}$, we know
            that $i$ is not a power of two. Thus, $i=2^km$ where
            $m >1$ and $k<r_{d-1}$. Then by Corollary \ref{c: alpha order label} we find $\ord_V(c_i\widetilde{\omega}_i) \leq k+1< r_{d-1}+1$. It follows
            that $\mu_\psi=r_{d-1}+1$. 

            \item Suppose $c_{n} = 0$. As in the
            previous case, the non-trivial terms in \eqref{eq: max order V} satisfy $\ord_V(c_i\widetilde{\omega}_i)<r_{d-1}+1$, . This gives
            an upper bound of $\mu_\psi < r_{d-1}$.
            Furthermore, since $c_{d-1} \neq 0$ the maximum in \eqref{eq: max order V} must be at least $v_2(d-1)+1$. Thus, we have
            $v_2(d-1)  \leq \mu_\psi < r_{d-1}$. To see that each possibility occurs,
            let $\mu$ satisfy $v_2(d-1) \leq \mu < r_{d-1}$. We can find $i$
            satisfying $d-1 \geq i >\delta $ with $v_2(i)=\mu$. 
            Then choose $\psi$ so that $\psi dx= \widetilde{\omega}_{d-1} + \widetilde{\omega}_i$.  \qedhere          
        \end{enumerate}
    \end{proof}
    
        We can describe the codimensions of the $k[V]$-module strata in the
        moduli space $\mathcal{M}_{X,d}$ of curves $Y$ admitting a $\Z/2\Z$-cover over $X$ branched at a
        single point with ramification break $d$. Let $Y \to X$ be
        a $\Z/2\Z$-cover ramified over a single point with ramification
        break $d$. We can assume the ramified point is at $Q$ by translating
        with the group law. Set
        \[ G_d \coloneqq H^0(X,dQ)-H^0(X,(d-1)Q),\]
        so that $G_d$ consists of rational functions on $X$ whose pole
        divisor is $dQ$. By Riemann-Roch, each cover is given by an equation $y^2 + y = \psi$ where $\psi\in G_d$. Let $M_\delta=\{ f^p+f\; | \; f \in H^0(X,\delta Q)\}$. Then $\psi$ and $\psi +h$ define
        the same $\Z/2\Z$-cover for any $h \in M_\delta$. 
        In particular, we may view $G_d/M_\delta$ as a parameter space of curves admitting a $\Z/2\Z$-cover to $X$.
        For any class $[\psi] \in G_d/M_\delta$, there is a unique representative
        such that
        \begin{equation*}
            \psi dx = c_{d-1}\widetilde{\omega}_{d-1} + \dots + c_1\widetilde{\omega}_1 + c_0 dx,
        \end{equation*}
        with $c_i=0$ for odd $i>1$, $c_0=0$, and $c_{d-1}\neq 0$. In particular,
        we see the parameter space $G_d/M_\delta$ is equal to $\mathbb{G}_m \times \mathbb{A}^{\delta}$ by sending $[\psi]$ to $(c_{d-1},c_{d-3},\dots, c_2,c_1)$.
        The moduli space $\mathcal{M}_{X,d}$ is a finite quotient of $G_d/M_\delta$, taking into account
        the automorphisms of $X$ that fix $Q$.
        
        From \eqref{eq: max order V} we see that the strata are defined by the vanishing
        and nonvanishing of $c_i$ for even $i$ in the range $\delta+1\leq i \leq d-1$.
        As in the proof of Theorem \ref{t: k[V] structure on ss ell curve}, we see that $\mu_\psi=r_{d-1}+1$ if and only if $c_{n} \neq 0$ for $n=2^{r_{d-1}}$.
        In particular, we see that $\mu_\psi=r_{d-1}+1$ occurs generically, and
        the complement has codimension one. For $\mu $ with $v_2(d-1) \leq \mu < r_{d-1}$, by \eqref{eq: max order V} and Lemma \ref{c: alpha order label} we see that $\mu_\psi = \mu$ if and only if 
        the following hold:
        \begin{enumerate}
            \item For any $i$ with $\delta+1\leq i \leq d-1$ satisfying $v_2(i)>\mu$ we
            have $c_i=0$. 
            \item For some $i$ with $\delta+1 \leq i \leq d-1$ with $v_2(i)=\mu$ we
            have $c_i \neq 0$.
        \end{enumerate}
        There are
        \[ \left\lfloor \frac{d-1}{2^{\mu+1}} \right \rfloor - \left\lfloor \frac{\delta}{2^{\mu+1}} \right \rfloor  = \left\lfloor \frac{d-1}{2^{\mu+1}} \right \rfloor - \left\lfloor \frac{d-1}{2^{\mu+2}} \right \rfloor. \]
        values of $i$ with $\delta+1\leq i \leq d-1$ and $v_2(i) > \mu$.
        To see this, note that there are $\left\lfloor \frac{d-1}{2^{\mu+1}} \right \rfloor+1$
        multiples of $2^{\mu+1}$ between $0$ and $d-1$.
        In particular, we see that the stratum corresponding to $\mu_\psi=\mu$ is
        irreducible and has codimension $\left\lfloor \frac{d-1}{2^{\mu+1}} \right \rfloor - \left\lfloor \frac{d-1}{2^{\mu+2}} \right \rfloor$.
        We summarize this discussion with the following theorem.
    \begin{theorem}
        \label{t: codimension of strata} Continue with the notation
        above. Then the following holds. 
        \begin{enumerate}
            \item A generic point in $\mathcal{M}_{X,d}$ 
            has the $V$-type determined by $\mu_\psi=r_{d-1}+1$.
            \item For $\mu$ with $v_2(d-1) \leq \mu < r_{d-1}$, the stratum
            of covers satisfying $\mu_\psi = \mu$ in $\mathcal{M}_{X,d}$ has codimension
            \[\left\lfloor \frac{d-1}{2^{\mu+1}} \right \rfloor - \left\lfloor \frac{d-1}{2^{\mu+2}} \right \rfloor. \]
        \end{enumerate}
    \end{theorem}

\subsection{Ekedahl-Oort types of \texorpdfstring{$\Z/2\Z$}{Z/2Z}-Covers of supersingular elliptic curves}
\label{ss: EO type of ss covers}

The genus of $Y$ is $(d+3)/2$ by Riemann-Hurwitz, so the span of $\beta_1, \beta_2$, $\widetilde{\omega}_1,\ldots, \widetilde{\omega}_{d-1}$, and $\omega_T = (z dx,0)$ has codimension one in $\HdR^1(Y)$.  We pick an arbitrary non-regular $\omega_T'$ which completes the basis.  (The exact form of $\omega_T'$ is irrelevant.)  As we do not have complete information about the symplectic pairing, we will identify $\HdR^1(Y)$ with its dual non-canonically using this chosen basis.

\begin{notation}
For $\omega \in \HdR^1(Y)$, let $\omega^*$ be the linear functional corresponding to $\omega$ upon identifying $\HdR^1(Y)$ with its dual using the chosen basis $\beta_1,\beta_2,\widetilde{\omega}_1,\ldots,\widetilde{\omega}_{d-1},\omega_T,\omega_T'$.  For $\omega_1,\ldots,\omega_r \in \HdR^1(Y)$, let $Z(\omega_1,\ldots,\omega_r)$ be the subspace of $\HdR^1(Y)$ consisting of the intersection of the kernels of $\omega_1^*,\ldots, \omega_r^*$.  
\end{notation}

For example $\widetilde{\omega}_i^*(\omega)$ is the coefficient of $\widetilde{\omega}_i$ when writing $\omega$ in terms of our chosen basis, while $Z(\beta_1,\widetilde{\omega}_{i})$ consists of $\omega \in \HdR^1(Y)$ for which the coefficients of $\beta_1$ and $\widetilde{\omega}_{i}$ are zero.

\begin{lemma} \label{lemma: superelliptic pairing}
With notation as above:
\begin{enumerate}
    \item  $V(\omega_T)$ is a $k$-linear combination of $\beta_1, \widetilde{\omega}_1 ,\ldots, \widetilde{\omega}_{(d-1)/2}$.  The coefficient of $\widetilde{\omega}_{(d-1)/2}$ is nonzero, and the coefficient of $\widetilde{\omega}_i$ is nonzero if and only if the coefficient of $t^{-(2i+1)}$ in the local expansion of $\psi$ at $Q$ is nonzero.
    \item  for $1 \leq i \leq (d-1)/2$
    \[
    \langle \widetilde{\omega}_i, \widetilde{\omega}_{i'} \rangle = \begin{cases}
        1 & \hbox{if $i + i' =d$}\\
     0 & \hbox{otherwise}.
    \end{cases}
    \]
    \item  $\beta_1$ pairs trivially with $\beta_2$ and with each $\widetilde{\omega}_i$ for $1 \leq i \leq d-1$, and $\beta_2$ pairs trivially with $\widetilde{\omega}_i$ for $1 \leq i \leq (d-1)/2$.
    \item  given $1 \leq i_1 < \ldots < i_r \leq (d-1)/2$,
    \[
    \spa\{\beta_1,\widetilde{\omega}_{i_1},\ldots,\widetilde{\omega}_{i_r}\} ^\perp = Z(\omega_T',\widetilde{\omega}_{d-i_1},\ldots, \widetilde{\omega}_{d-i_r}).
    \]
\end{enumerate}
\end{lemma}

\begin{proof}
For the first, recall that as $z^2 = z + \psi$ we have
\begin{equation}\label{eq: extra differential on bottom curve}
V(\omega_T) = (V(z^2dx + \psi dx) ,0) = (V(\psi dx),0).
\end{equation}
Thus the local expansion of $\psi$ at $Q$ controls $V(\omega_T)$.  In particular, if $\psi = \sum_{i=1}^{d} c_i t^{-i} + O(1)$ (with $c_d \neq 0$) then
\begin{equation}
V(\psi dx) = \left( \sum_{j=1}^{(d+1)/2} c_{2j-1}^{1/2} t^{-j} + O(1) \right) dt.
\end{equation}
Thus $V(\omega_T)$ is a $k$-linear combination of $\beta_1, \widetilde{\omega}_1 ,\ldots, \widetilde{\omega}_{(d-1)/2}$, and recalling the local expansions in Lemma~\ref{l: construct wij} we conclude the coefficient of $\widetilde{\omega}_{(d-1)/2}$ must be nonzero.  By similar reasoning, the coefficient of $\widetilde{\omega}_i$ is nonzero if and only if $c_{2i+1} \neq 0$.

The second follows from Proposition~\ref{proposition: symplectic pairing}, and the third from Lemma~\ref{lemma: M_0 orthogonal to W} and Lemma~\ref{lemma: M orthogonal to W_V}.
The last statement will follow from the second and third.  Since the pairing is non-degenerate and $\beta_1$ pairs trivially with all regular differentials as well as all the $\widetilde{\omega}_i$ and $\beta_2$, it must pair nontrivially with $\omega_T'$.  Thus when writing $\omega$ in the orthogonal complement in terms of the basis, they cannot have a $\omega_T'$ term. 
Furthermore, if $i \leq (d-1)/2$ then $\widetilde{\omega}_i$ pairs trivially with all basis elements except possibly for $\omega'_T$ and $\widetilde{\omega}_{d-i}$ (it is orthogonal to all regular differentials as it is regular).  Thus if $\omega$ is orthogonal to $\beta_1$ and $\widetilde{\omega}_i$ then the $\omega_T'$ and $\widetilde{\omega}_{d-i}$ terms are zero as desired.
\end{proof}

We will now give a couple of examples of computing $w(\HdR^1(Y))$ where $w$ is a simple word in $V$ and $\bot$.  

\begin{example} \label{ex:theoryd7}
Let $\pi : Y \to X$ be a double cover ramified at one point with ramification invariant $d=7$.  We will now establish that the two Ekedahl-Oort types observed in Example~\ref{example: different EO types} are the only possibilities.  Which one we get depends on whether the coefficient $c_5$ of $t^{-5}$ in the local expansion of $\psi$ is nonzero.  

In general, note that $V(\omega_T)$ is a linear combination of $\beta_1,\widetilde{\omega}_1,\widetilde{\omega}_2,\widetilde{\omega}_3$ and the coefficient of $\widetilde{\omega}_3$ must be nonzero.  Furthermore, $c_5 \neq 0$ if and only if $\widetilde{\omega}_2^*(V(\omega_T)) \neq 0$.
We may write $V^i(\HdR^1(Y))$ as a span of linearly independent vectors using Proposition~\ref{prop:Valpha}.

\begin{minipage}{.5 \textwidth}
If $c_5 \neq 0$, then:
\begin{align*}
V(\HdR^1(Y)) &= H^0(Y,\Omega^1_Y) \\ 
V^2(\HdR^1(Y)) &= \spa \{\beta_1, \widetilde{\omega}_1, V(\omega_T)\} \\
V^3(\HdR^1(Y)) &= \spa \{ \beta_1, V^2(\omega_T) \} = \spa \{ \beta_1, \widetilde{\omega}_1 \}  \\
V^4(\HdR^1(Y)) &= \spa \{ \beta_1 \} \\
V^5(\HdR^1(Y)) &= 0.
\end{align*}
\end{minipage}
\begin{minipage}{.5 \textwidth}
If $c_5 =0$, then:
\begin{align*}
V(\HdR^1(Y)) &= H^0(Y,\Omega^1_Y) \\ 
V^2(\HdR^1(Y)) &= \spa \{\beta_1, \widetilde{\omega}_1, V(\omega_T)\} \\
V^3(\HdR^1(Y)) &= \spa \{ \beta_1 \} \\
V^4(\HdR^1(Y)) &= 0 \\
V^5(\HdR^1(Y)) &= 0.
\end{align*}
\end{minipage}

Thus, when $c_5 \neq 0$ the $V$-type is $(5,3,2,1,0,\dots)$ and 
when $c_5=0$ the $V$-type is $(5,3,1,0,\dots)$.
In both cases, we see that $\spa \{\beta_1\}$ is part of the canonical filtration of $\HdR^1(Y)$. This allows us to construct more spaces in the canonical filtration:
\begin{align*}
\spa \{\beta_1\}^\perp &= Z(\omega'_T) \\
V(Z(\omega'_T)) &= \spa \{\beta_1, \widetilde{\omega}_1, \widetilde{\omega}_2, \widetilde{\omega}_3\} \\
V^2(Z(\omega'_T)) &= \spa \{\beta_1, \widetilde{\omega}_1 \}. 
\end{align*}
We conclude that $\HdR^1(Y)$ always has a canonical filtration whose first half is given by
\[
0 \subset \spa \{\beta_1\} \subset \spa \{ \beta_1, \widetilde{\omega}_1 \} \subset \spa \{\beta_1, \widetilde{\omega}_1, V(\omega_T)\}
 \subset \spa\{ \beta_1, \widetilde{\omega}_1,\widetilde{\omega}_2,\widetilde{\omega}_3\} \subset H^0(Y,\Omega^1_Y). \]

The final type of $Y$ then records the dimensions of the images of these subspaces under $V$. Note that we have already computed all these images; we see that the third entry of the final type depends on $c_5$ and the other entries do not. We conclude that the final type of $Y$ is $[0,1,1,2,3]$ when $c_5=0$ and it is $[0,1,2,2,3]$ when $c_5 \neq 0$.

\end{example}

\begin{example} \label{ex:theoryd15}
We now return to Example~\ref{ex:d15} and explain how we can obtain two covers with ramification invariant $d=15$ with  different Ekedahl-Oort types but where both curves have the same $V$-type (i.e. the $k[V]$-module structures of $H^0(Y,\Omega^1_Y)$ are the same.) %
As in the previous example, we see that
\[
V^2(\HdR^1(Y)) = \spa \{ \beta_1,\widetilde{\omega}_1,\widetilde{\omega}_2,\widetilde{\omega}_3,V(\omega_T) \}=:N_5.
\]
Here $N_i$ denotes the $i$-dimensional space in the canonical filtration of $\HdR^1(Y)$. 

Let $\psi= \sum_{i=1}^{15} c_i t^{-i} +O(1)$ be the local expansion of $\psi$ at the ramified point. As $c_{15} \neq 0$, it follows that $\widetilde{\omega}_7^*(V(\omega_T))\neq 0$, and hence the vectors defining $N_5$ are linearly independent. Further suppose that $c_9 \neq 0$, so that $\widetilde{\omega}_4^*(V(\omega_T)) \neq 0$.  Then we compute
\[
V^3(\HdR^1(Y)) = \spa \{\beta_1, \widetilde{\omega}_1, V^2(\omega_T)\} \quad V^4(\HdR^1(Y)) = \spa \{\beta_1,\widetilde{\omega}_1\} \quad V^5(\HdR^1(Y)) = \spa \{ \beta_1\}
\]
where in each case the specified vectors are linearly independent.  Thus,
the $V$-type of $Y$ is $(9,5,3,2,1,0,\dots)$, which determines the higher $a$-numbers. %
(There are other possibilities when $c_9=0$, but we do not consider those here.)

However, this is not enough information to determine the full Ekedahl-Oort type. More precisely, we will see that the final type of $Y$ depends on the coefficient $c_{13}$ of $t^{-13}$. Note that $c_{13}=\widetilde{\omega}_3^*(V^2(\omega_T))$. Now, we proceed to construct the canonical filtration of $\HdR^1(Y)$:
\begin{align*}
V^2\left(\HdR^1(Y)\right)^\perp &= N_5^\perp = Z \left(\omega'_T, \widetilde{\omega}_{14}, \widetilde{\omega}_{13}, \widetilde{\omega}_{12} \right) \cap \ker \langle V(\omega_T),- \rangle =: N_{13} \\
V\left(V^2\left(\HdR^1(Y)\right)^\perp\right) &= V(N_{13}) =  \spa \left\{\beta_1, \widetilde{\omega}_1, \ldots , \widetilde{\omega}_5, V(\omega_T) \right\} =: N_7 \\
V\left(V^2\left(\HdR^1(Y)\right)^\perp\right)^\perp &= N_7^\perp = Z\left(\omega'_T , \widetilde{\omega}_{14}, \widetilde{\omega}_{13}, \widetilde{\omega}_{12}, \widetilde{\omega}_{11}, \widetilde{\omega}_{10} \right) \cap \ker \langle V(\omega_T), - \rangle =: N_{11} \\
V\left(V\left(V^2\left(\HdR^1(Y)\right)^\perp\right)^\perp\right) &= V(N_{11}) = \spa \left\{\beta_1, \widetilde{\omega}_1, \widetilde{\omega}_2, \widetilde{\omega}_3, \widetilde{\omega}_4, V(\omega_T) \right\} =: N_6.
\end{align*}
We deduce from this that
\[
\dim_k (V(N_6)) = \dim_k(V(N_7)) = \begin{cases} 3 &\hbox{if $c_{13}=0$} \\ 4 &\hbox{if $c_{13} \neq 0$.} \end{cases}
\]
Furthermore, one can verify that all the other entries of the final type are determined by the assumptions $d=15$ and $c_9 \neq 0$. Thus we conclude that the final type of $Y$ is $[0,1,2,2,3,4,4,4,5]$ when $c_9\neq 0$ and $c_{13}\neq 0$, and that the final type of $Y$ is $[0,1,2,2,3,3,3,4,5]$ when $c_9 \neq 0$ and $c_{13}=0$.

\end{example}

On the other hand, in some special cases there is a single Ekedahl-Oort type.

\begin{theorem} \label{theorem: d = 2^n+1}
Let $X$ be the unique supersingular elliptic curve over $\overline{\mathbf{F}}_2$.  Let $\pi : Y \to X$ be a double cover ramified over exactly one point of $X$, with ramification invariant $d = 2^n +1$ for some positive integer $n$.  Then the Ekedahl-Oort type of $Y$ is determined by $d$: in particular if $d>3$ the final type is $$[0,1,2,3,3,\ldots,(d-1)/4+1 ,(d-1)/4+1,(d-1)/4+2 ].$$
If $d=3$ then the final type is $[0,1,2]$.
\end{theorem}

The proof is similar in spirit to Example~\ref{ex:theoryd7}, but simpler as the only relevant coefficient in the local expansion of $\psi$ is leading one (the coefficient of $t^{-d}$) which is automatically nonzero.

\begin{definition}
For an integer $0 < i \leq (d-1)/2$ let $j(i) \colonequals n - 1 - \lfloor \log_2(i) \rfloor$ and
\begin{equation}
U_i \colonequals \spa \{ \beta_1,\widetilde{\omega}_1,\ldots,\widetilde{\omega}_i, V^{j(i)}(\omega_T)\} \subseteq H^0(Y,\Omega^1_Y).
\end{equation}
Furthermore let $U_0 = \spa \{ \beta_1,\widetilde{\omega}_1\} = \spa \{\beta_1,V^n(\omega_T )\}$ and $U_{-1} \colonequals \spa \{ \beta_1 \}$.
\end{definition}

Note that $U_{(d-1)/2} = H^0(Y,\Omega^1_Y)$ since in that case $j(i) =  0$.

\begin{lemma} \label{lemma:d=2^n+1} If $0<i\leq (d-1)/2$ then 
\begin{enumerate}
    \item   $V(U_i) = U_{i'}$ where $i' = \lfloor i/2 \rfloor$.  
    \item  $V(U_i^\perp) = U_{i''}$ where $i'' = \lfloor (d-1-i)/2 \rfloor$. 
    \item  $\dim U_i = i+2$.
    \item  $U_{m-1} \subset U_m$ for $0\leq m \leq (d-1)/2$
\end{enumerate}
\end{lemma}

\begin{proof}
The claim about $V(U_i)$ follows from the definition using Proposition~\ref{prop:Valpha}.
Now let us consider $U_i^\perp$, and let $j = j(i)$ to simplify notation.  By Lemma~\ref{lemma: superelliptic pairing}, $U_i^\perp$ is a codimension one subspace of 
\[
Z(\omega'_T,\widetilde{\omega}_{d-1},\ldots, \widetilde{\omega}_{d-i}) = \spa \{ \beta_1,\beta_2,\widetilde{\omega}_1,\ldots,\widetilde{\omega}_{d-1-i},\omega_T\}.
\]
consisting of elements $\omega$ orthogonal to $V^{j}(\omega_T)$. Any regular differential is automatically orthogonal to $\omega_T$, but $\widetilde{\omega}_{m}$ need not be for $m > (d-1)/2$.  Now $\widetilde{\omega}_{(d-1)/2}^*(V(\omega_T)) = \widetilde{\omega}_{2^{n-1}}^*(V(\omega_T)) \neq 0$ so $\widetilde{\omega}^*_{2^{n-j}}( V^j(\omega_T)) \neq 0$, so the orthogonality relation becomes a linear relation of the form 
\[
c_{d-2^{n-j}} \widetilde{\omega}_{d-2^{n-j}}^*(\omega) + \ldots = 0
\]
where $c_{d-2^{n-j}} $ is known to be nonzero.  This means that $U_i^\perp$ contains an element of the form $c' \widetilde{\omega}_{d-2^{n-j}} + \widetilde{\omega}_m$ for each $m$.  Since $d$ is odd,
\[
V(c' \widetilde{\omega}_{d-2^{n-j}} + \widetilde{\omega}_m) = V(\widetilde{\omega}_m).
\]
Then using Proposition~\ref{prop:Valpha} we conclude that 
\[
V(U_i^\perp) = U_{i''}
\]

For the claim about dimensions, remember that $\widetilde{\omega}^*_{2^{n-j}}( V^j(\omega)) \neq 0$.  Thus $V^j(\omega)$ is not in the span of the linearly independent $\beta_1 , \widetilde{\omega}_1 ,\ldots, \widetilde{\omega}_i$.

Finally, observe that $U_{-1} \subset U_0 \subset U_1$.  If $j(m)=j(m-1)$, then it is clear by definition that $U_{m-1}\subset U_m$.  Otherwise $m=2^r$, in which case $V^{n-r}(\omega_T) \in \spa \{ \beta_1,\widetilde{\omega}_1,\ldots,\widetilde{\omega}_{m} \}$ since $\ord_Q(V^{n-r}(\omega_T)) = - 2^{r} - 1$.
\end{proof}

Similar reasoning deals with the exceptional case $i=0$, showing $V(U_0) = U_{-1}$, $V(U_0^\perp) = U_{(d-1)/2-1}$, and $\dim U_0 = 2$.

\begin{proof}[Proof of Theorem~\ref{theorem: d = 2^n+1}]
We claim that
\[
0 \subset U_{-1} \subset U_0 \subset U_1 \subset \ldots \subset U_{(d-1)/2} = H^0(Y) \subset U_{(d-1)/2-1}^\perp \subset \ldots \subset U_{-1}^\perp \subset \HdR^1(Y)
\]
is the final filtration for $\HdR^1(Y)$.
Lemma~\ref{lemma:d=2^n+1} shows that this is a filtration, that it is preserved by $V$ and $\perp$, and furthermore gives the dimensions.  As $\dim_k V(U_i) = \lfloor i/2 \rfloor +2 $ for $i > 0 $, we conclude that the final type is $[0,1,2,3,3,\ldots,(d-1)/4+1 ,(d-1)/4+1,(d-1)/4+2 ]$.
\end{proof}

We can now finish the proof of Corollary~\ref{cor:families}.

\begin{proof}[Proof of Corollary~\ref{cor:families}]
The first family, constructed as a cover of $\PP^1$, was analyzed in Remark~\ref{rmk:family1}.  The second family is similar, constructed as a cover of the supersingular elliptic curve $X$ ramified over the point $Q$ at infinity, with ramification break $d$, and relies on Theorem~\ref{theorem: d = 2^n+1}.  The only non-obvious point is the computation of the dimension of the family.  The extension of function fields is generated by adjoining a root of $z^2 + z = f$, where $f$ is a rational function on $X$ which is regular except at infinity.  The function $f$ is unique up to adding functions of the form $g^2+g$.  

By the Riemann-Roch theorem, there exists a rational function on $X$ which is regular except at infinity and has a pole of order $n$ at $Q$ for any $n >1$.  
Since $d>1$, we may assume that the order of $f$ at $Q$ is exactly $d$ by modifying $f$ by a function of the form $g^2 + g$.  Similarly, we assume that the coefficient of $t^{-2i}$ in the local expansion of $f$ at $Q$ is zero for $i>1$, and that there is no constant term in the local expansion.  Now, by Riemann-Roch, $\dim_k H^0(X,\cO_X(d Q)) = d$, so the space of possible $f$ which give non-isomorphic extension of the function field has dimension $d-1-(d-3)/2 = (d+1)/2$. 
There is a zero-dimensional space of automorphisms of $X$ which fix the point at infinity, so there is also a $(d+1)/2$-dimensional family of curves with constant Ekedahl-Oort type. 
\end{proof}

\section{Bounds on the final type when the base curve is not ordinary} \label{sec:bounds}

\subsection{The setting}

As before, let $\pi: Y \to X$ be a double cover of smooth, proper, connected curves over an algebraically closed field $k$ of characteristic $2$. In Theorem~\ref{thm: ordinary EO type}, we proved that, when $X$ is ordinary, the isomorphism class of $\Jac(Y)[p]$ is determined by the ramification invariants $d_i$, and the isomorphism class is built from local contributions at each ramified point. When $X$ is not ordinary, on the other hand, the isomorphism class of $\Jac(Y)[p]$ is not determined by $X$ and the ramification invariants $d_i$. See Example~\ref{example: different EO types}. 
In this section, we use the ideas from Section~\ref{section: omegaij} to prove bounds on the final type of a cover of a non-ordinary curve with given ramification invariants.  Examples~\ref{ex:theoryd7} and \ref{ex:theoryd15} may be helpful to keep in mind. %

Recall from Section~\ref{ss:final types}  that the final type of $Y$ measures the interaction between the operations $V$ and $\perp$, or equivalently the interaction between $V$ and $F^{-1}$. More precisely, it is determined by $\dim_k (V(w(\HdR^1(Y)))$ as $w$ ranges over all words in the letters $V$ and $\perp$. Thus it suffices to bound the dimensions of these spaces.  Lemma~\ref{lemma: decomposition k[F,V]} gives a decomposition of Dieudonn\'{e} modules
\begin{align} \label{eq:UZL decomposition}
    \HdR^1(Y) = U \oplus Z \oplus L,
\end{align}
where $V$ is bijective on $U$, $F$ is bijective on $Z$ and both operators are nilpotent on $L$. Recall the pairing $U \times Z \to k$ is perfect and the spaces $U \oplus Z$ and $L$ are orthogonal.

\begin{lemma} \label{lemma: reduce to L}
If $w$ is a simple word, then $w(\HdR^1(Y)) = U \oplus (w(L) \cap L)$.
\end{lemma}
\begin{proof}
We argue by induction on the length of $w$. The equation holds for $w=V$, since $V$ preserves the three summands $U$, $Z$ and $L$, and we have $V(U)=U$ and $V(Z)=0$. Now assume it holds for all simple words that have up to $n$ letters and consider a word $w$ of $n+1$ letters. If $w$ starts with $V$, write $w=Vw'$ and observe
\begin{align*}
w(\HdR^1(Y)) &= V(w'(\HdR^1(Y))) = V\left(U \oplus (w'(L) \cap L) \right)
= V(U) \oplus V(w'(L)\cap L) \\ &= U \oplus (V(w'(L)) \cap L) = U \oplus (w(L) \cap L)
\end{align*}
as desired. On the other hand, if $w$ starts with $\perp$, write $w=\perp w'$ and observe
\begin{align*}
w(\HdR^1(Y)) &= w'(\HdR^1(Y))^\perp = (U \oplus (w'(L) \cap L))^\perp = U^\perp \cap (w'(L) \cap L)^\perp  \\ &= (U \oplus L) \cap (U\oplus Z \oplus (w'(L)^\perp \cap L)) = U \oplus (w'(L)^\perp \cap L) = U \oplus (w(L) \cap L).\qedhere
\end{align*}
\end{proof}

Thus it suffices to determine the final type of $L$. In particular, it suffices to bound the dimensions of the spaces $w(L)$ for simple words $w$.

Recall that in Section~\ref{section: omegaij} (in particular, Definition~\ref{defn:wijtilde}), we constructed subspaces $W_i \subset L$ spanned by classes $\widetilde{\omega}_{i,j}$ and formed the sum $W = \bigoplus_{i=1}^m W_i$.  We also introduced subspaces $W_{i,l} \subset W_i$ with $\dim_k W_{i,l}=l$. Finally, recall the function $\phi(d,w)$ and the notation $\phi(w)=\sum_{i=1}^m \phi(d_i,w)$ introduced in Definition~\ref{def: phi}.

\begin{definition} \label{defn:Ww}
For a simple word $w$, let $W_w \colonequals \bigoplus_{i=1}^m W_{i,\phi(d_i,w)}$.
\end{definition}

Note that, by definition, $\dim_k (W_w)=\phi(w)$. Note also that $W_V=W \cap \HH^0(Y,\Omega_Y^1)$.

\begin{lemma} \label{lemma:projection}
For any simple word $w$, we have $\pi_W(V(W_w)) = W_{Vw}$ and $\pi_W(W_w^\perp) = W_{\perp w}$.
\end{lemma}

\begin{proof}
The proof is similar to the proof of Lemma~\ref{lemma: phi ordinary}. For the first claim let $\widetilde{\omega}_{i,j} \in W_w$. Then by Proposition~\ref{prop:nilpotent wij} and Lemma~\ref{lemma:Ldecomposition}
we have $\pi_W(V(\widetilde{\omega}_{i,j}))=\widetilde{\omega}_{i,j/2}$
for $j$ even and $0$ for $j$ odd. Thus,

$$\pi_W(V(W_w)) = \pi_W\left( V\left( \bigoplus_{i=1}^m W_{i,\phi(d_i,w)} \right) \right) = \bigoplus_{i=1}^m W_{i,\floor{\phi(d_i,w)/2}} = \bigoplus_{i=1}^m W_{i,\phi(d_i,Vw)} = W_{Vw} .$$

The second claim follows from Proposition~\ref{proposition: symplectic pairing} and Lemma~\ref{lemma:Ldecomposition}. More precisely, we obtain
\[
\pi_W(W_w^\perp) = \pi_W\left( \bigcap_{i=1}^m W_{i,\phi(d_i,w)}^\perp \right) = \bigoplus_{i=1}^m W_{i,d_i-1-\phi(d_i,w)} = \bigoplus_{i=1}^m W_{i,\phi(d_i,\perp w)} = W_{\perp w}. \qedhere
\]
\end{proof}

\subsection{Inductive approach on words}
In this section we bound $\dim_k (w(L))$ for simple words $w$.  
The strategy is to construct spaces that bound $w(L)$ via induction on the  number of occurrences of $\perp$ in $w$.
The necessary induction steps, proved in Lemmas~\ref{lemma: upper to lower} and~\ref{lemma: lower to upper}, are put together in Proposition~\ref{prop: w(L) bounds}.

Recall that $l_X=g_X-f_X$. Recall the spaces $M_0$, $M_1$, $M$, $T_0$, $T_1$ and $T$ defined in Lemma~\ref{lemma:Ldecomposition}.

\begin{lemma} \label{lemma: upper to lower}
Let $w$ be a simple word starting with $\perp$. Assume that there exists a space $W_w^U \subseteq T_1 \oplus W$ such that the inclusion $W_w \subseteq W_w^U$ has codimension at most $c$ and
$$w(L) \subseteq M \oplus T_0 \oplus W_w^U.$$
Then, for every $n$, there exists a space $W_{\perp V^n w}^L \supseteq W_V$ such that $W_{\perp V^n w}^L \subseteq W_{\perp V^n w}$ with codimension at most $c+l_X$ such that
\begin{equation} \label{eq: upper to lower}
V^n(w(L))^\perp \supseteq M_0 \oplus T_0 \oplus W_{\perp V^n w}^L.
\end{equation}
\end{lemma}

\begin{proof}
We begin by noting that
\begin{equation*} \label{eq: upper to lower 2}
V^n( w(L))^\perp \supseteq V^n(M \oplus T_0 \oplus W_w^U)^\perp = V^n(M \oplus W_w^U)^\perp \cap V^n(T_0)^\perp.
\end{equation*}
Now 
$V^n(M \oplus W_w) \subseteq M_0 \oplus W_{V^n w}$ and the inclusion $V^n(M \oplus W_w) \subseteq V^n(M \oplus W_w^U)$ has codimension at most $c$.  Hence there exists a space $W_{V^n w}^U$ containing $W_{V^n w}$ with codimension at most $c$ such that $V^n(M \oplus W_w^U) \subseteq M_0 \oplus W_{V^nw}^U$. Taking symplectic complements gives
$$V^n(M \oplus W_w^U)^\perp \supseteq (M_0 \oplus W_{V^nw}^U)^\perp = M_0^\perp \cap W_{V^n w}^{U ,\perp}.$$
By Lemma~\ref{lemma: M_0 orthogonal to W}, we know $M_0^\perp \supseteq M_0 \oplus W \oplus T_0$. Furthermore, $W_{V^n w}^{U ,\perp}$ is contained in $W_{V^n w}^\perp$ with codimension at most $c$. Projecting to both spaces to $W$, we see that $\pi_W(W_{V^n w}^{U ,\perp})$ is contained in $\pi_W(W_{V^n w}^\perp)= W_{\perp V^n w}$ with codimension at most $c$.
Thus $W_{\perp V^n w}^L\colonequals \pi_W(W_{V^n w}^{U ,\perp} ) \cap V^n(T_0)^\perp$ satisfies Equation~\eqref{eq: upper to lower} and has codimension at most $c+\dim_k(V^n(T_0)) \leq c+l_X$ inside $W_{\perp V^n w}$.  
Since $W_{V^n w}^U$ and $V^n(T_0)$ both consist of regular differentials, it follows that $W_V \subseteq \pi_W(W_{V^n w}^{U ,\perp} ) \cap V^n(T_0)^\perp = W_{\perp V^n w}^L$. 
\end{proof}

Recall the $n$-th higher $a$-number, denoted 
$$a_X^n := \dim_k ( \ker(V: \HH^0(X,\Omega_X^1) \to \HH^0(X,\Omega_X^1)))$$
and observe that $a_X^n \leq l_X$ for every $n$.

\begin{lemma} \label{lemma: lower to upper}
Let $w$ be a simple word starting with $\bot$. Assume that there exists a space $W_w^L \supseteq W_V$ such that the inclusion $W_w^L \subseteq W_w$ has codimension at most $c$ and
$$w(L) \supseteq M_0 \oplus T_0 \oplus W_w^L.$$
Then there exists a space $W_{\bot V^n w}^U \subseteq T_1 \oplus W$ such that the inclusion $W_{\bot V^n w} \subseteq W_{\bot V^n w}^U$ has codimension at most $c+l_X+a_X^n$ and
\begin{equation} \label{eq: lower to upper}
    V^n(w(L))^\perp \subseteq M \oplus T_0 \oplus W_{\bot V^n w}^U.
\end{equation}
If furthermore $\sum_{i=1}^m (d_i+1) > 4g_X-4$  holds, then the codimension of $W_{\bot V^n w} \subseteq W_{\bot V^n w}^U$ is at most $c+a_X^n$.
\end{lemma}

\begin{proof}
We begin by defining
\begin{equation} \label{eq:def wU}
W^U_{\bot V^nw} \colonequals \pi_{T_1 \oplus W}(V^n(M_0 \oplus T_0 \oplus W_w^L)^\perp) + W_{\bot V^n w} .
\end{equation}
Observe that
\begin{itemize}
    \item  $W^U_{\bot V^nw} = \left( \pi_{T_1 \oplus W}(V^n(M_0 \oplus W_w^L)^\perp) \cap \pi_{T_1 \oplus W}(V^n(T_0)^\perp) \right) + W_{\bot V^n w} $;
    \item $W_{\bot V^n w} \subseteq W^U_{\bot V^n w}$; and
    \item $W^U_{\bot V^nw} \subseteq T_1 \oplus W$.
\end{itemize}
To check \eqref{eq: lower to upper}, note that since $w(L) \supseteq M_0 \oplus T_0 \oplus W_w^L$ we know that
\begin{equation} \label{equation: Vnwlperp}
V^n(w(L))^\perp \subseteq V^n(M_0 \oplus T_0 \oplus W_w^L)^\perp = V^n(M_0 \oplus W_w^L)^\perp \cap V^n(T_0)^\perp.
\end{equation}
Now recalling Equation~\eqref{equation:regularperp}, observe that
\begin{equation} \label{eq:first projection}
V^n(M_0 \oplus W_w^L)^\perp = M_0 \oplus T_0 \oplus \pi_{M_1 \oplus T_1 \oplus W} (V^n(M_0 \oplus W_w^L)^\perp) \subseteq M \oplus T_0 \oplus \pi_{T_1 \oplus W} (V^n(M_0 \oplus W_w^L)^\perp) .
\end{equation}
Similarly as $V^n(T_0)^\perp \supseteq M_0 \oplus T_0$ we see
\[
V^n(T_0)^\perp = M_0 \oplus T_0  \oplus \pi_{M_1 \oplus T_1 \oplus W}(V^n(T_0)^\perp) \subseteq M \oplus T_0 \oplus \pi_{T_1 \oplus W}(V^n(T_0)^\perp).
\]
Therefore Equation~\eqref{equation: Vnwlperp} gives that $V^n(w(L))^\perp \subseteq M \oplus T_0 \oplus W^U_{\bot V^nw}$ as desired.

It remains to bound the codimension of $W_{\bot V^n w} \subseteq W_{\bot V^n w}^U$. 
We begin by noting $V^n(M_0)$ has dimension $l_X-a_X^n$ by the rank-nullity theorem. Moreover, using Lemma~\ref{lemma:projection} yields
$$\dim_k(\pi_W(V^n(W_w))) = \dim_k (W_{V^nw}) = \sum_{i=1}^m \phi(d_i,V^n w)=\phi(V^n w).$$
  It follows that 
\begin{align*}
    \dim_k (V^n(M_0 \oplus W_w^L)) &\geq (l_X-a_X^n) + \phi(V^n w)-c,
    \intertext{so that}
    \dim_k ( V^n(M_0 \oplus W_w^L)^\perp) &\leq 2l_Y - (l_X-a_X^n) - \phi(V^n w)+c \\
    &= 4l_X + \sum_{i=1}^m (d_i-1) - (l_X-a_X^n) - \sum_{i=1}^m \phi(d_i,V^n w)+c \\
    &= 3l_X + a_X^n + \phi(\bot V^n w) + c.
\end{align*}
The last step uses that $\phi(d_i,\bot V^n w) = d_i-1-\phi(d_i,V^n w)$. 
Then Equation~\eqref{eq:first projection} gives that
\begin{equation} \label{eq: upper bound general case}
\dim_k (\pi_{T_1 \oplus W}(V^n(M_0 \oplus W_w^L)^\perp)) \leq l_X + a_X^n +  \phi(\bot V^n w) + c
\end{equation}
Letting $c'$ be the codimension of $\pi_{T_1 \oplus W}(V^n(M_0 \oplus W_w^L)^\perp) \cap V^n(T_0)^\perp \subseteq \pi_{T_1 \oplus W}(V^n(M_0 \oplus W_w^L)^\perp)$, we conclude that
\[
\dim_k (\pi_{T_1 \oplus W} (V^n(M_0 \oplus T_0 \oplus W_w^L)^\perp)) \leq  \phi(\perp V^n w) + l_X + a_X^n +c - c'.
\]
Furthermore, as $W_{\bot V^n w} \subseteq \pi_{T_1 \oplus W}(V^n(M_0 \oplus W_w^L)^\perp)$ we see that 
\[
\codim_k\left(W_{\bot V^n w} \cap V^n(T_0)^\perp \subseteq W_{\bot V^n w} \right) \leq c'
\]
Letting $Z = \pi_{T_1 \oplus W} (V^n(M_0 \oplus T_0 \oplus W_w^L)^\perp)$, we estimate
\begin{align*}
    \codim_k(W_{\bot V^n w} \subseteq W_{\bot V^n w}^U) &\leq \dim_k Z - \dim_k( W_{\bot V^n w} \cap Z)\\
    &= \dim_k Z - \dim_k( W_{\bot V^n w} \cap V^n(T_0)^\perp)\\
    &\leq \left( \phi(\bot V^n w) + l_X + a_X^n +c -  c' \right)\\ &\hspace{10mm} - \left( \phi(\bot V^n w) -  \codim_k\left(W_{\bot V^n w} \cap V^n(T_0)^\perp \subseteq W_{\bot V^n w} \right)  \right) \\
    &\leq c+l_X+a_X^n.
\end{align*}
This gives the general claim about the codimension.  

We now prove the sharper upper bound under the assumption $\sum_{i=1}^m (d_i+1) > 4g_X-4$. Under this assumption, Lemma~\ref{lemma: M orthogonal to W_V} yields
$$M \subseteq (M_0 \oplus W_V)^\perp \subseteq V^n(M_0 \oplus W_w^L)^\perp.$$
As a result, the inclusion in Equation~\eqref{eq:first projection} becomes an equality:
$$V^n(M_0 \oplus W_w^L)^\perp = M \oplus T_0 \oplus \pi_{T_1 \oplus W} (V^n(M_0 \oplus W_w^L)^\perp).$$
After computing the dimension of the left-hand side in the same way as before, we now get
\begin{align*}
\dim_k (\pi_{T_1 \oplus W}(V^n(M_0 \oplus W_w^L)^\perp)) &= \dim_k (V^n(M_0 \oplus W_w^L)^\perp) - \dim_k (M \oplus T_0) \\
&= \dim_k (V^n(M_0 \oplus W_w^L)^\perp) - 3l_X \\
&\leq a_X^n +  \phi(\bot V^n w) + c.
\end{align*}
Comparing this to Equation~\eqref{eq: upper bound general case}, the bound has decreased by $l_X$ using the assumption $\sum_{i=1}^m (d_i+1) >4g_X -4$. From there on, the same argument building on Equation~\eqref{eq: upper bound general case} shows that
\[\codim_k (W_{\bot V^n w} \subseteq W_{\bot V^n w}^U) \leq c+a_X^n.\qedhere\]
\end{proof}

\begin{remark}
In some cases, this bound can be slightly improved by using Tango's theorem \cite{Tango}. For this, let 
\begin{equation} \label{eq:Tango def}
n(X)\colonequals  \frac{1}{2} \max \left \{ \sum_{P \in X(k)} \ord_P(df) \; | \; f \in k(X) \setminus k(X)^2   \right \}
\end{equation}
be the Tango number of $X$. When $\sum_{i=1}^m \left \lfloor \frac{d_i+1}{2^{n+1}} \right \rfloor \geq n(X)$, we have 
\begin{align*}
M_0 \subseteq V^n(M_0 \oplus W_V) &\subseteq V^n(M_0 \oplus W_w^L) 
\intertext{whence}
\dim_k(V^n(M_0 \oplus W_w^L)) &\geq l_X + \phi(V^n w) - c.
\end{align*}
In that case the bound on $\codim_k(W_{\perp V^n w}\subseteq W_{\perp V^n w}^U)$ is reduced by $a_X^n$. Since this improvement depends on $n$, it is cumbersome to include this condition in the next Proposition (and its inductive proof), but it does offer occasional improvement. %
\end{remark}

\begin{proposition} \label{prop: w(L) bounds}
    Let $w=\bot V^{n_t} \bot \ldots  \bot V^{n_1}$ be a simple word containing $t$ instances of $\bot$. There exists
\begin{itemize}
    \item a space $W_w^L \supseteq W_V$ that has codimension at most $\lfloor \tfrac{3t}{2} \rfloor l_X$ inside $W_w$, %
    \item and a space $W_w^U \subseteq T_1 \oplus W$ in which $W_w$ %
    has codimension at most $ \lceil \frac{3t}{2} \rceil l_X$,
\end{itemize}
such that the following holds:
$$M_0 \oplus T_0 \oplus W_{w}^L \subseteq w(L) \subseteq M \oplus T_0 \oplus W_{w}^U.$$
Furthermore, if we additionally assume $\sum_{i=1}^m (d_i+1) > 4g_X-4$, then both codimensions are at most $tl_X$.
\end{proposition}
\begin{proof}
The proof is by induction on $t$.  

To make the base case easier, we write $\bar{w}\colonequals \bot V^{n_t} \bot \ldots V^{n_1-1} \bot V$. Note that $\bar{w}(L) = w(L)$ since $V(L) = V(L)^\perp=L_0$. This reduces us to the base case $w=\bot V$:
$$w(L) = V(L)^\perp = V(L) = L_0 = M_0 \oplus T_0 \oplus W_V.$$
Thus for $w=\bot V$ we may set $W_w^L = W_w^U = W_w$, so both codimensions equal zero. 

For the induction step, assume the statement holds for a simple word $w$ containing $t$ instances of $\bot$, and let $n>0$.  Using $W^U_w$, Lemma~\ref{lemma: upper to lower} gives a space $W_{V^n \bot w}^L$, which has codimension at most $\left \lceil \frac{3t}{2} \right \rceil l_X + l_X =  \left \lfloor \frac{3(t+1)}{2} \right \rfloor l_X$ inside $W_{\bot V^n w}$. If we additionally assume $\sum_{i=1}^m (d_i+1) > 4g_X-4$, then the codimension is at most $tl_X + l_X = (t+1)l_X$.

Analogously, using Lemma~\ref{lemma: lower to upper} gives a space $W_{\bot V^n w}^U \subseteq T_1 \oplus W$, in which $W_{\perp V^n w}$ has codimension at most $$\left \lfloor \frac{3t}{2} \right \rfloor l_X + l_X + a_X^n \leq \left \lfloor \frac{3t}{2} \right \rfloor l_X + 2l_X = \left \lceil \frac{3(t+1)}{2} \right \rceil l_X,$$
as desired. If we additionally assume $\sum_{i=1}^m (d_i+1) > 4g_X -4$, then the codimension is at most $tl_X + a^n_X \leq t l_X + l_X = (t+1)l_X$.
\end{proof}

\subsection{Bounds on \texorpdfstring{$\dim_k ( w(L))$}{dim(w(L))}}

We now use Proposition~\ref{prop: w(L) bounds} to obtain restrictions on $\dim_k w(\HdR^1(Y))$ for simple words $w$. This results in Proposition~\ref{prop: bounds dim(w(L))}, which in turn gives restrictions on the final type of $Y$ (see Theorem~\ref{theorem: final type bounds}). The following lemmas allow us to strengthen the lower bound.

\begin{lemma} \label{lemma: BC8.3 adaptation}
Let $\eta=\omega_0 + \omega_1 y$ be a (not necessarily regular) differential on $Y$ and let $n$ be an integer. Assume $\ord_{P_i}(\omega_0) \geq -n$ and $V(\eta)=0$. Then we have
$$\ord_{P_i}(\pi_*(\eta)) = \ord_{P_i} (\omega_1) \geq 2 \left \lceil\frac{d_i-n}{2} \right \rceil.$$
\end{lemma}

\begin{proof}
This is an adaptation in characteristic $2$ of the proof of \cite[Theorem 8.3]{BoCaIwasawa}, with the difference that we no longer require $\eta$ to be regular. The decomposition $\eta=\omega_0 + \omega_1 y$ has implicitly fixed an equation $y^2+y=\psi$ defining $Y$.  Now observe
$$V(\eta) = V(\omega_0) + V(\omega_1 y) = V(\omega_0) + V(\omega_1(y^2+\psi)) = V(\omega_0) + V(\psi \omega_1) + y V(\omega_1) = 0,$$
so that $V(\omega_1)=0$ and $V(\omega_0) = V(\psi \omega_1)$. Since $\ord_{P_i}(\omega_0) \geq -n$, it follows that 
\begin{equation} \label{eq: BC8.3}
    \ord_{P_i}(V(\psi \omega_1)) = \ord_{P_i}(V(\omega_0)) \geq - \left \lfloor \frac{n+1}{2} \right \rfloor.
\end{equation}
Now, we may pick a uniformizer $z_i \in \widehat{\mathcal{O}}_{X,P_i}$ such that $\psi = c^2z_i^{-d_i}$ for some $c \in k^\times$. Since $V(\omega_1)=0$, there exists a function $f_1\in \widehat{\mathcal{O}}_{X,P_i} \cong k[\![z]\!]$ 
such that $\omega_1 = f_1^2 dz_i$. Then we obtain
$$ V(\psi \omega_1) = V(c^2f_1^2 z_i^{-d_i}z_i) = cf_1z_i^{-\frac{d_i+1}{2}} dz_i.$$
Combining this with Equation~\eqref{eq: BC8.3} yields 
\begin{align*}
    \ord_{P_i}(f_1) &\geq \frac{d_i+1}{2} - \left \lfloor \frac{n+1}{2} \right \rfloor = \left \lceil \frac{d_i-n}{2}  \right \rceil, \intertext{and therefore}
    \ord_{P_i}(\omega_1) &= \ord_{P_i}(f_1^2 dz_i) \geq 2 \left \lceil \frac{d_i-n}{2}  \right \rceil. \qedhere
\end{align*}
\end{proof}

\begin{corollary} \label{cor: BC8.4}
Consider a simple word $w=\bot V^{n_t} \bot \ldots V^{n_1} = \bot w'$ and let $(\eta,(f_i))$ represent a class in $M_0 \oplus T_0 \oplus W_w.$  If
\begin{equation} \label{eq: BC8.4}
\sum_{i=1}^m \left \lceil \frac{\phi(d_i,w')}{2} \right \rceil \geq g_X
\end{equation}
and $V(\eta)=0$, then it follows that $\pi_*(\eta)=0.$
\end{corollary}
\begin{proof}
Since $(\eta,(f_i))$ is in $M_0 \oplus T_0 \oplus W_w$, we can write $\eta=\omega_0 + \omega_1 y$, where $\omega_1$ is regular. Recalling Definitions \ref{defn:wijtilde} and \ref{defn:Ww}, note $\ord_{P_i}(\omega_0) \geq -\phi(d_i,w) - 1$ for every $i$. Then Lemma~\ref{lemma: BC8.3 adaptation} yields that 
$$\ord_{P_i}(\omega_1) \geq 2 \left \lceil \frac{d_i-1-\phi(d_i,w)}{2}  \right \rceil = 2 \left \lceil \frac{\phi(d_i,w')}{2}  \right \rceil.$$
Thus, if we consider the effective divisor
$$E= \sum_{i=1}^m 2 \left \lceil \frac{\phi(d_i,w')}{2}  \right \rceil [P_i],$$
we see $\pi_*(\eta) = \omega_1 \in H^0(X,\Om_X^1(-E))=0$ since $\deg(E)>2g_X-2$ by Equation~\eqref{eq: BC8.4}.
\end{proof}

Recall Equation~\eqref{eq:Tango def} defining the Tango number $n(X)$.
Note that $n(X) \leq g_X-1$ by \cite[Lemma 10]{Tango}. Define the divisor on $X$
$$D\colonequals  \sum_{i=1}^m \frac{d_i + 1}{2} [P_i].$$ 

\begin{definition} \label{def: L(X,pi,w)}
Consider a simple word $w=V^r \bot V^{n_t} \bot \ldots \bot V^{n_1}=V^r \bot w'$. Define
\begin{align*} \label{eq: defintion of L and U}
L_1(X, \pi , w) & \colonequals  \begin{cases} l_X - a_X^r &\hbox{if $\deg(\lfloor D/2^r \rfloor) < n(X)$}, \\
l_X &\hbox{if $\deg(\lfloor D/2^r \rfloor) \geq n(X)$}. 
\end{cases}\\
L_2(X,\pi,w) & \colonequals  \begin{cases}  \phi(w) - \left \lfloor \frac{3t}{2} \right \rfloor l_X &\hbox{if $\sum_{i=1}^m (d_i+1) \leq 4g_X -4$ }, \\
 \phi(w) - tl_X &\hbox{if $\sum_{i=1}^m (d_i+1) > 4g_X -4$ }.
\end{cases} \\
L_3(X,\pi,w) & \colonequals  \begin{cases} l_X  &\hbox{if $r=1$ and $\sum_{i=1}^m \left \lceil \frac{\phi(d_i,w')}{2} \right \rceil \geq g_X$}, \\
l_X-a_X^r &\hbox{otherwise}. 
\end{cases} \\
L(X,\pi,w) & \colonequals  L_1(X,\pi,w) + L_2(X,\pi,w) + L_3(X,\pi,w).
\end{align*}
\end{definition}

\begin{lemma} \label{lemma: L(X,pi,w)}
Consider a word $w=V^r \perp V^{n_t} \perp  \ldots \perp V^{n_1}= V^r \perp w'.$ We have
$$\dim_k (w(L)) \geq L(X,\pi,w).$$
\end{lemma}
\begin{proof}
We use the lower bound from Proposition~\ref{prop: w(L) bounds} to the word $\bot w'$, which implies
$$w(L) = V^r\left(w'(L)^\perp\right) \supseteq V^r\left(M_0 \oplus W_{\perp w'}^L \oplus T_0 \right)= V^r(M_0 \oplus W_{\perp w'}^L) + V^r(T_0).$$
where the subspace $W_{\perp w'}^L$ satisfies
$$\codim_k(W_{\perp w'}^L \subseteq W_{\perp w'}) \leq \begin{cases}
    \left \lfloor \frac{3t}{2} \right \rfloor l_X &\hbox{if $\sum_{i=1}^m (d_i+1) \leq 4g_X -4$ } \\
    tl_X &\hbox{if $\sum_{i=1}^m (d_i+1) > 4g_X-4$ },
\end{cases}$$
so that $\dim_k (\pi_W(V^r(W_{\perp w'}^L))) \geq L_2(X,\pi,w)$ by Lemma~\ref{lemma:projection}. Note that $\dim_k(V^r(M_0)) =l_X-a_X^r$. If furthermore $\deg(\lfloor D/2^r\rfloor ) \geq n(X)$ then Tango's theorem (see \cite{Tango} and \cite[Corollary 6.8]{groen}) implies that 
$$M_0 \subseteq V^r(M_0 \oplus W_V) \subseteq V^r(M_0 \oplus W_{\perp w'}^L),$$
so that in either case $\dim_k (V^r(M_0 \oplus W_{\perp w'}^L)) \geq L_1(X,\pi,w) + L_2(X,\pi,w).$

Finally, we analyze the contribution from $T_0$. By Lemma~\ref{lemma:Ldecomposition} we have $\dim_k (\pi_{T_0}(V^r(T_0))) = l_X-a_X^r$, which contributes to lower bound on $\dim_k(w(L)).$ If additionally $\sum_{i=1}^m \left \lceil \frac{\phi(d_i,w')}{2} \right \rceil \geq g_X$, then Corollary~\ref{cor: BC8.4} shows that $V(\eta)=0$ implies $\pi_*(\eta)=0$, so that
$$\ker\left (V : M_0 \oplus T_0 \oplus W_{\perp w'}^L \to M_0 \oplus T_0 \oplus W_{\perp w'}^L \right) \subseteq M_0 \oplus W_{\perp w'}^L .$$ 
In this case the contribution of $T_0$ to the lower bound is $l_X$. Thus in either case the contribution from $T_0$ is at least $L_3(X,\pi,w)$. Therefore we conclude 
\begin{align*}
    \dim_k (w(L)) & \geq \dim_k (V^r(M_0 \oplus W_{\perp w'}^L \oplus T_0)) \\ 
    &\geq L_1(X,\pi,w) + L_2(X,\pi,w) + L_3(X,\pi,w) = L(X,\pi,w). \qedhere
\end{align*}
\end{proof}

We now turn to the upper bound, which is more straightforward to define.
\begin{definition} \label{def: U(X,pi,w)}
Consider a simple word $w=V^r \bot V^{n_t} \bot \ldots \bot V^{n_1}=V^r\bot w'$. Define
$$U(X,\pi,w) \colonequals \begin{cases}  \phi(w) + \left \lceil \frac{3t+4}{2} \right \rceil l_X &\hbox{if $\sum_{i=1}^m (d_i+1) \leq 4g_X -4$ }\\  \phi(w) + (t+2)l_X&\hbox{if $\sum_{i=1}^m (d_i+1) > 4g_X-4$. }
\end{cases}$$
\end{definition}

\begin{lemma} \label{lemma: U(X,pi,w)}
For a simple word $w=V^r \bot V^{n_t} \bot \ldots \bot V^{n_1}=V^r\bot w'$, we have
$$\dim_k (w(L)) \leq U(X,\pi,w).$$
\end{lemma}

\begin{proof}
By Proposition~\ref{prop: w(L) bounds}, we have $w'(L)^\perp \subseteq M \oplus W_{\perp w'}^U \oplus T_0$ with 
\begin{equation*}
\codim_k(W_{\perp w'}\subseteq  W_{\perp w'}^U) \leq \begin{cases} \left \lceil \frac{3t}{2} \right \rceil l_X &\hbox{if $\sum_{i=1}^m (d_i+1) \leq 4g_X -4$ }\\ 
tl_X &\hbox{if $\sum_{i=1}^m (d_i+1) > 4g_X-4$. }
\end{cases}
\end{equation*}
By Proposition~\ref{prop:nilpotent wij}, we have 
\begin{align*}
V^r(M \oplus W_{\perp w'}) &\subseteq M_0 \oplus W_{V^r \perp w'} = M_0 \oplus W_w
\intertext{so that}
\dim_k(V^r(M \oplus W_{\perp w'}^U)) &\leq l_X +  \phi(w) + \codim_k(W_{\perp w'}\subseteq  W_{\perp w'}^U).
\end{align*}
Moreover, one observes $\dim_k(V^r(T_0)) \leq \dim_k(T_0) = l_X$ and deduces
\begin{equation*}
    \dim_k(w(L)) \leq \dim_k(V^r(M \oplus W_{\perp w'}^U \oplus T_0)) \leq \dim_k(V^r(M \oplus W_{\perp w'}^U)) + \dim_k(V^r(T_0)). \qedhere
\end{equation*}
\end{proof}

\begin{proposition} \label{prop: bounds dim(w(L))}
Consider a simple word $w$ in $V$ and $\bot$. Let $L(X,\pi,w)$ be as in Definition~\ref{def: L(X,pi,w)} and let $U(X,\pi,w)$ be as in Definition~\ref{def: U(X,pi,w)}. Then we have
$$L(X,\pi,w) \leq \dim_k( w(L)) \leq U(X,\pi,w).$$
\end{proposition}
\begin{proof}
Combine Lemma~\ref{lemma: L(X,pi,w)} and Lemma~\ref{lemma: U(X,pi,w)}.
\end{proof}

\subsection{Bounds on the final type}

Recall from Section~\ref{ss:final types} that the isomorphism class of $\Jac(Y)[p]$ is encoded in the final type, $\nu = [\nu_1, \ldots , \nu_{g_Y}]$, which may be interpreted as a non-decreasing function from $\{1, \ldots , g_Y\}$ to $\{1, \ldots , g_Y\}$. 
Proposition~\ref{prop: bounds dim(w(L))} can be interpreted as follows: for each word $w$, it provides a rectangle that this function must pass through. This is articulated in the following corollary.  

\begin{corollary} \label{cor: rectangle}
Let $\nu = [\nu_1, \ldots ,\nu_{g_Y}]$ be the final type of $\Jac(Y)[p]$ and consider a word $w$. Then there exists an integer $n$ with $L(X,\pi,w) \leq n \leq U(X,\pi,w)$ such that
$$f_Y + L(X,\pi,Vw) \leq \nu_{f_Y+n} \leq f_Y + U(X,\pi,Vw).$$
\end{corollary}
\begin{proof}
Let $n=\dim_k (w(L))$, so $L(X,\pi,w) \leq n \leq U(X,\pi,w)$ follows immediately from Proposition~\ref{prop: bounds dim(w(L))}. By Lemma~\ref{lemma: reduce to L}, the space $U \oplus w(L)$ occurs in the canonical filtration of $\HdR^1(Y)$, and therefore it occurs in any final filtration. The dimension of $U \oplus w(L)$ is $f_Y+n$. By definition, we have 
$$\nu_{f_Y+n} = \dim_k(V(U \oplus w(L))) = \dim_k(U \oplus V(W(L))) = f_Y + \dim_k(V(w(L))).$$
Then the corollary follows from applying Proposition~\ref{prop: bounds dim(w(L))} to the word $Vw$.
\end{proof}

Corollary~\ref{cor: rectangle} provides bounds on the final type of $\Jac(Y)[p]$, and the bounds on $\nu_{f_Y+n} = f_Y +\dim_k(V(w(L))$ do not fully take the value $n=\dim_k(w(L))$ into account.  We now improve these bounds by better accounting for the relationship between $w(L)$ and $V(w(L))$.

For ease of exposition, we fix a simple word $\widehat{w}=\bot V^{n_t} \bot \ldots \bot V^{n_1}$ and let $W_{\widehat{w}}^L$ and $W_{\widehat{w}}^U$ be as in Proposition~\ref{prop: w(L) bounds}. We introduce the following notation for any $r\geq 0$:
\begin{align*}
    \delta^r &:= \dim_k (V^r(M_0 \oplus T_0 \oplus W_{\widehat{w}})) \\
    \delta_L^r &:= \dim_k (V^r(M_0 \oplus T_0 \oplus W_{\widehat{w}}^L)) \\
    \delta_U^r &:= \dim_k (V^r(M \oplus T_0 \oplus W_{\widehat{w}}^U)) \\
    \varepsilon^r &:= \dim_k(V^r(\widehat{w}(L))) \\
    c_1^r &:= \codim_k (V^r(M_0 \oplus T_0 \oplus W_{\widehat{w}}^L)\subseteq V^r(\widehat{w}(L))) = \varepsilon^r-\delta_L^r \\
    c_2^r &:= \codim_k (V^r(\widehat{w}(L))\subseteq V^r(M \oplus T_0 \oplus W_{\widehat{w}}^U)) = \delta_U^r-\varepsilon^r \\
    e_1^r &:= \codim_k (V^r(M_0 \oplus T_0 \oplus W_{\widehat{w}}^L)\subseteq V^r(M_0\oplus T_0 \oplus W_{\widehat{w}})) = \delta^r-\delta_L^r \\
    e_2^r &:= \codim_k (V^r(M_0 \oplus T_0 \oplus W_{\widehat{w}})\subseteq V^r(M_0 \oplus T_0 \oplus W_{\widehat{w}}^U)) = \delta_U^r-\delta^r.
\end{align*}

The situation is illustrated in the following diagram where the labels indicate codimension:
\begin{equation*} \label{diag: c_1^r etc}
\begin{tikzcd}
                                               & V^r(M\oplus T_0 \oplus W_{\widehat{w}}^U)                                                           &                                                                        \\
                                               &                                                                                                     &                                                                        \\
V^r(\widehat{w}(L)) \arrow[ruu, "c_2^r", hook] &                                                                                                     & V^r(M_0 \oplus T_0 \oplus W_{\widehat{w}}) \arrow[luu, "e_2^r"', hook] \\
                                               &                                                                                                     &                                                                        \\
                                               & V^r(M_0 \oplus T_0 \oplus W_{\widehat{w}}^L). \arrow[luu, "c_1^r", hook] \arrow[ruu, "e_1^r"', hook] &                                                                       
\end{tikzcd}
\end{equation*}

Now Proposition~\ref{prop: w(L) bounds} and the fact that $M_0$ has codimension $l_X$ in $M$ gives
\begin{align}
    e_1^0 &\leq \begin{cases}  \left \lfloor \frac{3t}{2} \right\rfloor l_X &\hbox{if $\sum_{i=1}^m (d_i+1) \leq 4g_X -4$ }  \\ t l_X &\hbox{if $\sum_{i=1}^m (d_i+1) > 4g_X -4$ } \end{cases} \label{eq: e_1^0} \\
    e_2^0 &\leq \begin{cases}  \left \lceil \frac{3t+2}{2} \right\rceil l_X &\hbox{if $\sum_{i=1}^m (d_i+1) \leq 4g_X -4$ }  \\ (t+1) l_X &\hbox{if $\sum_{i=1}^m (d_i+1) > 4g_X -4$.} \end{cases} \label{eq: e_2^0}
\end{align}

Our goal is now to bound $\varepsilon^{r+1}$ based on $w=V^r \widehat{w}$ and $n=\varepsilon^r$, so that $\nu_{f_Y+n} = f_Y + \varepsilon^{r+1}$. The following Lemma records that $\varepsilon^{r+1}$ cannot be further away from the bounds of Proposition~\ref{prop: bounds dim(w(L))} than $\varepsilon^r$ is. 

\begin{lemma} \label{lemma: epsilon delta}
We have
\begin{equation*} \label{eq: epsilon delta}
    \varepsilon^r + (\delta_U^{r+1} - \delta_U^r) \leq \varepsilon^{r+1} \leq \varepsilon^r + (\delta_L^{r+1} - \delta_L^r).
\end{equation*}
\end{lemma}
\begin{proof}
We focus first on the lower bound. By definition, we have $\varepsilon^r= \delta_U^r-c_2^r$, so $c_2^r=\varepsilon^r-\delta_U^r$. Next, we observe $c_2^{r+1} \leq c_2^r$, since there is a surjection
$$V^r(M\oplus T_0 \oplus W_{\widehat{w}}^U) / V^{r}(\widehat{w}(L)) \to V^{r+1}(M\oplus T_0 \oplus W_{\widehat{w}}^U) / V^{r+1}(\widehat{w}(L)).$$
This gives
$$\varepsilon^{r+1} = \delta_U^{r+1} - c_2^{r+1} \geq \delta_U^{r+1} - c_2^r = \delta_U^{r+1} - (\delta_U^r - \varepsilon^r) = \varepsilon^r + (\delta_U^{r+1}-\delta_U^r).$$
The upper bound is established in an analogous fashion using $\widehat{w}(L)$ and $M_0 \oplus T_0 \oplus W_{\widehat{w}}^L$.
\end{proof}

The next step is to bound the differences $\delta_U^{r+1}-\delta_U^r$ and $\delta_L^{r+1}-\delta_L^r$. Note that, by definition of $e_1^r$ and $e_2^r$, we have
\begin{align*}
    \delta_L^{r+1} -\delta_L^r &= (\delta^{r+1} -e_1^{r+1}) - (\delta^r-e_1^r) = (\delta^{r+1}-\delta^r) + (e_1^r - e_1^{r+1}) \\
    \delta_U^{r+1}-\delta_U^r &= (\delta^{r+1} + e_2^{r+1}) - (\delta^r+e_2^r) = (\delta^{r+1} - \delta^r) - (e_2^r-e_2^{r+1}).
\end{align*}

\begin{lemma} \label{lemma: delta^r+1-delta^r}
We have
$$ \phi(V^{r+1}\widehat{w}) - \phi(V^r \widehat{w})- 2a_X^{r+1}\leq \delta^{r+1}-\delta^r \leq  \phi(V^{r+1}\widehat{w}) - \phi(V^r\widehat{w})+a_X^r.$$
\end{lemma}
\begin{proof}
We first prove the inequality
\begin{equation} \label{eq: delta^r}
\phi(V^r \widehat{w}) + 2(l_X-a_X^r) \leq \delta^r \leq \phi(V^r \widehat{w}) + 2l_X.
\end{equation}
The proof of the lower bound in Equation~\eqref{eq: delta^r} is analogous to the proof of Lemma~\ref{lemma: L(X,pi,w)}. Namely, we observe
\begin{align*} 
\delta^r = \dim_k (V^r(M_0 \oplus T_0 \oplus W_{\widehat{w}})) &\geq \dim_k (\pi_{M_0}(V^r(M_0))) + \dim_k (\pi_{T_0}(V^r(T_0))) + \dim_k (\pi_W(V^r(W_{\widehat{w}}))) \\ &= 2(l_X-a_X^r) + \phi(V^r\widehat{w}).
\end{align*}
For the upper bound of Equation~\eqref{eq: delta^r}, we observe that $V^r(M_0 \oplus W_{\widehat{w}}) \subseteq M_0 \oplus W_{V^r\widehat{w}}$. Therefore,
\begin{align*}
    \delta^r = \dim_k (V^r(M_0 \oplus T_0 \oplus W_{\widehat{w}})) &\leq \dim_k(M_0 \oplus W_{V^r\widehat{w}}) + \dim_k (T_0) \\ &=2l_X+\phi(V^r \widehat{w}). 
\end{align*}

The lower bound in the lemma follows immediately from subtracting Equation~\eqref{eq: delta^r} from its analog for $\delta^{r+1}$. For the upper bound, observe that
\begin{align*}
\delta^{r+1} &= \dim_k (V^{r+1} (M_0 \oplus W_{\widehat{w}} \oplus T_0)) \\
&\leq \dim_k (V^{r+1}(M_0 \oplus W_{\widehat{w}})) + l_X \\
&= \phi(V^{r+1}\widehat{w}) + \dim_k(\pi_{M_0}( V^{r+1}(M_0 \oplus W_{\widehat{w}}))) + l_X
\end{align*}
and that 
\begin{equation*}
\delta^r  \geq \phi(V^r\widehat{w}) + \dim_k(\pi_{M_0}( V^{r}(M_0 \oplus W_{\widehat{w}}))) + (l_X-a_X^r).
\end{equation*}
Then the lemma follows as $\dim_k(\pi_{M_0}( V^{r+1}(M_0 \oplus W_{\widehat{w}}))) \leq \dim_k(\pi_{M_0}(V^r(M_0 \oplus W_{\widehat{w}})))$. 
\end{proof}

\begin{proposition} \label{prop: n and w bounds}
Suppose $w=V^r \bot V^{n_t} \bot \ldots \bot V^{n_1}$ and $n=\dim_k(w(L))$. Then we have
$$n+\phi(Vw) - \phi(w) -2a_X^{r+1} - e_2^0 \leq \dim_k(V(w(L))) \leq n+ \phi(Vw)-\phi(w) +a_X^r +e_1^0.$$
\end{proposition}
\begin{proof}
This follows from Lemmas~\ref{lemma: epsilon delta} and ~\ref{lemma: delta^r+1-delta^r} and the observation $0 \leq e_j^r-e_j^{r+1} \leq e_j^0$.
\end{proof}

Proposition~\ref{prop: n and w bounds} provides bounds on $\dim_k(V(w(L))$ in terms of $n$ and $w$, assuming $n=\dim_k(w(L))$. This complements Proposition~\ref{prop: bounds dim(w(L))}, which gives bounds on $\dim_k(w(L))$ for any word.

\begin{theorem} \label{theorem: final type bounds}
Let $Y \to X$ be a double cover with ramification invariants $d_1, \ldots ,d_m$. Let $w=V^{n_s} \bot \ldots \bot V^{n_1}$. Let $[\nu_1, \ldots , \nu_{g_Y}]$ be the final type of $Y$ and set $l:=f_Y+\phi(w)+2l_X$. Then we have
\[ \left| \nu_l - (f_Y + \phi(Vw) + l_X)  \right|  \leq
\begin{cases}
    \left \lceil \frac{3s+2}{2} \right \rceil l_X &\hbox{if $\sum_{i=1}^m (d_i+1) \leq 4g_X -4$ }  \\ (s+1)l_X &\hbox{if $\sum_{i=1}^m (d_i+1) > 4g_X -4$.}
\end{cases}
\]
\end{theorem} 

\begin{proof}
For ease of exposition, we assume $\sum_{i=1}^m (d_i+1) > 4g_X -4$, so that Lemma~\ref{lemma: M orthogonal to W_V} applies and the bounds in Proposition~\ref{prop: bounds dim(w(L))} are sharper. The case $\sum_{i=1}^m (d_i+1) \leq 4g_X -4$ works analogously. Throughout this proof, we use the simplification $a_X^r \leq l_X$ for every $r$.

Let $n=\dim_k(w(L))$ and let $t=s-1$. Then Proposition~\ref{prop: bounds dim(w(L))} implies
$$\phi(w) - tl_X \leq L(X,\pi,w) \leq n =\dim_k(w(L)) \leq U(X,\pi,w) \leq \phi(w)+(t+2)l_X.$$
Recalling Lemma~\ref{lemma: reduce to L}, note that $\nu_{f_Y+n}=f_Y + \dim_k(V(w(L)))$. We consider two cases.
\begin{itemize}
    \item Assume $f_Y+n\leq l$. Then Proposition~\ref{prop: bounds dim(w(L))} yields
    $$\nu_{l} \geq \nu_{f_Y+n} \geq f_Y + L(X,\pi,Vw) \geq f_Y+\phi(Vw)-tl_X = (f_Y+\phi(Vw)+l_X)-sl_X,$$
    which proves the lower bound of the theorem. For the upper bound, we use Proposition~\ref{prop: n and w bounds} and Equation~\eqref{eq: e_1^0}:
    \begin{align*}
        \nu_{f_Y+n} &\leq f_Y+ n+\phi(Vw)-\phi(w)+l_X + tl_X. 
        \intertext{Now using that $\nu_j\leq \nu_{j+1}\leq \nu_j+1$ we see }
        \nu_{l} &\leq l +\phi(Vw) - \phi(w) + (t+1) l_X .
        \intertext{Finally recalling that $l=f_Y+\phi(w)+2l_X$ yields}
        \nu_{l} &\leq f_Y+ \phi(Vw) + (t+3)l_X  = (f_Y+\phi(Vw)+l_X)+(s+1)l_X.
    \end{align*}
    Thus the assertion is proven when $f_Y+n\leq l$.

    \item Assume $f_Y+n>l$. Then Proposition~\ref{prop: bounds dim(w(L))} immediately yields the upper bound
    $$\nu_{l} \leq \nu_{f_Y+n} \leq f_Y + U(X,\pi,Vw) \leq f_Y+\phi(Vw)+(t+2)l_X = (f_Y+\phi(Vw)+l_X)+sl_X.$$
    For the lower bound, Proposition~\ref{prop: n and w bounds} and Equation~\eqref{eq: e_1^0} show
    \begin{align*}
    \nu_{f_Y+n} &\geq f_Y + n + \phi(Vw) - \phi(w)-2l_X - (t+1)l_X.
    \intertext{Then we conclude that }
    \nu_{l} &\geq l +\phi(Vw)-\phi(w) - 2l_X - (t+1)l_X\\
    & =f_Y + \phi(Vw) - (t+1)l_X = (f_Y+\phi(Vw)+l_X) -(s+1)l_X,
    \end{align*}
    which finishes the proof. \qedhere
\end{itemize}
\end{proof}

\begin{remark}
Note that Theorem~\ref{thm: ordinary EO type} is a special case of Theorem~\ref{theorem: final type bounds} in which $X$ is ordinary. More precisely, one recovers Corollary~\ref{cor: nu ordinary}, which is equivalent to Theorem~\ref{thm: ordinary EO type}, by substituting $l_X=0$ in Theorem~\ref{theorem: final type bounds}.
\end{remark}

\subsection{One point covers}

We now apply Theorem~\ref{theorem: final type bounds} to double covers that are branched at exactly one point. This leads to Theorem~\ref{theorem: log bounds}. Let $Y \to X$ be a double cover branched at one point with ramification break $d$.

One of the inputs of Theorem~\ref{theorem: final type bounds} is the word $w$ that is needed to construct a space in the canonical filtration of a given size. The bounds are proportional to the number of occurrences of $\perp$ in $w$. The following lemma will allow us to bound this number uniformly.

\begin{lemma} \label{lemma: log}
Let $d$ be a positive integer. Every integer $0 \leq m \leq \frac{d-1}{2}$ can be written as $m=\phi(d,w)$ for some word $w=V^{n_s} \bot \ldots \bot V^{n_1}$ with $s \leq  \left \lceil \log_2 \left( \frac{d-1}{2} \right) \right \rceil$.
\end{lemma}

\begin{proof}
Let $M_t$ be the maximal gap between consecutive values of $\phi(d,w) \leq \frac{d-1}{2}$ for simple words $w$ with at most $t$ instances of $\bot$.  (By consecutive we mean that there does not exist a word $w_3$ of the same form such that $\phi(d,w_1) < \phi(d,w_3) < \phi(d,w_2)$.)  We will prove that $M_t$ decreases at the appropriate rate as $t$ increases.

Let $i=\phi(d,w_1)$ and $j=\phi(d,w_2)$. Then recalling Definition~\ref{def: phi} we see
$$\left \lfloor \frac{|i-j|}{2^n} \right \rfloor  \leq |\phi(d,V^n\bot w_1) - \phi(d,V^n\bot w_2) | \leq \left \lceil \frac{|i-j|}{2^n} \right \rceil. $$
This implies $M_{t+1} \leq \left \lceil  \frac{M_t}{2} \right\rceil$. Note that $M_0=\frac{d-1}{2}$, so that $M_s=1$ when $s\geq \left \lceil \log_2 \left( \frac{d-1}{2} \right) \right \rceil$.
\end{proof}

This lemma can be interpreted as follows: all spaces in the (unique) final filtration of $W$ can be constructed using a word containing only a small number of applications of $\perp$. This provides restrictive bounds on the final type of $Y$, as is recorded by the following theorem.

\begin{theorem} \label{theorem: log bounds}
Let $Y\to X$ be branched at one point with ramification invariant $d$. Let $f_Y=2f_X$ be the $p$-rank of $Y$. Let $l_X=g_X-f_X$ be the local rank of $X$. Denote the final type of $Y$ by $[\nu_1, \ldots, \nu_{g_Y} ]$. Then the following restrictions hold:

\begin{enumerate}
    \item For $1\leq l \leq f_Y$, we have $\nu_l=l$.
    \item For $f_Y< l < f_Y+2l_X,$ we have $f_Y \leq \nu_l \leq l-1$.
    \item For $f_Y+2l_X \leq l \leq g_Y$, we have 
    \begin{equation} \label{eq: log bounds}
        \left | \nu_l - \left(f_Y + \left \lfloor \frac{l-f_Y}{2} \right \rfloor \right)   \right| \leq \begin{cases}   \frac{3\lceil \log_2(d-1) \rceil}{2}  l_X &\hbox{if $d \leq 4g_X-5$}  \\ \ceiling{\log_2(d-1)} l_X &\hbox{if $d > 4g_X-5$.} \end{cases} 
    \end{equation}
\end{enumerate}
\end{theorem}
\begin{proof}
For $1 \leq l \leq f_Y$, the statement follows immediately from the fact $f_Y=\max\{j \; | \; \nu_j=j\}$.

This fact also implies $\nu_{f_Y+1}=f_Y$. Combining this with the rule $\nu_j\leq \nu_{j+1} \leq \nu_j +1$ gives the statement for $f_Y<l<f_Y+2l_X$.

Finally, assume $f_Y+2l_X \leq l \leq g_Y = f_Y+2l_X + \frac{d-1}{2}$. Let $m:=l-f_Y-2l_X$, so that $0 \leq m \leq \frac{d-1}{2}$. Then Lemma~\ref{lemma: log} implies that there exists a word $w=V^{n_s} \perp \ldots \perp V^{n_1}$ with $s\leq \left \lceil \log_2\left(\frac{d-1}{2}\right) \right \rceil$ such that $m=\phi(d,w)$ and hence $l=f_Y+\phi(d,w)+2l_X$. We apply Theorem~\ref{theorem: final type bounds} to this word $w$, yielding
$$\left| \nu_l - (f_Y + \phi(Vw) + l_X)  \right|  \leq
\begin{cases}
    \frac{3\lceil \log_2(d-1) \rceil}{2} l_X &\hbox{if $d \leq 4g_X -5$ }  \\ \ceiling{\log_2(d-1)}l_X &\hbox{if $d > 4g_X -5$.}
\end{cases}$$
Here we have used $\log_2\left(\frac{d-1}{2}\right)+1 = \log_2(d-1)$. The proof is concluded by observing
\[
\phi(d,Vw)=\left \lfloor \frac{\phi(d,w)}{2} \right \rfloor = \left \lfloor \frac{l-f_Y-2l_X}{2} \right \rfloor = \left \lfloor \frac{l-f_Y}{2} \right \rfloor- l_X. \qedhere
\]
\end{proof}

\begin{remark}
The aim of Theorem~\ref{theorem: log bounds} is to provide a uniform bound on the possible amount of variation of the final type of $Y$. The bounds will not be optimal for all $l$. For instance, when $l=f_Y+2l_X+\phi(d,w)$ for a word $w=V^{n_s} \perp \ldots \perp V^{n_1}$ with $s$ much smaller than $\log_2(d-1)$, then Theorem~\ref{theorem: final type bounds} prescribes stricter bounds.
\end{remark}

\begin{remark}
It is interesting to interpret Theorem~\ref{theorem: log bounds} in the context when the base curve $X$ is fixed and the ramification invariant $d$ increases. The genus $g_Y=2g_X+\frac{d-1}{2}$ grows linearly with $d$, while the interval allowed by the bounds grows logarithmically. Since $$\displaystyle\lim_{d \to \infty} \frac{\log_2(d-1)l_X}{2g_X+\frac{d-1}{2}}=0,$$ it follows that, as $d$ grows, the error in approximating $\nu_l$ by $f_Y + \left \lfloor \frac{l-f_Y}{2} \right \rfloor$ becomes negligible.
\end{remark}

\bibliographystyle{alpha}
\bibliography{Bibliography}

\begin{thebibliography}{KMU25}

\bibitem[BC20]{boohercais}
Jeremy Booher and Bryden Cais.
\newblock {$a$}-numbers of curves in {A}rtin-{S}chreier covers.
\newblock {\em Algebra \& Number Theory}, 14(3):587--641, 2020.

\bibitem[BC22]{BoCaIwasawa}
Jeremy Booher and Bryden Cais.
\newblock Iwasawa theory for p-torsion class group schemes in characteristic p.
\newblock {\em Nagoya Mathematical Journal}, page 1–54, 2022.

\bibitem[Che63]{chevalley}
Claude Chevalley.
\newblock {\em Introduction to the theory of algebraic functions of one variable}, volume No. VI of {\em Mathematical Surveys}.
\newblock American Mathematical Society, Providence, RI, 1963.

\bibitem[Col98]{coleman98}
Robert~F. Coleman.
\newblock Duality for the de {R}ham cohomology of an abelian scheme.
\newblock {\em Ann. Inst. Fourier (Grenoble)}, 48(5):1379--1393, 1998.

\bibitem[Con00]{conrad_00}
Brian Conrad.
\newblock {\em Grothendieck duality and base change}.
\newblock Lecture notes in mathematics, 1750. Springer, Berlin ;, 2000.

\bibitem[CU23]{caisulmer}
Bryden {Cais} and Douglas {Ulmer}.
\newblock {$p$-torsion for unramified Artin--Schreier covers of curves}.
\newblock {\em arXiv e-prints}, page arXiv:2307.16346, July 2023.

\bibitem[Dem72]{DemGroups}
Michel Demazure.
\newblock {\em Lectures on $p$-divisible groups}.
\newblock Springer-Verlag Berlin Heidelberg, 1972.

\bibitem[EP13]{ElkinPries}
Arsen Elkin and Rachel Pries.
\newblock {Ekedahl–{O}ort strata of hyperelliptic curves in characteristic 2}.
\newblock {\em Algebra \& Number Theory}, 7(3):507 -- 532, 2013.

\bibitem[FP13]{PriesConstant}
Shawn Farnell and Rachel Pries.
\newblock Families of {A}rtin-{S}chreier curves with {C}artier-{M}anin matrix of constant rank.
\newblock {\em Linear Algebra Appl.}, 493(7):2158--216, 2013.

\bibitem[Gar23]{garnek}
J\polhk{e}drzej Garnek.
\newblock {$p$}-group {G}alois covers of curves in characteristic {$p$}.
\newblock {\em Trans. Amer. Math. Soc.}, 376(8):5857--5897, 2023.

\bibitem[Gar25]{garnek2}
J\polhk{e}drzej Garnek.
\newblock {$p$}-group {G}alois covers of curves in characteristic {$p$}. {II}.
\newblock {\em Doc. Math.}, 30(2):347--377, 2025.

\bibitem[Gro23]{StevenPhD}
Steven~R. Groen.
\newblock {\em p-Torsion of {A}belian varieties in characteristic p}.
\newblock PhD thesis, University of Warwick, \url{http://webcat.warwick.ac.uk/record=b3970553}, 2023.

\bibitem[Gro24]{groen}
Steven~R. Groen.
\newblock Powers of the cartier operator on artin-schreier covers.
\newblock {\em International Journal of Number Theory}, 20(07):1901 -- 1941, 2024.

\bibitem[Gur01]{gurski}
Nick Gurski.
\newblock Differentials of the second kind in characteristic $p$.
\newblock \url{https://web.ma.utexas.edu/users/voloch/Preprints/gurski.pdf}, April 2001.

\bibitem[Har77]{Hartshorne}
Robin Hartshorne.
\newblock {\em Algebraic Geometry}.
\newblock Springer-Verlag New York, 1977.

\bibitem[Kat70]{Katz-Nilpotent_connections_and_monodromy_theory}
Nicholas~M. Katz.
\newblock Nilpotent connections and the monodromy theorem: {A}pplications of a result of {T}urrittin.
\newblock {\em Inst. Hautes \'Etudes Sci. Publ. Math.}, 39:175--232, 1970.

\bibitem[KM21]{Kramer-Miller}
Joe Kramer-Miller.
\newblock {$p$}-adic estimates of exponential sums on curves.
\newblock {\em Algebra Number Theory}, 15(1):141--171, 2021.

\bibitem[KMU25]{Kramer-Miller-Upton}
Joe Kramer-Miller and James Upton.
\newblock Newton polygons of sums on curves {I}: local-to-global theorems.
\newblock {\em Math. Ann.}, 391(1):1347--1394, 2025.

\bibitem[Moo01]{Moonengsas}
Ben Moonen.
\newblock {\em {Group Schemes with Additional Structures and Weyl Group Cosets}}, pages 255--298.
\newblock Birkh{\"a}user Basel, Basel, 2001.

\bibitem[Oda69]{Oda}
Tadao Oda.
\newblock The first {D}e {R}ham cohomology group and {D}ieudonn\'e modules.
\newblock {\em Annales scientifiques de l'\'Ecole Normale Sup\'erieure}, Ser. 4, 2(1):63--135, 1969.

\bibitem[Oor01]{OortStrat}
Frans Oort.
\newblock {\em A Stratification of a Moduli Space of Abelian Varieties}, pages 345--416.
\newblock Birkh{\"a}user Basel, Basel, 2001.

\bibitem[Pri08]{pries_guide}
Rachel Pries.
\newblock A short guide to {$p$}-torsion of abelian varieties in characteristic {$p$}.
\newblock In {\em Computational arithmetic geometry}, volume 463 of {\em Contemp. Math.}, pages 121--129. Amer. Math. Soc., Providence, RI, 2008.

\bibitem[PZ12]{prieszhu12}
Rachel Pries and Hui~June Zhu.
\newblock The {$p$}-rank stratification of {A}rtin-{S}chreier curves.
\newblock {\em Ann. Inst. Fourier (Grenoble)}, 62(2):707--726, 2012.

\bibitem[Ros53]{rosenlicht}
Maxwell Rosenlicht.
\newblock Differentials of the second kind for algebraic function fields of one variable.
\newblock {\em Ann. of Math. (2)}, 57:517--523, 1953.

\bibitem[Ser88]{Serre1988}
Jean-Pierre Serre.
\newblock {\em Algebraic groups and class fields}, volume 117 of {\em Graduate Texts in Mathematics}.
\newblock Springer-Verlag, New York, 1988.
\newblock Translated from the French.

\bibitem[Sub75]{SubraoDS}
Doré Subrao.
\newblock The p-rank of {A}rtin-{S}chreier curves.
\newblock {\em Manuscripta Math.}, 16:169--193, 1975.

\bibitem[Tan72]{Tango}
Hiroshi Tango.
\newblock {On the behavior of extensions of vector bundles under the Frobenius map}.
\newblock {\em Nagoya Mathematical Journal}, 48(none):73 -- 89, 1972.

\bibitem[Vol88]{voloch88}
Jos\'e{}~Felipe Voloch.
\newblock A note on algebraic curves in characteristic {$2$}.
\newblock {\em Comm. Algebra}, 16(4):869--875, 1988.

\bibitem[Wei25]{weir_code}
Colin Weir.
\newblock {EO}type {G}ithub {R}epository.
\newblock \url{https://github.com/cjweir/EOType}, 2025.

\end{thebibliography}

\end{document}